\documentclass[oneside,english]{amsart}
\usepackage{amsthm}
\usepackage{amstext}
\usepackage{amssymb}
\usepackage{esint}
\usepackage{graphicx}
\makeatletter
\numberwithin{equation}{section}
\numberwithin{figure}{section}
\theoremstyle{plain}
\newtheorem{thm}{\protect\theoremname}
  \theoremstyle{plain}
  \newtheorem{lem}[thm]{\protect\lemmaname}
  \theoremstyle{remark}
  \newtheorem{rem}[thm]{\protect\remarkname}
  \theoremstyle{definition}
  \newtheorem{example}[thm]{\protect\examplename}
  \theoremstyle{plain}
  \newtheorem{cor}[thm]{\protect\corollaryname}
  \theoremstyle{plain}
  \newtheorem{prop}[thm]{\protect\propositionname}
  
\numberwithin{thm}{section}

\makeatother

\usepackage{babel}
  \providecommand{\corollaryname}{Corollary}
  \providecommand{\examplename}{Example}
  \providecommand{\lemmaname}{Lemma}
  \providecommand{\propositionname}{Proposition}
  \providecommand{\remarkname}{Remark}
\providecommand{\theoremname}{Theorem}

\begin{document}

\title{Outliers in the spectrum of large deformed unitarily invariant models}

\author{S. T. Belinschi, H. Bercovici, M. Capitaine, and M. F\'evrier}

\address{S. T. Belinschi: CNRS - Institut de Math\'ematiques de Toulouse and Queen's University 
and Institute of Mathematics ``Simion Stoilow'' of the Romanian Academy;
IMT\\
118 Rte de Narbonne\\
F-31062 Toulouse Cedex 09, France.}

\email{serban.belinschi@math.univ-toulouse.fr}

\address{H. Bercovici: Indiana University\\
Department of Mathematics\\
Indiana University\\
Bloomington, IN 47405, USA.}

\email{bercovic@indiana.edu}

\address{M. Capitaine:
CNRS, Institut de Math\'ematiques de Toulouse\\
118 Route de Narbonne\\
F-31062 Toulouse Cedex 09, France.}

\email{mireille.capitaine@math.univ-toulouse.fr}

\address{M. F\'evrier: Universit\'e Paris Sud\\
Universit\'e Paris Sud, Laboratoire de Math\'ematiques \\
B\^{a}t. 425 91405 Orsay Cedex, France.}

\email{maxime.fevrier@math.u-psud.fr}

\thanks{Research of HB was supported in part by NSF grants. For part of the work, STB was supported 
by a Discovery Grant from the Natural Sciences and Engineering Research Council of Canada and
a RIG from Queen's University. The work on this paper was initiated 
while STB was visiting the Institute of Mathematics of Toulouse as ``Ma\^itre de conf\'erence 
invit\'e'' and largely completed while all four authors visited the Fields Institute during the 
2013 special programme on non-commutative distributions. We gratefully acknowledge the 
hospitality of the institutions involved. We also thank Charles Bordenave for providing us the simulation
from Example \ref{ex:5.14}}

\begin{abstract}
In this paper we characterize the possible outliers in the spectrum of large deformed unitarily
invariant additive and multiplicative models, as well as the eigenvectors corresponding to them.
We allow  both the non-deformed unitarily invariant model and the perturbation matrix  to have  non-trivial  limiting spectral measures and  spiked outliers in their spectrum.
We uncover a remarkable new phenomenon: a 
single spike  can generate asymptotically several
outliers in the spectrum of the deformed model. The free subordination functions play a key role in this analysis.
\end{abstract}
\maketitle

\section{Introduction}

\subsection{Statement of the problem}\markboth{}{}The set of possible spectra for the sum of two deterministic Hermitian 
matrices $A_{N}$ and $B_{N}$ depends in complicated ways on the
spectra of $A_{N}$ and $B_{N}$ (see \cite{Fulton98}). Nevertheless,
if one adds some randomness to the eigenspaces of $B_{N}$ then, as
$N$ becomes large, free probability provides a good understanding
of the behavior of the spectrum of this sum.
 More precisely, set
$X_{N}=A_{N}+U_{N}^{*}B_{N}U_{N}$,
where $U_{N}$ is a random unitary  matrix distributed
according to  the Haar measure on the unitary group ${\rm U}(N)$, and suppose that the empirical 
eigenvalue distributions of $A_{N}$ and $B_{N}$ converge weakly
to compactly supported distributions $\mu$ and $\nu$, respectively.
Building on the groundbreaking result of Voiculescu \cite{Voiculescu91},
Speicher proved in \cite{Speicher93a} the almost sure weak convergence
of the empirical eigenvalue distribution of $X_{N}$ to the free additive convolution
$\mu\boxplus\nu$. This convolution is again a compactly supported
probability measure on $\mathbb{R}$. A similar result holds for products of matrices: if 
$A_N,B_N$ are in addition assumed to be nonnegative definite, then 
the empirical eigenvalue distribution of $A_N^{1/2}U_{N}^{*}B_{N}U_{N}A_N^{1/2}$ 
converges to the free multiplicative convolution $\mu\boxtimes\nu$, a compactly
supported probability measure on $[0,+\infty)$. (We recall that
$A_N^{1/2}U_{N}^{*}B_{N}U_{N}A_N^{1/2}$ and $B_N^{1/2}U_{N}^{*}A_{N}U_{N}B_N^{1/2}$
have the same eigenvalue distribution, and that $\boxtimes$ is a commutative operation.) Finally, if
$A_N$ and $B_N$ are deterministic unitary matrices, their empirical eigenvalue distributions are 
supported
on the unit circle $\mathbb T=\{z\in\mathbb C\colon |z|=1\}$ and 
the empirical  eigenvalue distribution of $A_NU_{N}^{*}B_{N}U_{N}$ 
converges to the free multiplicative convolution $\mu\boxtimes\nu$, a 
probability measure supported on $\mathbb T$.
(We refer the reader to \cite{VDN92}
for an introduction to free probability theory. We describe later
the mechanics of calculating the free convolutions $\boxplus$ and $\boxtimes$.)

The fact that the empirical eigenvalue distribution of $X_{N}$  converges weakly to
$\mu\boxplus\nu$ does not mean that all the eigenvalues of $X_{N}$
are close to the support of this measure. There can be outliers, though
they must not affect the limiting empirical eigenvalue distribution. Sometimes one can argue
that these outliers must in fact exist. For instance, the case when the rank $r$ of $A_N$ 
and its nonzero eigenvalues are fixed is investigated by Benaych-Georges and Nadakuditi in 
\cite{BGRao09}. Denote by 
\[
\gamma_{1}\geq\cdots\geq\gamma_{s}>0>\gamma_{s+1}\geq\cdots\geq\gamma_{r}
\]
these fixed eigenvalues. Of course, in this case $\mu=\delta_{0}$
is a point mass at $0$, so the {\em limiting} behavior of the empirical eigenvalue
distribution of $X_{N}$ is not affected by such a matrix $A_{N}$.
More precisely, the empirical  eigenvalue distribution of $X_{N}=A_{N}+U_{N}^{*}B_{N}U_{N}$
converges almost surely  to the limiting spectral distribution $\nu$ of $B_{N}$.
One can however detect, among the outlying eigenvalues of $X_{N}$,
the influence of the eigenvalues of $A_{N}$ above a certain critical
threshold. We use the notation
\[
\lambda_{1}(X)\geq\cdots\geq\lambda_{N}(X)
\]
 for the eigenvalues of an $N\times N$ matrix $X$,
 repeated according to multiplicity. The Cauchy-Stieltjes transform
of a finite positive Borel measure $\nu$ on $\mathbb{R}$ is given by
$$
G_{\nu}(z)=\int_{\mathbb{R}}\frac{d\nu(t)}{z-t}
$$
for $z$ outside the support of $\nu$, and $G_{\nu}^{-1}$ denotes
the inverse of this function relative to composition. When the support
of $\nu$ is contained in the compact interval $[a,b]$, the branch of the inverse function
$G_{\nu}^{-1}$ that satisfies $G_\nu^{-1}(0)=\infty$  is defined and real-valued on the interval $(\alpha,\beta)$, where
\[
\alpha=\lim_{x\uparrow a}G_{\nu}(x),\quad\beta=\lim_{x\downarrow b}G_{\nu}(x).
\]
The following result is proved in \cite[Theorems 2.1 and 2.2]{BGRao09}.
\begin{thm}\label{thm:RRao}
   \ 
\begin{enumerate}
\item
Denote by $a$ and $b$ the infimum and supremum
of the support of $\nu$, respectively. Assume that the smallest and largest eigenvalues
of $B_{N}$ converge almost surely to $a$ and $b$, respectively.
Then, almost surely for $1\le i\le s$, 
\[
\lim_{N\to\infty}\lambda_{i}(X_{N})=\begin{cases}
G_{\nu}^{-1}(1/\gamma_{i}), & \gamma_{i}>1/\beta,\\
b, & \text{ otherwise.}
\end{cases}
\]
Similarly, almost surely for $0\le j\le r- s-1$, 
\[
\lim_{N\to\infty}\lambda_{N-j}(X_{N})=\begin{cases}
G_{\nu}^{-1}(1/\gamma_{r-j}), & \gamma_{r-j}<1/\alpha,\\
a, & \text{ otherwise.}
\end{cases}
\]
\item Fix $i_0\in\{1,\dots,r\}$ such that $1/\gamma_{i_0}\in(\alpha,\beta)$. For each $N$
define 
\[
\lambda(N)=\begin{cases}
\lambda_{i_0}({X_N}), & \gamma_{i_0}>0,\\
\lambda_{N-r+i_0}({X_N}), & \gamma_{i_0}<0.
\end{cases}
\]
and let $u_N$ be a unit-norm eigenvector of ${X_N}$
associated to the eigenvalue $\lambda(N)$. Then the following almost sure 
limits hold{\em:}
$$
\lim_{N\to\infty}\|P_{\ker(\gamma_{i_0}I_N-A_N)}u_N\|^2=
\frac{-1}{\gamma_{i_0}^2G_{\nu}'\left(G_{\nu}^{-1}(1/\gamma_{i_0})\right)},
$$
and
$$
\lim_{N\to\infty}\|P_{\ker(\gamma_{i}I_N-A_N)}u_N\|^2=0,\quad i\ne i_0.$$
\end{enumerate}
\end{thm}

This result lies in the lineage of recent, and not so recent, works studying the influence
of fixed rank additive or multiplicative perturbations on the extremal
eigenvalues of classical random matrix models, the seminal paper being
\cite{BBP05}, where the so-called BBP phase transition was observed.
See also \cite{John, BBP05, BaikSil06} for sample covariance matrices,
\cite{FurKom81, Peche06,FePe, CDF09, PRS} for deformed Wigner models
and \cite{LV} for information-plus-noise models. These investigations were
first extended to perturbations of arbitrary 
rank in \cite{RaoSil09} and \cite{BaiYao08b} for sample covariance
matrices,  in \cite{CDFF10} for deformed Wigner models and 
in \cite{Capitaine14} for information-plus-noise models. 
It is pointed out in \cite{CDFF10}  that the subordination function
(relative to the free additive convolution of a semicircular distribution
with the limiting spectral distribution of the perturbation) plays
an important role in the fact that some eigenvalues of the deformed
Wigner model separate from the bulk. Note that in \cite{Capitaine11}
it is explained how the results of \cite{RaoSil09} and \cite{BaiYao08b}
in the setting of sample covariance matrices can also be described in terms
of the subordination function related to the free multiplicative convolution
of a Marchenko-Pastur distribution with the limiting spectral distribution of the multiplicative
perturbation. 

Similar results have been obtained in \cite[Theorems 2.7 and 2.8]{BGRao09} for 
multiplicative perturbations of the type $A_N^{1/2}U_{N}^{*}B_{N}U_{N}A_N^{1/2}$, where
$A_N-I_N\ge0$ is of fixed rank $p\in\mathbb N$ and $B_N\ge0$.

In this paper we investigate the asymptotic behaviour of the eigenvalues and corresponding
eigenvectors of the following models:
\begin{itemize}
\item $X_N=A_N+U_N^*B_NU_N$, where $A_N=A_N^*,B_N=B_N^*$ are deterministic, and $U_N$ 
is a Haar-distributed unitary random matrix;
\item $X_N=A_N^{1/2}U_{N}^{*}B_{N}U_{N}A_N^{1/2},$ where $A_N,B_N\ge0$ are deterministic, 
and $U_N$ is a Haar-distributed unitary random matrix;
\item $X_N=A_NU_N^*B_NU_N$, where $A_N,B_N\in {\rm U}(N)$ are deterministic unitary matrices and 
$U_N$ is a Haar-distributed unitary random matrix.
\end{itemize}
In the first two models we assume that $A_N$ and $B_N$ have compactly supported limiting 
eigenvalue distributions $\mu$ and $\nu$, respectively. In the third model we assume that the
supports of $\mu$ and $\nu$ are not equal to the entire unit circle. A fixed number 
$p\in\mathbb N$ of eigenvalues of $A_N$---the {\em spikes}---lay outside the support of 
$\mu$ for all $N\in\mathbb N$, but all other eigenvalues of $A_N$ converge uniformly to the 
support of $\mu$. A similar hypothesis is made about $B_N$ and $\nu$. We answer the following 
questions:
\begin{itemize}
\item When are some of the eigenvalues of $X_N$ almost surely located outside the support of the 
limiting spectral distribution $\mu\boxplus\nu$ (respectively $\mu\boxtimes\nu$) of $X_N$? In
other words, when does the spectrum of $X_N$ have {\em outliers} almost surely?
\item What is the behaviour of the eigenvectors corresponding to the outliers of $X_N$?
\end{itemize}

When there are no spikes,
it was shown in \cite{ColMal11} that, given a neighborhood $V$ of the support of $\mu\boxplus\nu$,
the eigenvalues of $X_{N}$ are almost surely contained in $V$
for large $N\in\mathbb{N}$. This paper extends the results of \cite{BGRao09}
to perturbations of arbitrary rank, and it also extends the free probabilistic
interpretation of outlier phenomena in terms of subordination functions
as described in \cite{CDFF10} for deformations of Wigner models.
Indeed, the occurence and role of Biane's subordination functions \cite{Biane98} in the 
analysis of the interaction spikes/outliers is quite natural. We clarify this in the 
additive case through the following heuristics, leaving the precise statements to Section
\ref{sec:Notation-and-Statements}.
\subsection{Heuristics}
Let ${\bf a}$ and ${\bf b}$ be two free selfadjoint random variables in a tracial W*-probability space. 
Biane showed \cite[Theorem 3.1]{Biane98} that there exists an analytic self-map $\omega\colon\mathbb
C^+\to\mathbb C^+$ of the upper half-plane $\mathbb C^+=\{x+iy\colon y>0\}$ 
(depending on ${\bf a,b}$) so that 
\begin{equation}\label{BianeCondExp}
\mathbb E_{\mathbb C[{\bf a}]}\left[(z-({\bf a+b}))^{-1}\right]=(\omega(z)-{\bf a})^{-1},
\quad z\in\mathbb C^+.
\end{equation}
Here $\mathbb E_{\mathbb C[{\bf a}]}$ denotes the conditional expectation onto the
von Neumann algebra generated by ${\bf a}$ and $1$. The function $\omega$ is referred to as the {\em
subordination function}. (This formula is a particular case of Biane's 
result. Formula \eqref{BianeCondExp} appears in this form in \cite[Appendix]{Voiculescu00}.) 
The subordination function continues analytically via Schwarz reflection through the 
complement in $\mathbb R$ of the spectrum of ${\bf a+b}$. If $A_N\to{\bf a}$ in distribution as
$N\to\infty$, while a single eigenvalue $\theta$, common to all of the matrices $A_N$, stays outside the 
spectrum of ${\bf a}$, this eigenvalue will disappear in the large $N$ limit, in the sense that it will not 
influence the spectrum of ${\bf a}$. Thus, the analytic function $\omega(z)$ will  not 
be prevented {\it a priori} from taking the value $\theta$.

However, if relation \eqref{BianeCondExp} were true with ${\bf a}$ and ${\bf b}$ replaced by
$A_N$ and $U_{N}^{*}B_{N}U_{N}$, respectively, and with the {\em same} subordination function 
$\omega$,
then any number $\rho$ in the domain of analyticity of $\omega$ so that $\omega(\rho)
=\theta$ must generate a polar singularity for the right-hand side of \eqref{BianeCondExp}.
Therefore, each such $\rho$ must generate a similar singularity for the term on the left-hand side
of the same equality, thus necessarily producing an eigenvalue of $A_N+U_{N}^{*}B_{N}U_{N}$. 
While this scenario is not true, we prove that an approximate version does hold. Namely, we show  that a 
compression of the matrix 
$$
\mathbb E\left[(z-(A_N+U_N^*B_NU_N))^{-1}\right]^{-1}+A_N
$$
to a subspace $V_N$ is close to $\omega(z)I_{V_N}$, almost surely as $N\to\infty$. This insight is crucial 
in our arguments.

It follows from our results that a remarkable new phenomenon occurs: a 
single spike of $A_{N}$ can generate asymptotically a finite, possibly arbitrarily large,
set of outliers for $X_{N}$. This arises from the fact that the restriction
to the real line of some subordination functions may be many-to-one, that is, with the  above
notation, the set $\omega^{-1}(\{\theta\})$ may have cardinality strictly
greater than 1, unlike the subordination function related to free convolution with
a semicircular distribution that was used in \cite{CDFF10}.

The case of multiplicative perturbations is based on similar ideas, with
the subordination function $\omega$ replaced by its multiplicative counterparts 
\cite[Theorems 3.5 and 3.6]{Biane98}. 

In addition to outliers, we investigate the corresponding eigenspaces of $X_N$. It turns out that 
the angle between these eigenvectors and the eigenvectors associated to the original spikes
is determined by Biane's subordination function, this time via its derivative.

The paper is organized as follows. In Section 2, we describe in detail the matrix models
to be analysed, and state the main results of the paper. In Section 3 we introduce  
free convolutions and the analytic transforms involved in their study. We give the functional 
equations characterising the subordination functions. In Section 4 we collect and prove a number of of auxiliary results, and in Section 5 we prove the main results.

\section{Notation and statements of the main results\label{sec:Notation-and-Statements}}

We denote by $\mathbb{C}^{+}=\{z\in\mathbb{C}\colon\Im z>0\}$ the  upper half-plane,
by $\mathbb{C}^{-}=\{z\in\mathbb{C}\colon\Im z<0\}$ the lower half-plane, and by
$\mathbb D=\{z\in\mathbb C\colon|z|<1\}$ the unit disc in $\mathbb{C}$. The
topological boundary of the unit disc is denoted by $\mathbb T=\partial\mathbb D$.
For any vector subspace $E$ of $\mathbb{C}^m$, $P_E$ denotes the orthogonal projection onto $E$.
$M_{m}(\mathbb{C})$ stands for the set of $m\times m$ matrices with complex entries, 
 ${\rm GL}_{m}(\mathbb{C})$ for the subset of invertible matrices, and ${\rm U}(m)\subset
{\rm GL}_{m}(\mathbb{C})$ for the unitary group. The operator norm of a 
matrix $X$ is $\|X\|$, its spectrum is $\sigma(X)$ its kernel is $\ker X$, its trace is $\mathrm{Tr}_m(X)$
and its normalized trace is $\mathrm{tr}_m(X)=\frac1m\mathrm{Tr}_m(X)$. The 
eigenvalues of a Hermitian matrix $X$ are denoted 
\[
\lambda_{1}(X)\geq\cdots\geq\lambda_{m}(X),
\]
and the probability measure
\[
\mu_{X}=\frac{1}{m}\sum_{i=1}^{m}\delta_{\lambda_{i}(X)}
\]
is the empirical  eigenvalue distribution of $X$. When $X$ is unitary, its eigenvalues are ordered 
decreasingly according to the size of their principal arguments in $[0,2\pi)$. 
If $X\in M_m(\mathbb C)$ is a normal matrix, we denote by $E_X$ its spectral measure. Thus, if 
$S\subseteq \mathbb C$ is a Borel set, then $E_X(S)$ is the orthogonal projection onto the linear span of
all eigenvectors of $X$ corresponding to eigenvalues in $S$. The support 
of a measure $\rho$ on $\mathbb{C}$ is denoted $\text{supp}(\rho)$. Given any set
$K\subseteq\mathbb R$, we define its $\varepsilon$-neighbourhood by 
$$
K_\varepsilon=\left\{x\in\mathbb R\colon\inf\{|x-y|\colon y\in K\}<\varepsilon\right\}.
$$
As long as there is no risk of confusion, the same notation will be used when $K$ and $K_\varepsilon$ are 
subsets of $\mathbb T$. Open intervals in $\mathbb R$ and open arcs in $\mathbb T$ are denoted 
$(a,b)$.

As already seen in Section 1, the Cauchy (or Cauchy-Stieltjes) transform of a finite positive Borel measure
$\rho$ on $\mathbb C$ is an analytic function
defined by
\begin{equation}\label{CauchyTr2}
G_\rho(z)=\int_\mathbb C\frac{1}{z-t}\,d\rho(t),\quad z\in\mathbb C\setminus\text{supp}(\rho).
\end{equation}
IWe only consider finite measures $\rho$ supported  $\mathbb R$ or 
$\mathbb T$. We denote by $F_\rho$ the reciprocal of $G_\rho,$ that is, $F_\rho(z)=1/G_\rho
(z)$. The moment generating function for $\rho$ is
\begin{equation}\label{MomGen1}
\psi_\rho(z)=\int_\mathbb C\frac{tz}{1-tz}\,d\rho(t),\quad z\in\mathbb C\setminus\left\{
z\in\mathbb C \colon\frac1z\not\in\text{supp}(\rho)\right\}.
\end{equation}
The  $\eta$-transform of $\rho$ is defined as 
$$
\eta_\rho(z)=\frac{\psi_\rho(z)}{1+\psi_\rho(z)}.
$$ 
The relevant analytic properties of the transforms above are presented in Subsections 
\ref{sec:boxplus}--\ref{sec:boxtimesT}.

The {\em free additive convolution} of the Borel probability measures $\mu$ and $\nu$ on 
$\mathbb R$ is denoted by $\mu\boxplus\nu$ and the {\em free multiplicative convolution} 
of the Borel probability measures $\mu$ and $\nu$ either on $[0,+\infty)$ or on $\mathbb T$
is denoted by $\mu\boxtimes\nu$. Given $\mu,\nu$, denote by $\omega_1$
and $\omega_2$ the {\em subordination functions} corresponding to the free convolution 
$\mu\boxplus\nu$. They are known to be analytic on the complement of $\text{supp}(\mu\boxplus
\nu)$. As the name suggests, they satisfy an analytic subordination property in the sense of
Littlewood:
\begin{equation}\label{subord1}
G_{\mu\boxplus\nu}(z)=G_\mu(\omega_1(z))=G_\nu(\omega_2(z)).
\end{equation}
A similar result holds for the multiplicative counterparts of the subordination functions. We have:
\begin{equation}\label{subord1'}
\psi_{\mu\boxtimes\nu}(z)=\psi_\mu(\omega_1(z))=\psi_\nu(\omega_2(z)),
\end{equation}
where $\omega_1$ and $\omega_2$ are analytic on $\{z\in\mathbb C\colon1/z\not\in\text{supp}
(\mu\boxtimes\nu)\}$. Free convolutions are defined in subsections 
\ref{sec:boxplus}--\ref{sec:boxtimesT}, and the subordination functions are defined via
functional equations in subsections \ref{subsubsec:eq+}--\ref{subsubsec:eqT}

In the following three subsections we describe our models and state the main results. To avoid
dealing with too many special cases, we make the following technical assumption, which will be
in force for the remainder of the paper, except for Remark \ref{atoms}.
\begin{equation}\label{hyp:2.5}
\text{Neither of the two limiting measures }\mu,\nu\text{ is a point mass}.
\end{equation}
Under this assumption, the subordination functions extend continuously to the boundary
(see Lemmas \ref{lem:extension}, \ref{lem:extensionX}, and \ref{lem:extensionT}). Our results
however remain substantially valid without this assumption, and we discuss in Remark \ref{atoms} the 
relevant modifications.

\subsection{Additive perturbations\label{subsec:+perturb}}
Here are the ingredients for constructing the additive matrix 
model $X_N=A_N+U_N^*B_NU_N$:
\begin{itemize}
\item Two compactly supported Borel probability measures $\mu$ and $\nu$ on $\mathbb 
R$.
\item A positive integer $p$ and
fixed real numbers $$\theta_{1}\ge\theta_{2}\ge\cdots\ge\theta_{p}$$
which do not belong to $\text{supp}(\mu)$.

\item A sequence $(A_{N})_{N\in\mathbb N}$ of deterministic Hermitian matrices of size
$N\times N$ such that 

\begin{itemize}
\item   $\mu_{A_{N}}$ converges weakly to $\mu$ as $N\to\infty$;

\item for $N\ge p$ and $\theta\in\{\theta_1,\dots,\theta_p\}$, the sequence $(\lambda_{n}(A_{N}))_{n=1}^{N}$ satisfies  $$\mathrm{card}(\{n:\lambda_n(A_N)=\theta\})=\mathrm{card}(\{i:\theta_i=\theta\});$$

\item the eigenvalues of $A_{N}$ which are not equal to some $\theta_{i}$
converge uniformly to $\text{supp}(\mu)$ as $N\to\infty$, that is
\[
\lim_{N\to\infty}\max_{\lambda_{n}(A_{N})\notin\{\theta_{1},\dots,\theta_{p}\}}\text{dist}(\lambda_{n}(A_{N}),\text{supp}(\mu))=0.
\]
\end{itemize}

\item A positive integer $q$ and
fixed real numbers $$\tau_{1}\ge \tau_{2}\ge\cdots\ge\tau_{q}$$
which do not belong to $\text{supp}(\nu)$.
\item A sequence $(B_{N})_{N\in\mathbb N}$ of deterministic Hermitian matrices of size
$N\times N$ such that 

\begin{itemize}
\item $\mu_{B_{N}}$ converges weakly to $\nu$ as $N\to\infty$;

\item for $N\ge q$ and $\tau\in\{\tau_1,\dots,\tau_q\}$, the sequence $(\lambda_{n}(B_{N}))_{n=1}^{N}$ satisfies  $$\mathrm{card}(\{n:\lambda_n(B_N)=\tau\})=\mathrm{card}(\{i:\tau_i=\tau\});$$

\item the eigenvalues of $B_{N}$ which are not equal to some $\tau_{j}$
converge uniformly to $\text{supp}(\nu)$ as $N\to\infty$.
\end{itemize}

\item A sequence $(U_{N})_{N\in\mathbb N}$ of unitary random matrices such that the distribution
of $U_{N}$ is the normalized Haar measure on the unitary group ${\rm U}(N)$.
\end{itemize}

It is known from \cite{Voiculescu91} that the asymptotic eigenvalue distribution of $X_N$ is
$\mu\boxplus\nu$. In the following statement we take advantage of the fact, discussed later, that the functions $\omega_1$ and
$\omega_2$  extend continuously to the real line.  The points in $\mathbb R\setminus\mathrm{supp}(\mu\boxplus\nu)$ satisfying $\omega_1(\rho)=\theta$ are isolated but they may accumulate to $\mathrm{supp}(\mu\boxplus\nu)$.
We denote by $P_N$ and $Q_N$ the projection onto the space generated by the eigenvectors 
corresponding to the spikes of $A_N$ and $B_N$, respectively.  These can also be written as
\begin{equation}\label{PQ}
P_N=E_{A_N}(\{\theta_1,\dots,\theta_p\}),\quad Q_N=E_{B_N}(\{\tau_1,\dots,\tau_q\}), 
\end{equation}
in terms of the spectral measures of $A_N$ and $B_N$.
\begin{thm}\label{Main+}
With the above notation, set $K=\mathrm{supp}(\mu\boxplus\nu)$,
$$K'=K\cup \left[\bigcup_{i=1}^p
\omega_1^{-1}(\{\theta_i\})\right]\cup\left[\bigcup_{j=1}^q
\omega_2^{-1}(\{\tau_j\})\right],
$$
and let $\omega_1,\omega_2$ be the subordination functions satisfying {\rm (\ref{subord1})}.
The following results hold almost surely for large $N${\rm:}
\begin{enumerate}

\item Given $\varepsilon>0$, we have $\sigma(X_N)\subset K'_\varepsilon.$

\item Fix a number $\rho\in K'\setminus K$, let $\varepsilon>0$ be such that $(\rho-2\varepsilon,
\rho+2\varepsilon)\cap K'=\{\rho\}$, and set $k=\mathrm{card}(\{i:\omega_1(\rho)=\theta_i\})$, 
$\ell=\mathrm{card}(\{j:\omega_2(\rho)=\tau_j\})$. Then
$$
\mathrm{card}(\{\sigma(X_N)\cap(\rho-\varepsilon,\rho+\varepsilon)\})=k+\ell.
$$

\item With $\rho$ and $\varepsilon$ as in part $(2)$, we have
$$\left\|P_NE_{X_N}((\rho-\varepsilon,\rho+\varepsilon))P_N-\frac1{\omega'_1(\rho)}E_{A_N}(\{\omega_1(\rho)\})\right\|<\varepsilon$$
and
$$
\left\|Q_NU_NE_{X_N}((\rho-\varepsilon,\rho+\varepsilon))U_N^*Q_N-\frac1{\omega'_2(\rho)}E_{B_N}(\{\omega_2(\rho)\})\right\|<\varepsilon.
$$

\item With $\rho$ and $\varepsilon$ as in part $(2)$, suppose in addition that $\ell=0$. Then
$$
\left|\|E_{A_N}(\{\omega_1(\rho)\})\xi\|^2-\frac{\|\xi\|^2}{\omega'_1(\rho)}\right|<\varepsilon\|\xi\|^2,
\quad \xi\in E_{X_N}((\rho-\varepsilon,\rho+\varepsilon))\mathbb C^N.
$$
Analogously, if $k=0$, then
$$
\left|\|E_{U_N^*B_NU_N}(\{\omega_2(\rho)\})\xi\|^2-\frac{\|\xi\|^2}{\omega'_2(\rho)}\right|<\varepsilon
\|\xi\|^2,\quad \xi\in E_{X_N}((\rho-\varepsilon,\rho+\varepsilon))\mathbb C^N.
$$
\end{enumerate}
\end{thm}

\begin{rem}
In case $k>0$ and $\ell>0$, the result of (3) above implies the following. 
Let $\{\xi_1,\dots,
\xi_{k+\ell}\}$ be an orthonormal basis of $E_{X_N}((\rho-\varepsilon,\rho+\varepsilon))$. Then, almost
surely for large $N$, we have
$$
\left|\sum_{n=1}^{k+\ell}\|E_{A_N}(\theta)\xi_n\|_2^2 -
\frac{\delta_{\omega_1(\rho),\theta}k}{\omega_1'(\rho)}\right|<\varepsilon
$$ 
and 
$$
\left|\sum_{n=1}^{k+\ell}\|E_{U_N^*B_NU_N}(\tau)\xi_{n}\|^2_2
-\frac{\delta_{\omega_2(\rho),\tau}\ell}{\omega_2'(\rho)}\right|<\varepsilon,
$$
for $\theta,\tau\in\mathbb R$, where $\delta_{\omega_2(\rho),\tau}$ is the usual Kronecker symbol.
Thus, assertion (4) is a strenghtening of (3) in the special case $k\ell=0$.
\end{rem}

\subsection{Multiplicative perturbations of nonnegative matrices\label{subsec:Xperturb+}}

Here are the ingredients for constructing the multiplicative model
$X_N=A_N^{1/2}U_N^*B_NU_NA_N^{1/2}$:
\begin{itemize}
\item Two compactly supported Borel probability measures $\mu$ and $\nu$ on $[0,+
\infty)$.
\item A positive integer $p$, and
fixed positive numbers 
$$
\theta_{1}\ge\theta_{2}\ge\cdots\ge\theta_{p}>0
$$
which do not belong to $\text{supp}(\mu)$. 
\item A sequence $(A_{N})_{N\in\mathbb N}$ of deterministic nonnegative matrices of size
$N\times N$ such that 
\begin{itemize}
\item $\mu_{A_{N}}$ converges weakly to $\mu$ as $N\to\infty$,
\item for $N\ge p$  and $\theta\in\{\theta_1,\dots,\theta_p\},$ the sequence  $\{\lambda_{n}(A_{N})\}_{n=1}^{N}$ 
satisfies $\mathrm{card}(\{n\colon\lambda_n(A_N)=\theta\})=\mathrm{card}(\{i\colon\theta_i=\theta\});$
\item the eigenvalues of $A_{N}$ which are not equal to some $\theta_{i}$
converge uniformly to $\text{supp}(\mu)$ as $N\to\infty$.
\end{itemize}

\item A positive integer $q$, and fixed positive numbers 
$$
\tau_{1}\ge\tau_{2}\ge\cdots\ge\tau_{q}>0
$$
which do not belong to $\text{supp}(\nu)$. 
\item A sequence $(B_{N})_{N\in\mathbb N}$ of deterministic nonnegative matrices of size
$N\times N$ such that 

\begin{itemize}
\item $\mu_{B_{N}}$ converges weakly to $\nu$ as $N\to\infty$,
\item for $N\ge q$  and $\tau\in\{\tau_1,\dots,\tau_q\},$ the sequence  $\{\lambda_{n}(B_{N})\}_{n=1}^{N}$ 
satisfies $\mathrm{card}(\{n\colon\lambda_n(B_N)=\tau\})=\mathrm{card}(\{i\colon\tau_i=\tau\});$
\item the eigenvalues of $B_{N}$ which are not equal to some $\tau_{j}$
converge uniformly to $\text{supp}(\nu)$ as $N\to\infty$.
\end{itemize}

\item A sequence $(U_{N})_{N\in\mathbb N}$ of random matrices such that the distribution
of $U_{N}$ is the normalized Haar measure on the unitary group ${\rm U}(N)$.
\end{itemize}
It is known from \cite{Voiculescu91} that the asymptotic eigenvalue distribution of $X_N$ is
$\mu\boxtimes\nu$.
The projections $P_N$ and $Q_N$ used in the following statement were defined in \eqref{PQ}.

\begin{thm}\label{MainX+}
With the above notations, let $\omega_1,\omega_2$ be the subordination functions satisfying 
\emph{(\ref{subord1'})}, set $K=\mathrm{supp}(\mu\boxtimes\nu)$, 
$v_j(z)=\omega_j\left({1}/{z}\right)$, $j=1,2$, and
$$
K'=K\cup\left[\bigcup_{i=1}^p
v_1^{-1}(\{1/\theta_i\})\right]\cup\left[\bigcup_{j=1}^q
v_2^{-1}(\{1/\tau_j\})\right].
$$ 
The following results hold almost surely for large $N${\rm:}
\begin{enumerate}
\item Given  $\varepsilon>0$ we have $\sigma(X_N)\subset K'_\varepsilon.$

\item Fix a positive number $\rho\in K'\setminus K$, let 
$\varepsilon>0$ be such that $(\rho-2\varepsilon,\rho+2\varepsilon)\cap K'=\{\rho\}$ and set 
$k=\mathrm{card}(\{i\colon v_1(\rho)=1/\theta_i\})$, $\ell=\mathrm{card}(\{j\colon v_2(\rho)=1/\tau_j\})$. Then
$$
\mathrm{card}(\{\sigma(X_N)\cap(\rho-\varepsilon,\rho+\varepsilon)\})=k+\ell.
$$

\item With $\rho$ and $\varepsilon$ as in part $(2)$, we have
$$\left\|P_NE_{X_N}((\rho-\varepsilon,\rho+\varepsilon))P_N-
\frac{\rho\omega_1(1/\rho)}{\omega'_1(1/\rho)}E_{A_N}(\{1/\omega_1(1/\rho)\})\right\|<\varepsilon$$
and
$$
\left\|Q_NU_NE_{X_N}((\rho-\varepsilon,\rho+\varepsilon))U_N^*Q_N-
\frac{\rho\omega_2(1/\rho)}{\omega_2'(1/\rho)}E_{B_N}(\{1/\omega_2(1/\rho)\})\right\|<\varepsilon.
$$

\item With $\rho$ and $\varepsilon$ as in part $(2)$, suppose in addition that $\ell=0$. Then
$$
\left|\|E_{A_N}(\{1/\omega_1(1/\rho)\})\xi\|^2-
\frac{\|\xi\|^2\rho\omega_1(1/\rho)}{\omega'_1(1/\rho)}\right|<\varepsilon\|\xi\|^2,
\quad \xi\in E_{X_N}((\rho-\varepsilon,\rho+\varepsilon))\mathbb C^N.
$$
Analogously, if $k=0$, then
$$
\left|\|E_{U_N^*B_NU_N}(\{1/\omega_2(1/\rho)\})\xi\|^2-
\frac{\|\xi\|^2\rho\omega_2(1/\rho)}{\omega'_2(1/\rho)}\right|<\varepsilon
\|\xi\|^2,\quad \xi\in E_{X_N}((\rho-\varepsilon,\rho+\varepsilon))\mathbb C^N.
$$
\end{enumerate}

\end{thm}

\subsection{Multiplicative perturbations of unitary matrices\label{subsec:XperturbT}}
Finally, we describe the ingredients for the construction of the multiplicative matrix model
$X_N=A_NU_N^*B_NU_N$ with unitary $A_N$ and $B_N$:
\begin{itemize}
\item Two Borel probability measures $\mu$ and $\nu$ on $\mathbb T$ with nonzero first moments
such that $\text{supp}(\mu\boxtimes\nu)\neq\mathbb T$.
\item A positive integer $p$ and fixed complex numbers 
$
\theta_{1},\cdots,\theta_{p}\in\mathbb T
$ which do not belong to $\text{supp}(\mu)$ and such that 
$$2\pi>\arg\theta_{1}\ge\cdots\ge\arg\theta_{p}\ge0.$$

\item A sequence $(A_{N})_{N\in\mathbb N}$ of deterministic unitary matrices of size
$N\times N$ such that 

\begin{itemize}
\item $\mu_{A_{N}}$ converges weakly to $\mu$ as $N\to\infty$,
\item for $N\ge p$  and $\theta\in\{\theta_1,\dots,\theta_p\},$ the sequence  $\{\lambda_{n}(A_{N})\}_{n=1}^{N}$ 
satisfies 
$$
\mathrm{card}(\{n\colon\lambda_n(A_N)=\theta\})=\mathrm{card}(\{i\colon\theta_i=\theta\});
$$

\item the eigenvalues of $A_{N}$ which are not equal to some $\theta_{i}$
converge uniformly to $\text{supp}(\mu)$ as $N\to\infty$.
\end{itemize}

\item A positive integer $q$ and
fixed complex numbers 
$
\tau_{1},\dots,\tau_{q}\in\mathbb T
$  which do not belong to $\text{supp}(\nu)$ and such
that $$2\pi>\arg\tau_{1}\ge\cdots\ge\arg\tau_{q}\ge0.$$

\item A sequence $(B_{N})_{N\in\mathbb N}$ of deterministic unitary matrices of size
$N\times N$ such that 

\begin{itemize}
\item $\mu_{B_{N}}$ converges weakly to $\nu$ as $N\to\infty$,
\item for $N\ge q$  and $\tau\in\{\tau_1,\dots,\tau_q\},$ the sequence  $\{\lambda_{n}(B_{N})\}_{n=1}^{N}$ 
satisfies 
$$
\mathrm{card}(\{n\colon\lambda_n(B_N)=\tau\})=\mathrm{card}(\{i\colon\tau_i=\tau\});
$$

\item the eigenvalues of $B_{N}$ which are not equal to some $\tau_{j}$
converge uniformly to $\text{supp}(\nu)$ as $N\to\infty$.
\end{itemize}

\item A sequence $(U_{N})_{N\in\mathbb N}$ of random matrices such that the distribution
of $U_{N}$ is the normalized Haar measure on the unitary  group ${\rm U}(N)$.
\end{itemize}

It is known from \cite{Voiculescu91} that the asymptotic eigenvalue distribution of $X_N$ is
$\mu\boxtimes\nu$. When $\rho\in\mathbb T$ and $\varepsilon>0$, the interval $(\rho-\varepsilon,\rho+\varepsilon)$ consists of those numbers in $\mathbb T$ whose argument differs from $\arg\rho$ by less than $\varepsilon$.  With this convention,
Theorem \ref{MainX+} holds \emph{verbatim} in the unitary case as well.


\section{Free convolutions\label{sec:Free-convolution}}

Free convolutions arise as natural analogues of classical convolutions
in the context of free probability theory. 
For two Borel probability measures $\mu$ and $\nu$ on the real line, one defines the free 
additive convolution $\mu\boxplus\nu$ as the distribution of $a+b$, where $a$ and $b$
are free self-adjoint random variables with distributions $\mu$ and
$\nu$, respectively. Similarly, if both $\mu,\nu$ are supported on $[0,+\infty)$ or on $
\mathbb T$, their free multiplicative convolution $\mu\boxtimes\nu$ is the distribution of
the product $ab$, where, as before, $a$ and $b$
are free, positive in the first case, unitary in the second, random variables with distributions 
$\mu$ and $\nu$, respectively. The product $ab$ of two
free positive random variables is usually not positive, but it has the same moments as the 
positive random variables $a^{1/2}ba^{1/2}$ and $b^{1/2}ab^{1/2}.$ We refer to \cite{VDN92} for an
introduction to free probability theory and to \cite{Voiculescu86,V2,BercoviciVoiculescu} for the
definitions and main properties of free convolutions. In this section,
we recall the analytic approach developed in \cite{Voiculescu86,V2}
to calculate the free convolutions of measures, as well as the analytic subordination
property \cite{voic-fish1,Biane98,Voiculescu00} and related results.

\subsection{Additive free convolution\label{sec:boxplus}}
Recall from \eqref{CauchyTr2} the definition of the Cauchy-Stieltjes of a finite positive Borel
measure $\mu$ on the real line:
$$
G_\mu(z)=\int_{\mathbb R}\frac{1}{z-t}\,d\mu(t),\quad z\in\mathbb C\setminus\text{supp}
(\mu).
$$
This function maps $\mathbb{C}^{+}$ to $\mathbb{C}^{-}$ and $\lim_{y\uparrow+\infty}iyG_{\mu}(iy)=\mu(\mathbb{R})$.
Conversely, any analytic function $G:\mathbb{C}^{+}\to\mathbb{C}^{-}$
for which $\lim_{y\uparrow+\infty}iyG(iy)$ is finite is of the form
$G=G_{\mu}|_{\mathbb{C}^{+}}$ for some finite positive Borel
measure $\mu$ on $\mathbb R$. When $\mu$ has compact support, the
function $G_{\mu}$ is also analytic at $\infty$ and $G_{\mu}(\infty)=0$
(see \cite[Chapter 3]{akhieser} for these results).
The measure $\mu$ can be recovered from its Cauchy-Stieltjes transform
as a weak limit
\begin{equation}\label{ConvPoisson}
d\mu(x)=\lim_{y\downarrow0}\frac{-1}{\pi}\Im G_{\mu}(x+iy)\, dx.
\end{equation}
(This holds for signed measures as well.) 
The density of (the absolutely continuous part of) $\mu$ relative
to Lebesgue measure is calculated as 
\[
\frac{d\mu}{dx}(x)=\lim_{y\downarrow0}\frac{-1}{\pi}\Im G_{\mu}(x+iy)
\]
for almost every $x$ relative to the Lebesgue measure. In particular, $\mathbb{R}\setminus{\rm supp}
(\mu)$ can be described as the set of those points $x\in\mathbb{R}$ with
the property that $G_{\mu}|_{\mathbb{C}^{+}}$ can be continued analytically
to an open interval $I\ni x$ such that $G_{\mu}|I$ is real-valued. On the
other hand, 
\[
\lim_{y\downarrow0}\frac{-1}{\pi}\Im G_{\mu}(x+iy)=+\infty
\]
almost everywhere relative to the singular part of $\mu$. 
Indeed, these facts follow from the straightforward observation that $(-\pi)^{-1}\Im G_{\mu}(x+iy)
,$ $y>0$, is the Poisson integral of $\mu$. See \cite{GarnettBook}
for these aspects of harmonic analysis.

It is often convenient to work with the reciprocal Cauchy-Stieltjes
transform $F_{\mu}(z)={1}/{G_{\mu}(z)}$, which defines an analytic
self-map of the upper half-plane. This function enjoys the following properties:
\begin{itemize}
\item[(a)] For any $z\in\mathbb C^+$, $\Im F_\mu(z)\ge\mu(\mathbb R)^{-1}\Im z$. If
equality holds at one point of $\mathbb C^+$, then it holds at all points, and $\mu=
\mu(\mathbb R)\delta_{-\mu(\mathbb R)\Re F_\mu(i)}$.
\item[(b)] In particular, the function 
\begin{equation}\label{h}
h_\mu(z)=F_\mu(z)-\mu(\mathbb R)^{-1}z,\quad z\in\mathbb C^+,
\end{equation}
is a self-map of $\mathbb C^+$ unless $\mu$ is a point mass, in which case $h_\mu$ is a real constant.
\item[(c)] If $\mu$ is compactly supported, there exist a real number $\alpha$ and a finite positive Borel 
measure $\rho$ on $\mathbb R$ with $\text{supp}(\rho)$ included in the convex hull of
$\text{supp}(\mu)$ such that
\begin{equation}\label{Nevanlinna}
F_\mu(z)=\alpha+\mu(\mathbb R)^{-1}z+\int_\mathbb R\frac{1}{t-z}\,d\rho(t),\quad
z\in\mathbb C\setminus\text{supp}(\rho).
\end{equation}
Conversely, if $F\colon\mathbb C^+\to\mathbb C^+$ extends to an analytic real-valued function to the 
complement in $\mathbb R$ of a compact set,  and if $\lim_{y\to+\infty}F(iy)=\infty$, then there exists a 
compactly supported positive Borel measure $\mu$ on $\mathbb R$ so that $F=F_\mu$. The value
$\mu(\mathbb R)$ is determined by $\mu(\mathbb R)=\lim_{y\to+\infty}iy/F(iy)$.
\item[(d)] If $\mu(\mathbb R)=1$ and $\rho$ is as in \eqref{Nevanlinna}, we have $\rho(\mathbb 
R)=\int_\mathbb R\left(t-\int_\mathbb R s\,d\mu(s)\right)^2\,d\mu(t)$ and 
$\alpha=-\int_\mathbb Rt\,d\mu(t)$. 
\end{itemize}
Equation \eqref{Nevanlinna} is a special case of the Nevanlinna representation of analytic self-maps
of the upper half-plane \cite[Chapter 3]{akhieser}:
\begin{equation}\label{N}
F(z)=a+bz+\int_\mathbb R\frac{1+tz}{t-z}\,d\Omega(t),\quad z\in\mathbb C^+,
\end{equation}
where $a\in\mathbb R, b\ge0$ and $\Omega$ is a positive finite Borel measure on $\mathbb R$.
We identify $a=\Re F(i)$, $b=\lim_{y\to+\infty}F(iy)/iy$, $\Omega(\mathbb R)=\Im F(i)-b.$ If
$(1+t^2)\in L^1(\mathbb R,d\Omega(t))$, then  
\eqref{N} reduces to \eqref{Nevanlinna}, with $b=\mu(\mathbb R)^{-1}$ and $d\rho(t)=(1+t^2)
\,d\Omega(t),\alpha=a-\int t\,d\Omega(t)$.

The Cauchy-Stieltjes transform of a compactly supported probability
measure $\mu$ is conformal in the neighborhood of $\infty$, and
its functional inverse $G_{\mu}^{-1}$ is meromorphic at zero with
principal part $1/z$. The $R$-transform \cite{Voiculescu86} of $\mu$ is the convergent
power series defined by 
\[
R_{\mu}(z)=G_{\mu}^{-1}(z)-\frac{1}{z}.
\]
 The free additive convolution of two compactly supported probability
measures $\mu$ and $\nu$ is another compactly supported probability
measure characterized by the identity 
\[
R_{\mu\boxplus\nu}=R_{\mu}+R_{\nu}
\]
satisfied by these convergent power series.

\subsection{Multiplicative free convolution on $[0,+\infty)$\label{sec:boxtimes+}}
Recall from \eqref{MomGen1} the definition of the moment-generating function of a 
Borel probability measure $\mu$ on $[0,+\infty)$:
$$
\psi_\mu(z)=\int_{[0,+\infty)}\frac{zt}{1-zt}\,d\mu(t), \quad z\in\mathbb C\setminus
\left\{z\in\mathbb C\colon\frac1z\in\text{supp}(\mu)\right\}.
$$
This function is related to the Cauchy-Stieltjes transform of $\mu$ via the relation
$$\psi_\mu(z)=\frac1zG_\mu\left(\frac1z\right)-1.$$ It satisfies the following properties, for 
which we refer to \cite[Section 6]{BercoviciVoiculescu}:
\begin{itemize}
\item $\psi_\mu(\mathbb C^+)\subseteq\mathbb C^+$.
\item $\psi_\mu((-\infty,0))\subseteq(\mu(\{0\})-1,0)$ and 
$$
\psi_\mu(i\mathbb C^+)
\subseteq\left\{z\in\mathbb C\colon\left|z-\frac{\mu(\{0\})-1}{2}\right|<\frac{1-\mu(\{0\})}{2}\right\}.
$$
In addition,
$$
\lim_{x\downarrow-\infty}\psi_\mu(x)=\mu(\{0\})-1,\quad\lim_{x\uparrow0}\psi_\mu(x)
=0,\quad \lim_{x\uparrow0}\psi_\mu'(x)=\int_{[0,+\infty)}t\,d\mu(t).
$$
\item In particular, if $\text{supp}(\mu)$ is compact and not equal to $\{0\}$, then $\psi_\mu$ is 
injective on some neighbourhood of zero in $\mathbb C$.
\item $\psi_\mu$ is injective on $i\mathbb C^+$.
\end{itemize}
It is often convenient to work with the eta transform, or Boolean cumulant function,
$$\eta_\mu(z)=\frac{\psi_\mu(z)}{1+\psi_\mu(z)}.$$ It inherits from $\psi$ the following
properties:
\begin{itemize}
\item[(a)] $\pi>\arg\eta_\mu(z)\ge\arg z$ for all $z\in\mathbb C^+$, where $\arg$ takes
values in $(0,\pi)$ on $\mathbb C^+$. Moreover, if equality holds for one point in 
$\mathbb C^+$, it holds for all points in $\mathbb C^+$, and $\mu=\delta_{\eta_\mu(z)/
z}=\delta_{\eta_\mu'(0)}$ for any $z\in\mathbb C^+$.
\item[(b)] $\lim_{x\uparrow0}\eta_\mu(x)=0$ and $\lim_{x\uparrow0}\eta_\mu'(x)=\int_{[0,
+\infty)}t\,d\mu(t).$ In particular, if $\text{supp}(\mu)$ is compact and different from $\{0\}$, 
then $\eta_\mu$ is injective on some neighbourhood of zero in $\mathbb C$.
\item[(c)] If $\mu\ne\delta_0$, $\eta_\mu$ is strictly increasing from $(-\infty,0]$ to $(\mu(\{0\})^{-1}(\mu(\{0\})-1),
0]$, where $\mu(\{0\})^{-1}(\mu(\{0\})-1)$ should be replaced by $-\infty$ if $\mu(\{0\})=0$. 
Moreover, $\eta_\mu$ is injective on $i\mathbb C^+$.
\item[(d)] Conversely, if an analytic function $\eta\colon\mathbb C^+\to\mathbb C^+$
satisfies $\pi>\arg\eta(z)\ge\arg z$ for all $z\in\mathbb C^+$ and 
$\lim_{x\uparrow0}\eta(x)=0$, then $\eta=\eta_\mu$ for some Borel probability measure on
$[0,+\infty)$ \cite[Proposition 2.2]{B-B-imrn}.
\end{itemize}
The 
$\Sigma$-transform \cite{V2,BercoviciVoiculescu} of a compactly supported measure  $\mu\ne\delta_0$ 
is the convergent power series defined by
$$
\Sigma_\mu(z)=\frac{\eta_\mu^{-1}(z)}{z},
$$
where $\eta_{\mu}^{-1}$ is the inverse of $\eta_\mu$ relative to composition.
The free multiplicative convolution of two compactly supported probability
measures $\mu\ne\delta_0\ne\nu$ is another compactly supported probability
measure characterized by the identity 
\[
\Sigma_{\mu\boxtimes\nu}(z)=\Sigma_{\mu}(z)\Sigma_{\nu}(z)
\]
in a neighbourhood of $0$.

\subsection{Multiplicative free convolution on $\mathbb T$\label{sec:boxtimesT}}
The analytic transforms involved in the study of multiplicative free convolution on 
$\mathbb T$ are formally the same ones as in Subsection \ref{sec:boxtimes+}, but their analytical 
properties are different. Thus,
$$
\psi_\mu(z)=\int_{\mathbb T}\frac{zt}{1-zt}\,d\mu(t), \quad z\in\mathbb C\setminus
\left\{z\in\mathbb C \colon\frac1z\in\text{supp}(\mu)\right\}.
$$
It satisfies $\Re\psi_\mu(z)>-\frac12$ for all $|z|<1$.
We work almost exclusively with the eta transform, or Boolean cumulant function,
$$\eta_\mu(z)=\frac{\psi_\mu(z)}{1+\psi_\mu(z)}.$$ The following properties of $\eta_\mu$ are 
relevant to our study:
\begin{itemize}
\item[(a)] For any $z\in\mathbb D$, we have $|\eta_\mu(z)|\leq|z|$.
If equality holds at one point in $\mathbb D\setminus\{0\}$, it holds 
at all points in $\mathbb D$, and $\mu=\delta_{\eta_\mu(z)/z}=\delta_{\eta_\mu'(0)}$ for 
any $z\in\mathbb D\setminus\{0\}$.
\item[(b)] $\eta_\mu(0)=0$ and $\eta_\mu'(0)=\int_{\mathbb T}t\,d\mu(t).$ In particular
$\eta_\mu$ is injective on a neighbourhood of zero in $\mathbb C$ if and only if
$\int_{\mathbb T}t\,d\mu(t)\neq0$.
\item[(c)] The function $\eta_\mu$ continues via Schwarz reflection through the set 
$\{z\in\mathbb T\colon \overline{z}\not\in\text{supp}(\mu)\}$, that is
$$
\eta_\mu(z)=\frac{1}{\overline{\eta_\mu\left(\frac1{\overline{z}}\right)}},\quad|z|>1.
$$
\item[(d)] For almost all points $1/x$ with respect to the absolutely continuous part of $\mu$ (relative 
to the Haar measure on $\mathbb T$), the nontangential limit of $\eta_\mu$ at $x$ belongs to 
$\mathbb D$, and for almost all points $1/x$ in the complement of the support of the absolutely continuous part of $\mu$, 
the nontangential limit of $\eta_\mu$ at $x$ belongs to $\mathbb T$. Moreover, if $\mu$ has
a singular component, then for almost all points $1/x$ with respect to this component, the
nontangential limit of $\eta_\mu$ at $x$ equals one.
\item[(e)] Conversely, if an analytic function $\eta\colon\mathbb D\to\mathbb D$
satisfies $\eta(0)=0$, then $\eta=\eta_\mu$ for a unique Borel probability measure on
$\mathbb T$ \cite[Proposition 3.2]{B-B-imrn}.

\end{itemize}
When $\int_{\mathbb T}t\,d\mu(t)\neq0$, we define the
$\Sigma$-transform \cite{V2,BercoviciVoiculescu} of $\mu$ as the convergent power series
$$
\Sigma_\mu(z)=\frac{\eta_\mu^{-1}(z)}{z}.
$$
Again, the free multiplicative convolution of two probability measures $\mu$ and $\nu$ with nonzero first 
moments is another probability measure with nonzero first moment characterized by the identity 
\[
\Sigma_{\mu\boxtimes\nu}(z)=\Sigma_{\mu}(z)\Sigma_{\nu}(z)
\]
in a neighbourhood of $0$.

If both of $\mu$ and $\nu$ have zero first moment, then
$\mu\boxtimes\nu$ is the Haar (uniform) distribution on $\mathbb T$, see \cite{VDN92}. From 
now on, we always assume that all our probability measures on $\mathbb T$ have 
nonzero first moments.

\subsection{Analytic subordination}

The analytic subordination phenomenon for free convolutions, as described in  
\eqref{subord1} and \eqref{subord1'}, was first noted by Voiculescu in \cite{voic-fish1} for free 
additive convolution of compactly supported probability measures. Later, Biane \cite{Biane98}
extended the result to free additive convolutions of arbitrary probability measures on $\mathbb R$,
and also found a subordination result for multiplicative free convolution. More importantly, he proved the 
stronger result (see heuristics in the introduction) that the conditional expectation of the resolvent of
a sum or product of free random variables onto the algebra generated by one of them is
in fact also a resolvent. In \cite{Voiculescu00}, Voiculescu deduced this property from
the fact that such a conditional expectation is a coalgebra morphism for certain
coalgebras, and through this observation he extended the subordination property to
free convolutions of operator-valued distributions. For our purposes, considerably less
than that is required: we essentially need only complex analytic properties of these 
functions. 
\subsubsection{The subordination functions equations for free additive 
convolution\label{subsubsec:eq+}} Given Borel probability measures $\mu$ and $\nu$ on 
$\mathbb R$, there exist two unique analytic functions $\omega_1,\omega_2\colon
\mathbb C^+\to\mathbb C^+$ such that 
\begin{enumerate}
\item $\lim_{y\to+\infty}\omega_j(iy)/iy=1$, $j=1,2$;
\item
\begin{equation}\label{EqSubord+}
\omega_1(z)+\omega_2(z)-z=F_\mu(\omega_1(z))=F_\nu(\omega_2(z))=F_{\mu\boxplus
\nu}(z),\quad z\in\mathbb C^+.
\end{equation}
\item In particular (see \cite{BelBer07}), for any $z\in\mathbb C^+\cup\mathbb R$ so that 
$\omega_1$ is analytic at $z$, $\omega_1(z)$ is the attracting fixed point of the self-map
of $\mathbb C^+$ defined by 
$$
 w\mapsto F_\nu(F_\mu(w)-w+z)-(F_\mu(w)-w).
$$
A similar statement, with $\mu,\nu$ interchanged, holds for $\omega_2$.
\end{enumerate}
We note that \eqref{EqSubord+} implies that the functions $\omega_1,\omega_2$ 
continue analytically accross an interval $(\alpha,\beta)\subseteq\mathbb R$ such that 
$\omega_1|_{(\alpha,\beta)}$ and $\omega_2|_{(\alpha,\beta)}$ are real valued if and only if 
the same is true for $F_{\mu\boxplus\nu}$. For the sake of providing an intrinsic characterization 
of the correspondence betweem spikes and outliers, we formalize and slightly strenghten this 
remark in the following lemma.  Here we use the functions $h_\mu,h_\nu$ defined by (\ref{h}).
\begin{lem}\label{lem:extension}
Consider two compactly supported Borel probability measures $\mu$ and $\nu$, neither of them 
a point mass. Then the subordination functions $\omega_1$ and $\omega_2$ have extensions
to $\overline{\mathbb{C}^{+}}\cup\{\infty\}$ with
the following properties{\rm:}
\begin{enumerate}
\item [(a)]$\omega_1,\omega_2:\overline{\mathbb{C}^{+}}\cup\{\infty\}\to
\overline{\mathbb{C}^{+}}\cup\{\infty\}$ are continuous.
\item [(b)] If $x\in\mathbb{R}\setminus\text{{\rm supp}}(\mu\boxplus\nu)$
then the functions $\omega_1$ and $\omega_2$ continue meromorphically to a neighborhood of $x$,
$\omega_1(x)=h_\nu(\omega_2 (x))+x\in(\mathbb{R}\cup\{\infty\})\setminus{\rm supp}(\mu)$, and $\omega_2(x)=h_\mu(\omega_1 (x))+x\in(\mathbb{R}\cup\{\infty\})\setminus{\rm 
supp}(\nu)$. If $\omega_1(x)=\infty$, then $\omega_2(x)=x-\int_\mathbb R t\,d\mu(t)\in
\mathbb R$, and if $\omega_2(x)=\infty$, then $\omega_1(x)=x-\int_\mathbb R t\,d\nu(t)\in
\mathbb R$.
\item[(c)] Conversely, suppose that $\omega_1$ continues meromorphically to a neighbourhood of 
a point $x\in\mathbb R$ and that $\omega_1(y)\in\mathbb R$ when $y\in(x-\delta,x+\delta)\setminus
\{x\}$ for some $\delta>0$. If $x\in\mathrm{supp}(\mu\boxplus\nu)$, then $x$ is an isolated atom for
$\mu\boxplus\nu$.
\end{enumerate}
\end{lem}
In the context of part (b) of the above lemma, we note that $h_\mu$ is analytic around infinity, and $h_\mu(\infty)=-\int_\mathbb R t\,d\mu(t)$. 
\begin{proof}
Part (a) was proved in \cite[Theorem 3.3]{Belinschi08}. Fix $x\in\mathbb{R}\setminus
{\rm supp}(\mu\boxplus\nu)$. Equation \eqref{EqSubord+} indicates that 
$\omega_1$ and $\omega_2$ must take real values on $(x-\delta,x+\delta)\setminus\{x\}$ for some $
\delta>0$. Schwarz reflection implies that $\omega_1$ and $\omega_2$ have meromorphic continuations
with real values on $\mathbb R$ accross the corresponding intervals.

The relation $G_{\mu\boxplus\nu}(z)=G_\mu(\omega_1(z))$ shows that the limit $\lim_{z\to y}
G_\mu(\omega_1(z))$ is real for $y\in(x-\delta,x+\delta)$ and therefore $\mu(\{\omega_1(y)
\colon y\in(x-\delta,x+\delta)\})=0$ by the Stieltjes inversion formula. In particular, $\omega_1(x)
\not\in\mathrm{supp}(\mu)$. To conlcude the proof of $(b)$, suppose that 
$\omega_1(x)=\infty$. It follows from 
\eqref{EqSubord+} in conjunction with items (c) and (d) of subsection \ref{sec:boxplus} that
$$
\omega_2(x)=\lim_{z\to x}F_\mu(\omega_1(z))-\omega_1(z)+z=x+\lim_{w\to\infty}F_\mu(w)-
w=x-\int_\mathbb R t\,d\mu(t),
$$
so that $\omega_2$ is analytic at $x$. The statement for $\omega_2(x)=\infty$ follows by symmetry.

Finally, suppose that the hypotheses of (c) is satisfied. It was observed in \cite{Belinschi08} that
$\omega_2(y)$ is also real for $y\in(x-\delta,x+\delta)$. (Indeed, if $\Im \omega_2(y)>0$,
relation  \eqref{EqSubord+} implies 
\begin{equation}\label{plunk}
\omega_1(y)+\omega_2(y)=y+F_{\mu\boxplus\nu}(y)=y+F_\nu(\omega_2(y)),
\end{equation}
and therefore $\Im F_\nu(\omega_2(y))=\Im\omega_2(y)$. This relation can only hold when $\nu$ is
a point mass, a case which we excluded.) Now, \eqref{plunk} implies that $F_{\mu\boxplus\nu}$ is
continuous and real-valued on $(x-\delta,x+\delta)\setminus\{x\}$, and this yields the desired
conclusion via the Stieltjes inversion formula.
\end{proof}

\subsubsection{The subordination functions equations for multiplicative free
convolution on $[0,+\infty)$\label{subsubsec:eqX}} Given Borel probability measures $\mu,
\nu$ on $[0,+\infty)$, there exist two unique analytic functions $\omega_1,\omega_2
\colon\mathbb C\setminus[0,+\infty)\to\mathbb C\setminus[0,+\infty)$ so that 
\begin{enumerate}
\item $\pi>\arg\omega_j(z)\ge\arg z$ for $z\in\mathbb C^+$ and $j=1,2$;
\item
\begin{equation}\label{EqSubordX+}
\frac{\omega_1(z)\omega_2(z)}{z}=\eta_\mu(\omega_1(z))=\eta_\nu(\omega_2(z))=\eta_{\mu\boxtimes
\nu}(z),\quad z\in\mathbb C\setminus[0,+\infty).
\end{equation}
\item In particular (see \cite{BelBer07}), for any $z\in\mathbb C^+\cup\mathbb R$ so that 
$\omega_1$ is analytic at $z$, the point $h_1(z):=\omega_1(z)/z$ is the attracting fixed point of 
the self-map of ${\mathbb C}\setminus[0,+\infty)$ defined by
$$
 w\mapsto\frac{w}{\eta_\mu(zw)} 
\eta_\nu\left(\frac{\eta_\mu(zw)}{w}\right).
$$
A similar statement, with $\mu,\nu$ interchanged, holds for $\omega_2$.
\end{enumerate}
A version of Lemma \ref{lem:extension} holds for multiplicative free convolution on $[0,+\infty)$. 
Since the proof is similar to the proof of Lemma \ref{lem:extension} and of Lemma \ref{lem:extensionT} 
below, we omit it. Item (a) appears in the proof of \cite[Theorem 3.2]{AIHP}.
\begin{lem}\label{lem:extensionX}
Consider two compactly supported Borel probability measures $\mu,\nu$ on $[0,+\infty)$,
neither of them a point mass. Then the restrictions of the subordination functions $\omega_1$ and 
$\omega_2$ to $\mathbb C^+$ have extensions to $\overline{\mathbb{C}^{+}}\cup\{\infty\}$ with the 
following properties:
\begin{enumerate}
\item [(a)] $\omega_1,\omega_2:\overline{\mathbb{C}^{+}}\cup\{\infty\}\to\overline{\mathbb{C}^{+}}\cup\{\infty\}$ are continuous.
\item [(b)] If $1/x\in\mathbb{R}\setminus\text{{\rm supp}}(\mu\boxtimes\nu)$
then the functions $\omega_1$ and $\omega_2$ continue analytically to a neighborhood of $x$,
$1/\omega_1(x)
=\omega_2(x)/x\eta_{\nu}(\omega_2(x))
\in\mathbb{R}\setminus{\rm supp}(\mu)$, 
and $1/\omega_2 (x)=\omega_1(x)/x\eta_{\mu}(\omega_1(x))\in\mathbb{R}\setminus\text{{\rm supp}}(\nu)$.
\end{enumerate}
\end{lem}

\subsubsection{The subordination functions equations for multiplicative free convolution on 
$\mathbb T$\label{subsubsec:eqT}} Given Borel probability measures $\mu,\nu$ on $\mathbb T$ with 
nonzero first moments, there exist unique analytic functions $\omega_1,\omega_2
\colon\mathbb D\to\mathbb D$  so that 
\begin{enumerate}
\item $|\omega_j(z)|\le|z|$, $j=1,2$, $z\in\mathbb D$;
\item
\begin{equation}\label{EqSubordXT}
\frac{\omega_1(z)\omega_2(z)}{z}=\eta_\mu(\omega_1(z))=\eta_\nu(\omega_2(z))
=\eta_{\mu\boxtimes\nu}(z),\quad z\in\mathbb D.
\end{equation}
\item In particular (see \cite{BelBer07}), if $z\in\mathbb D\cup\mathbb T$ and
$\omega_1$ is analytic at $z$, then the point $h_1(z):=\omega_1(z)/z$ is the attracting fixed point of 
the map
$$
\mathbb D\ni w\mapsto\frac{w}{\eta_\mu(zw)} 
\eta_\nu\left(\frac{\eta_\mu(zw)}{w}\right)\in\mathbb D.
$$
A similar statement, with $\mu,\nu$ interchanged, holds for $\omega_2$.
\end{enumerate}
\begin{lem}\label{lem:extensionT}
Consider two Borel probability measures $\mu,\nu$ on $\mathbb T$ with nonzero first moments, 
neither of them a point mass. Suppose that $\mathbb T\setminus\mathrm{supp}(\mu\boxtimes\nu)
\neq\varnothing$. Then the subordination functions $\omega_1$ and $\omega_2$ 
have extensions to $\mathbb{T}$ with the following properties:
\begin{enumerate}
\item [(a)] $\omega_1,\omega_2:\mathbb{D}\cup\mathbb{T}\to\mathbb{D}\cup\mathbb{T}$ are 
continuous.
\item [(b)] If $1/x\in\mathbb{T}\setminus\text{{\rm supp}}(\mu\boxtimes\nu)$
then the functions $\omega_1$ and $\omega_2$ continue analytically to a neighborhood of $x$,
$1/\omega_1(x)=\omega_2(x)/x\eta_{\nu}(\omega_2(x))\in\mathbb{T}\setminus{\rm supp}(\mu)$, 
and $1/\omega_2(x)=\omega_1(x)/x\eta_{\mu}(\omega_1(x))\in\mathbb{T}\setminus\text{{\rm supp}}(\nu)$.
\end{enumerate}
\end{lem}
\begin{proof}
Part (a) can be found in \cite[Theorem 3.6]{AIHP}. Fix $1/x\in\mathbb{T}\setminus
{\rm supp}(\mu\boxtimes\nu)$. Equation \eqref{EqSubordXT} coupled with items (d) and (e)
of Subsection \ref{sec:boxtimesT} indicate clearly that 
$\omega_1,\omega_2$ must take values in $\mathbb T$ at least a.e. on a neighbourhood of $x$. 
As proved in \cite[Proposition 1.9 (a)]{AIHP} if, say, $\omega_1$ does not 
reflect analytically through a neighbourhood of $x$, then for any $\epsilon>0$ the 
set of nontangential limits $\{\sphericalangle\lim_{z\to c}\omega_1(z)\colon \arg(xe^{-i\epsilon})
<\arg(c)<\arg(xe^{i\epsilon})\}$ of $\omega_1$ around $x$ is dense in $\mathbb T$. As 
$\mathbb T\setminus\text{supp}(\mu)$ is nonempty, many of these limits will fall in the domain 
of analyticity of $\eta_\mu$. In particular, we may choose an arbitrary interval $I=\{e^{it}\colon
t\in[s_1,s_2]\}$ strictly included in the domain of analyticity of $\eta_\mu$ and we will be able to 
find points $1/c_n\in\mathbb{T}\setminus{\rm supp}(\mu\boxtimes\nu)$ tending to $1/x$ so 
that $\sphericalangle\lim_{z\to c_n}\omega_1(z)=d_n\in I.$ Obviously, in that case any limit
point of $\{d_n\}_{n\in\mathbb N}$ will still belong to $I$, and hence be in the domain of
analyticity of $\eta_\mu$. Pick such a limit point $d_0$. Note that, as a trivial consequence of the 
Julia-Carath\'eodory Theorem \cite[Chapter I, Exercises 6 and 7]{GarnettBook}, $\eta'_\mu(w)
>0$ for any $w\in\mathbb T$ in the domain of analyticity of $\eta_\mu$, and thus $\eta_\mu'
(d_0)>0$, which implies that $\eta_\mu$ is conformal on a neighbourhood $U$ of $d_0$ (in
$\mathbb C)$. Now recall that
$$
\eta_{\mu\boxtimes\nu}(c_n)=\sphericalangle\lim_{z\to c_n}\eta_{\mu\boxtimes\nu}(z)=
\sphericalangle\lim_{z\to c_n}\eta_\mu(\omega_1(z))=\eta_\mu(d_n).
$$
Letting $n$ go to infinity (along a subsequence, if necessary), and recalling that $1/x\not\in{\rm 
supp}(\mu\boxtimes\nu),$ we obtain $\eta_{\mu\boxtimes\nu}(x)=\eta_\mu(d_0)$. Both 
functions are analytic around the two respective points from $\mathbb T$, so the conformality
of $\eta_\mu$ on $U$ allows us to find $\eta_\mu(U)$ as a neighbourhood of 
$\eta_{\mu\boxtimes\nu}(x)$ on which the compositional inverse $\eta_\mu^{-1}$ can be 
defined. We write $\eta_\mu^{-1}\circ\eta_{\mu\boxtimes\nu}$ on some convex neighbourhood 
$W$ of $x$ which is small enough so that $\eta_{\mu\boxtimes\nu}(W)\subset\eta_\mu(U)$ (the
existence of $W$ is guaranteed by the continuity of $\eta_{\mu\boxtimes\nu}$ around $x$). Pick 
points $z_n\in\mathbb D$ so that $|z_n-c_n|<\frac1n$. Clearly $\lim_{n\to\infty}z_n=x$, so 
that for all $n\in\mathbb N$ large enough, $z_n\in W$. Pick a piecewise linear path going 
consecutively through the $z_n$'s, so that, by the convexity of $W$, this path stays in $W$
and necessarily converges to $x$. For any $z$ in this path, we have
$$
\eta_{\mu\boxtimes\nu}(z)=\eta_\mu(\omega_1(z))\implies
\omega_1(z)=(\eta_\mu^{-1}\circ\eta_{\mu\boxtimes\nu})(z),
$$
which, by analytic continuation and analyticity of $\eta_\mu^{-1}\circ\eta_{\mu\boxtimes\nu}$
on $W$, forces $\omega_1$ to be analytic on $W$, providing a contradiction. Thus,
$\omega_1$ extends analytically through $x$. The same argument shows that $\omega_2$
extends analytically around $x$. The last statement of (b) above follows again from
the simple remark that $\omega_2(x)=x\eta_{\mu}(\omega_1(x))/\omega_1(x).$
\end{proof}
Unlike the case of free additive convolution, the functions $\omega_1$ and $\omega_2$ are bounded
on $\mathbb D$ and hence do not have pole singularities on $\mathbb T$.

\section{Preliminary results}

The proofs of our main results will be based largely on both scalar- and matrix-valued 
analytic function methods, as well as on some elementary results from operator theory. We start
by collecting  some results which apply to both additive and multiplicative
models. We use the notation introduced in Section \ref{sec:Notation-and-Statements}.

\subsection{Boundary behaviour and convergence of some families of analytic functions}

The following convergence result for sequences of Nevanlinna-type
functions is necessary in the analysis of eigenvectors corresponding to outliers. 
$C({\bf X})$ denotes the space of complex-valued continuous functions on a topological 
space ${\bf X}$. We use the notation $|\rho|$ for the total variation measure of a signed measure 
$\rho$ on $\mathbb R$. That is, $|\rho|=\rho^++\rho^-$, where $\rho=\rho^+-\rho^-$ is the
Hahn decomposition of $\rho$. The total variation of $\rho$ is $\|\rho\|=|\rho|(\mathbb R)$.

\begin{lem}\label{conv}
Let  $\{\rho_N\}_{N\in\mathbb N}$ be a family of signed measures on $\mathbb R$
satisfying the following properties{\rm:}
\begin{itemize}
\item There exists  $m\in\mathbb R$ such that $\mathrm{supp}(\rho_N)\subseteq[-m,m]$
for all $N\in\mathbb N$;
\item $\sup_{N\in\mathbb N}{\|\rho_N\|}<\infty$;
\item $\rho_N\to0$ in the weak${}^*$-topology as $N\to\infty$, that is, 
$$\lim_{N\to\infty}\int_\mathbb Rf(t)\,d\rho_N(t)=0\quad f\in C([-m,m]).$$
\end{itemize}
Then there exists a sequence $\{v_N\}_{N\in\mathbb N}\subset(0,+\infty)$ converging to zero,
independent of $z$,
such that 
$$
\left|\int_{\mathbb R}\frac{1}{z-t}\,d\rho_N(t)\right|<
\frac{|z|^2+\sup_{N\in\mathbb N}{\|\rho_N\|}}{(\Im z)^2}v_N,\quad z\in\mathbb C^+,N\in
\mathbb N.
$$
\end{lem}
\begin{proof}
Fix $\varepsilon>0$. It is known from
\cite{akhieser} that
\begin{equation}\label{condition}
\lim_{N\to\infty}
\int_{\mathbb R}\frac{1}{z-t}\,d\rho_N(t)=0\textrm{ uniformly in }z\in\mathbb C^++iy\\
\end{equation}
for any fixed $y>0$. We prove that there exists $N(\varepsilon)\in\mathbb N$ such that 
$$
\left|\int_{\mathbb R}\frac{1}{z-t}\,d\rho_N(t)\right|
\frac{(\Im z)^2}{|z|^2+\sup_{N\in\mathbb N}{\|\rho_N\|}}<\varepsilon,\quad  z\in
\mathbb C^+,N\ge N(\varepsilon).
$$
The proof is then completed by setting $$v_N=\sup_{z\in\mathbb C^+}
\left|\int_{\mathbb R}\frac{1}{z-t}\,d\rho_N(t)\right|
\frac{(\Im z)^2}{|z|^2+\sup_{N\in\mathbb N}{\|\rho_N\|}}.$$

Indeed, as noted in \eqref{condition}, there exists  $N(\varepsilon)\in\mathbb 
N$ such that  $$\left|\int_{\mathbb R}\frac{1}{z-t}\,d\rho_N
(t)\right|<\varepsilon,\quad z\in\mathbb C^++i\varepsilon,N\ge N(\varepsilon).$$ On the other hand, for $\Im z
\in(0,\varepsilon]$ we have
\begin{eqnarray*}
\left|\int_{\mathbb R}\frac{1}{z-t}\,d\rho_N(t)\right|
\frac{(\Im z)^2}{|z|^2+\sup_{N\in\mathbb N}{\|\rho_N\|}} & \leq &
\int_{\mathbb R}\left|\frac{1}{z-t}\right|\,d|\rho_N|(t)
\frac{(\Im z)^2}{|z|^2+\sup_{N\in\mathbb N}{\|\rho_N\|}}\\
&\leq & \frac{1}{\Im z}|\rho_N|([-m,m])\frac{(\Im z)^2}{|z|^2+\sup_{N\in\mathbb N}{\|\rho_N\|}}\\
& \leq & \Im z\frac{\|\rho_N\|}{|z|^2+\sup_{N\in\mathbb N}{\|\rho_N\|}} < \Im z\leq\varepsilon.
\end{eqnarray*}
We conclude that $\lim_{N\to\infty}v_N=0$, as desired.
\end{proof}

An analogous result holds for $\mathbb T$.

\begin{lem}\label{convT}
Let $\{\omega_N\}_{N\in\mathbb N}$ be a family of analytic self-maps of the unit disc
such that
$\omega_N(0)=0$ and the limit $\omega(z)=\lim_{N\to\infty}\omega_N(z)$ exists for all $z\in
\mathbb D$.
Then there exists a sequence $\{v_N\}_{N\in\mathbb N}\subset(0,+\infty)$ converging to zero,
independent of $z$, such that
$$
\left|\omega_N(z)-\omega(z)\right|<
\frac{4}{(1-| z|)^2}v_N,\quad z\in\mathbb C^+,N\in
\mathbb N.
$$
\end{lem}
\begin{proof}
There exists a unique probability 
measure $\rho_N$ on $\mathbb R$ such that
\begin{equation}\label{4.4}
\omega_N(z)=2iz\int_\mathbb R\frac{1}{\frac{t-i}{t+i}-z}\,d\rho_N(t),\quad|z|<1.
\end{equation}
This is seen by applying \eqref{N} to the map 
$$
\mathbb C^+\ni z\mapsto i+\omega_N
\left(\frac{z-i}{z+i}\right)\in\mathbb C^+.
$$
Similarly,
$$
\omega(z)=2iz\int_\mathbb R\frac{1}{\frac{t-i}{t+i}-z}\,d\rho(t),\quad|z|<1.
$$
Since $t\mapsto\frac{t-i}{t+i}$ is 
a bijection from $\mathbb R$ to $\mathbb T\setminus\{1\}$, we have
$$
\left|\frac{t-i}{t+i}-z\right|\ge1-|z|,\quad z\in\mathbb D.
$$ 
Thus, 
$$
\left|\omega_N(z)-\omega(z)\right|(1-|z|)^2\leq2|z|(1-|z|)^2\int_\mathbb R\left|
\frac{1}{\frac{t-i}{t+i}-z}\right|\,d|\rho_N-\rho|(t)\leq4(1-|z|)<{\epsilon},
$$
provided that $|z|>1-\varepsilon/4$.
 On the other hand, Montel's theorem implies that the convergence
 $\omega_N\to\omega$ is uniform on compact subsets of $\mathbb D$, 
and thus there exists $N(\varepsilon)\in\mathbb N$ such that 
$$
|\omega_N(z)-\omega(z)|<
\varepsilon,\quad|z|\leq1-\epsilon/4,N\ge N(\varepsilon).
$$
The sequence
$$
v_N=\sup_{z\in\mathbb D}\left|\omega_N(z)-\omega(z)\right|(1-|z|)^2
$$
satisfies the conclusions of the lemma.
\end{proof}

The following lemma from  \cite[Appendix]{CD07}  is proved using 
ideas from \cite{HT}.  We use the notation $\mathcal{D}(\mathbb{R})$ for the space of infinitely 
differentiable, compactly supported functions $ \varphi:\mathbb{R}\to\mathbb{C}$.
\begin{lem}\label{HT} Let $\Delta$ be an analytic function on $\mathbb{C}\setminus \mathbb{R}$ 
which satisfies
\begin{equation*}\label{nestimgdif}
\vert \Delta(z)\vert \leq (\vert z\vert +K)^\alpha P(\vert \Im z\vert ^{-1})
\end{equation*} 
for some numbers $\alpha\ge1$, $K\ge0$, and polynomial $P$ with nonnegative coefficients.
For every $h\in\mathcal D(\mathbb R)$  there exists a constant $c>0$ independent of $\Delta$ such that
$$
\limsup _{y\rightarrow 0^+}\left|\int _{\mathbb{R}}h(x)
\Delta(x+iy)dx\right| \le c.
$$
\end{lem}

We also record a result from \cite[Lemma 6.3]{Capitaine11} on the boundary behaviour of
a certain Poisson kernel convolution (recall \eqref{ConvPoisson} and the comments following it).
We use $\mathbb E$
to denote the expectation. If $v\colon\mathbb R\to\mathbb C$ is a continuous function
and $A$ is a selfadjoint matrix, then $v(A)$ is constructed using the continuous functional calculus.

\begin{lem}\label{approxcauchy}
Given a deterministic $N\times N$ Hermitian matrix $C_N$, a random $N\times N$
Hermitian matrix $X_N$, and a continuous function $h\colon\mathbb R\to\mathbb R$ with compact support, we have
$$
\mathbb{E}\left[{ \rm Tr}_N \left[ h (X_N) C_N\right] \right]= \lim_{y\downarrow 0}
\frac{1}{\pi} \Im \int_\mathbb R  \mathbb{E} \left[ {\rm Tr}_N 
[\left(X_N-(t+iy)I_N\right)^{-1}C_N] \right]h(t)\, dt. 
$$
\end{lem}

\subsection{Matrix-valued functions and maps}
An essential ingredient in our analysis is the resolvent of $A_N$ and of the matrices
$X_N$ (depending on the model considered, $X_N=A_N+U_N^*B_NU_N$, $X_N=A_N^{1/2}
U_N^*B_NU_NA_N^{1/2}$ or $X_N=A_NU_N^*B_NU_N$). We denote by
\begin{equation}\label{RN}
R_N(z)=\left(zI_N-X_N\right)^{-1},\quad z\not\in\sigma(X_N)
\end{equation}
the resolvent of $X_N$. It is a random matrix-valued rational function with poles in  $\mathbb R$ for the 
first two models and $\mathbb T$  for the third. For the first two models, it has the following properties:
\begin{enumerate}
\item $R_N(\overline{z})=R_N(z)^*$. In particular, $R_N(x)$ is selfadjoint if $x\in\mathbb R
\setminus\sigma(X_N).$
\item $R_N$ is analytic at $\infty$, and $\lim_{z\to\infty}zR_N(z)=I_N$, where $I_N$ denotes the 
$N\times N$ identity matrix. The limit is in the norm topology of $M_N(\mathbb C)\otimes 
L^\infty({\rm U}(N),m_N)$, where $m_N$ denotes the Haar measure on ${\rm U}(N)$.
\item With the notation $\Re T=(T+T^*)/2$ and $\Im T=(T-T^*)/2i$ for the real and imaginary
parts of $T$, respectively, we have 
$$
-\Im R_N(z)=\Im z\left((\Im z)^2I_N+(\Re zI_N-X_N)^2\right)^{-1}\ge\frac{\Im z}{|z|^2+2|\Re z|\|
X_N\|+\|X_N\|^2}I_N.
$$
This last quantity is uniformly bounded below in $N$ for $z$ in any fixed compact set $K
\subseteq\mathbb C^+$. In particular, if $C>0$ is such that $\sup_N  \|X_N\| \leq C$,
\begin{equation}
-\Im\mathbb E\left[(zI_N-X_N)^{-1}\right]\ge\frac{\Im z}{|z|^2+2|\Re z|C+C^2}I_N.
\end{equation}
\end{enumerate}
For the unitary model, a slightly different property is needed.
\begin{enumerate}
\item[(a)] If $z\in\mathbb D$, we have $\sigma(zX_N)\subset\mathbb D$, and thus  $\sigma\left((I_N-zX_N)^{-1}\right)\subset\{w\in\mathbb C\colon\Re w>
1/2\}$.  Therefore 
$$
\Re\left[\frac1zR_N\left(\frac1z\right)\right]>\frac12I_N,\quad z\in\mathbb D.
$$ 
(This observation uses the fact that $X_N$ is unitary, and hence normal.)
\end{enumerate}

The following lemma is a fairly straightforward  generalization of a result of Hurwitz. A similar result appears 
in \cite{BGRao09}. In the statement, we use $K_\delta$ to denote the subset of $\gamma$ consisting of all points at distance strictly less 
than $\delta$ from $K$. In the special case $K=\{\rho\}$, we write $(\rho-\delta,\rho+\delta)$ instead of $K_\delta$.

\begin{lem}\label{alt-Benaych-Rao}
Let $\gamma$ be a simple analytic curve in ${\mathbb C}$,  let $K\subsetneq\gamma$ 
be compact, and let $r$ be a positive integer.  Consider an analytic function 
$F\colon\overline{\mathbb{C}}\setminus K\to M_{r}
(\mathbb{C})$ such that $F(z)$ is diagonal
for each $z\in{\mathbb{C}}\setminus 
K$, $F(\infty)=I_r$ and $
z\mapsto(F(z))_{ii}\in\mathbb C$ has only
simple zeros, all of which are contained in $\gamma\setminus K$,
$1\leq i\leq r$. Fix $\delta>0$ such that $\det(F)$ has no zeroes on the boundary of $K_\delta$ relative to 
$\gamma$, and let $
\rho_{1},\dots,\rho_{s}\in\gamma$
be a list of those points $z\in\mathbb C\setminus K_\delta$ for which $F(z)$  is not
invertible.

Suppose that there exist  positive numbers $\{\delta_{N}\}_{N\in\mathbb{N}}$
 and  analytic maps $F_{N}\colon\overline{\mathbb{C}}\setminus 
K_{\delta_{N}}\to M_{r}(\mathbb{C})$, $N\in\mathbb N,$
such that{\rm:}
\begin{enumerate}
\item $\lim_{N\to\infty}\delta_N=0;$
\item $F_{N}(z)$ is invertible for  $z\in\mathbb{C}\setminus\gamma$ and $N\in\mathbb N;$ and
\item $F_{N}$ converges  to $F$ uniformly on compact subsets of $\overline{\mathbb{C}}
\setminus K$.
\end{enumerate}

Then{\rm:}
\begin{enumerate}
\item [(i)] $\dim(\ker(F(\rho_{j}))$ equals the order of $\rho_{j}$ as
a zero of $z\mapsto\det(F(z));$
\item [(ii)] Given $\varepsilon>0$ such that 
\[
\varepsilon<\frac{1}{2}\min\{|\rho_{i}-\rho_{j}|,{\rm dist}(\rho_{i},K_{\delta})\colon1\leq i\neq j\leq{s}\},
\]
there exists an integer $N_{0}$ such that for $N\ge N_{0}$, we have
\begin{enumerate}
\item [-]counting multiplicities, $\det(F_{N})$ has exactly $\dim(\ker(F(\rho_{j})))$ zeroes in
$(\rho_{j}-\varepsilon,\rho_{j}+\varepsilon)\subset\gamma$, $j =1,\ldots,s$,
and
\item [-] $\{z\in
{\mathbb{C}}\setminus K_{\delta}\colon\det(F_{N}(z))=0\}\subset\bigcup_{j=1}^{s}(\rho_{j}-\varepsilon,\rho_{j}+\varepsilon).$ 
\end{enumerate}
\end{enumerate}
\end{lem}

\begin{proof}
Assertion (i) is obvious.
The functions $f_{N}(z)=\det(F_{N}(z))$ converge to  $f(z)=\det(F(z))$ 
uniformly on compact subsets of $\overline{\mathbb{C}}\setminus K$.  
The Theorem of Hurwitz (see \cite[Kapitel 8.5]{Remmert84}) guarantees that, for sufficiently large $N$, $f_{N}$ 
has  (counting multiplicities) exactly as many zeros as $f$ in $\overline{\mathbb{C}}
\setminus K_{\delta}$. All the zeros of $f_{N}$ were assumed to be in $\gamma$ and 
therefore these zeros cluster around $\{\rho_{1},\dots,\rho_{s}\}$
in the following sense: for any given $\varepsilon>0$, there exists
an $N_{\varepsilon}\in\mathbb{N}$ so that 
\[
\{z\in
{\mathbb{C}}\setminus K_{\delta}\colon\det(F_{N}(z))=0\}\subset
\bigcup_{j=1}^{s}(\rho_{j}-\varepsilon,\rho_{j}+\varepsilon)
\]
when $N\ge N_{\varepsilon}$. When $\varepsilon>0$ is small enough,
there are (counting multiplicities) exactly $\mathrm{dim}(\ker(F(\rho_{j}))$
zeros of $f_{N}$ in $(\rho_{j}-\varepsilon,\rho_{j}+\varepsilon)$. 
\end{proof}

Later, we  apply this lemma to $\gamma=\mathbb R$ and $\gamma=\mathbb T$,
in order to control the behaviour of functions related to the resolvent  $R_N$.\\

Next, we collect some facts about matrix functions and maps on matrix spaces
which commute with the operation of conjugation by unitary matrices. First, an analog of the
Nevanlinna representation for matrix-valued functions \cite[Section 5]{gesz-tse}. Let
$m>0$ be fixed and let $F\colon\mathbb C\setminus[-m,m]\to M_N(\mathbb C)$ be an analytic
function. Assume that $\Im F(z)=(F(z)-F(z)^{*})/2i$ is nonnegative definite
for $z\in\mathbb{C}^{+}$, and $F(x)=F(x)^{*}$ for $x\in\mathbb{R}\setminus[-m,m]$.
Then $F$ can be represented as
\begin{equation}\label{4.5}
F(z)=A+Bz-\int_{[-m,m]}\frac{d\rho(t)}{z-t},\quad z\in\mathbb{C}\setminus[-m,m],
\end{equation}
where $A$ is a selfadjoint matrix, $B\ge0$, and $\rho$ is a  measure
with values in $M_{N}(\mathbb{C})$ such that $\rho(S)\ge0$
for every Borel set $S \subset\mathbb{R}$. Observe that
\[
\rho(\mathbb{R})=\lim_{z\to\infty}z(A+Bz-F(z)).
\]
The norm of such a function can obviously be estimated as
\[
\|F(z)\|\le\|A\|+\|B\||z|+\frac{\|\rho(\mathbb{R})\|}{\Im z},\quad z\in\mathbb{C}^{+}.
\]
The specific situation we have in mind is as follows. Let $X$ be
a random selfadjoint matrix in $M_{N}(\mathbb{C})$
such that $\|X\| \leq m$ almost surely. Pick $b\in M_N(\mathbb C)$ such that
$\Im b:=(b-b^*)/2i>0$ (that is, $\Im b$ is positive definite).
The matrix $\mathbb{E}[(\Re b+z\Im b-X)^{-1}]$ is analytic in $z$, it is invertible for
$$
z\in\mathbb{C}\setminus[-(m+\|\Re b\|)\|(\Im b)^{-1}\|,(m+\|\Re b\|)\|(\Im b)^{-1}\|],
$$ 
and it is selfadjoint for $$
z\in\mathbb{R}\setminus[-(m+\|\Re b\|)\|(\Im b)^{-1}\|,(m+\|\Re b\|)\|(\Im b)^{-1}\|].
$$
Moreover, 
$$
\Im\mathbb{E}[(\Re b+z\Im b-X)^{-1}] < 0,\quad z\in\mathbb{C}^{+}.
$$
It follows that the function $F(z)=(\mathbb{E}[(\Re b+z\Im b-X)^{-1}])^{-1}$
satisfies the properties required for it to have a representation of the form
\eqref{4.5}. The matrices $A,B,$ and $\rho(\mathbb{R})$ are easily determined. Indeed, we have
\begin{eqnarray*}
(\mathbb{E}[(\Im b-\varepsilon (X-\Re b))^{-1}])^{-1}&=&\Im b-\varepsilon\mathbb{E}[X-\Re b]-\varepsilon^{2}\left[\mathbb{E}[X(\Im b)^{-1}X]\right.\\
& & \mbox{}-\left.\mathbb{E}[X](\Im b)^{-1}\mathbb E[X]\right]+O(\varepsilon^{3})
\end{eqnarray*}
as $\varepsilon\to0$. Substituting $\varepsilon=1/z$, we obtain
\begin{equation}\label{4.6}
F(z)=z\Im b-\mathbb{E}[X]+\Re b -
\frac{\mathbb{E}[(X-\mathbb{E}[X])(\Im b)^{-1}(X-\mathbb{E}[X])]}{z}+O\left(\frac{1}{z^{2}}\right)
\end{equation}
as $z\to\infty$. This yields 
\begin{equation}\label{4.7}
A=-\mathbb{E}[X]+\Re b,\quad \textbf{$B=\Im b$},\quad
\text{and}\quad \rho(\mathbb{R})=\mathbb{E}[(X-\mathbb{E}[X])(\Im b)^{-1}(X-\mathbb{E}[X])].
\end{equation}
We are mostly, but not exclusively, interested in the case $\Im b=I_N$.

These observations apply to the variables $X_N$ from our models. We begin with 
$X_N=A_{N}+U_{N}^{*}B_{N}U_{N}$, 
where $(A_{N})$ and $(B_{N})$ are any sequences of deterministic real diagonal matrices of size
$N\times N$ with uniformly bounded norms and limiting distributions $\mu$ and $ \nu$,
respectively. 
As before, $R_{N}(z)=(zI_{N}-X_{N})^{-1}$. More generally, if 
$b\in M_{N}(\mathbb{C})$
satisfies 
$\Im b>0$, then $R_{N}(b)=(b-X_{N})^{-1}$.
\begin{lem}\label{l}
The function $b\mapsto\mathbb{E}[R_{N}(b)]$ takes values in ${\rm GL}(N)$ whenever $\Im b>0$.
Moreover,
$$
\Im\mathbb{E}[R_{N}(b)]^{-1}\ge\Im b\text{ and }\|\mathbb E[R_{N}(b)]^{-1}\|\le\|b\|+C_1+{4C_2}{\|(\Im b)^{-1}\|},\quad\Im b>0,
$$
 where $C_{1}=\sup_{N}(\Vert A_{N}\Vert+\Vert B_{N}\Vert)$,
and $C_{2}=\sup_{N}({\rm tr}_{N}(B_{N}^{2})-[{\rm tr}_{N}(B_{N})]^{2})$. 
In particular,  
\begin{equation}
\Im\mathbb{E}[R_{N}(z)]^{-1}\ge\Im zI_N\text{ and }\Vert\mathbb{E}[R_{N}(z)]^{-1}\Vert\le|z|+C_{1}+\frac{4C_{2}}{|\Im z|},\quad z\in\mathbb{C}^+.\label{borne+}
\end{equation}
\end{lem}

\begin{proof}
The relation $\Im\mathbb{E}[R_{N}(b)]^{-1}\ge\Im b$ follows from \cite[Remark 2.5]{BPV12}.
The second inequality  follows immediately from the observations preceding the
lemma, and from the fact that for any deterministic matrix $Z$,
\begin{eqnarray*}
\lefteqn{\mathbb{E}[U_{N}^{*}B_{N}U_{N}ZU_{N}^{*}B_{N}U_{N}]-\mathbb{E}[U_{N}^{*}B_{N}U_{N}]Z\mathbb{E}[U_{N}^{*}B_{N}U_{N}]}\\
&=({\rm tr}_{N}(B_{N}^{2})-[{\rm tr}_{N}(B_{N})]^{2})\left(\frac{N^2}{N^2-1}\text{tr}_N(Z)I_N-
\frac{1}{N^2-1}Z\right).\qedhere\end{eqnarray*}
\end{proof}

In some situations it is convenient to see how $\mathbb E[(zI_N-X_N)^{-1}]$ depends on $A_N$; recall that $X_N=A_N+U_N^*B_NU_N$. This is achieved to some extent by the following lemma (see also \cite{Kargin11} ).
\begin{lem}\label{bicommutant} 
Fix a matrix $B_N\in M_{N}(\mathbb{C})$. Let $b\in M_N(\mathbb C)$ be such that
$b-U^{*}B_{N}U$ is invertible for every  $U\in{\rm U}(N)$,  consider the random matrix
$R(b)=(b-U_{N}^{*}B_{N}U_{N})^{-1}$ and its expected value $G(b)=\mathbb{E}[(b-U_{N}^{*}B_{N}U_{N})^{-1}]$.
Then{\rm:}
\begin{enumerate}
\item For every $Y\in M_{N}(\mathbb{C})$ we have 
\begin{equation}
G(b)Y-YG(b)=
\mathbb{E}[R(b)(Yb-bY)R(b)].\label{eq:first-commutation}
\end{equation}
If $G(b)$ is invertible, we also have
\begin{eqnarray}
\lefteqn{Y(G(b)^{-1}-b)-(G(b)^{-1}-b)Y=}\label{eq:second-commutation}\\
 & &G(b)^{-1}\mathbb{E}[(R(b)-G(b))(Yb-bY)(R(b)-G(b))]G(b)^{-1}.\nonumber
\end{eqnarray}

\item $G(b)\in\{b\}''$, where $\{b\}''$ denotes the double commutant of $b$
in $M_{N}(\mathbb{C})$.
\end{enumerate}
\end{lem}
\begin{rem}
The conclusion of item (2) of the above lemma applies to any complex differentiable map $f$ defined
on an open set in $M_N(\mathbb C)$ with the property that $f(V^*bV)=V^*f(b)V$ for all $V\in{\rm U}(N)$.
\end{rem}
\begin{proof}
The analytic function $H(Y)=\mathbb{E}[(b-e^{iY}U_{N}^{*}B_{N}U_{N}e^{-iY})^{-1}]$
is defined in an open set containing the selfadjoint matrices. Moreover,
left invariance of the Haar measure implies that $H$ is constant
on the selfadjoint matrices. Since the selfadjoint matrices form a
uniqueness set for analytic functions, we deduce that $H$ is constant
in a neighborhood of the selfadjoint matrices. In particular, given
$Y\in M_{N}(\mathbb{C})$, the function
\[
\mathbb{E}[(b-e^{\varepsilon Y}U_{N}^{*}B_{N}U_{N}e^{-\varepsilon Y})^{-1}]
\]
does not depend on $\varepsilon$ for small $\varepsilon\in\mathbb{C}$.
Differentiation at $\varepsilon=0$ yields the identity 
\[
\mathbb{E}[R(b)(U_{N}^{*}B_{N}U_{N}Y-YU_{N}^{*}B_{N}U_{N})R(b)]=0.
\]
Using now the fact that $R(b)U_{N}^{*}B_{N}U_{N}=-I_N+R(b)b$ and $U_{N}^{*}B_{N}U_{N}R(b)=-I_N+bR(b)$, we obtain
\[
\mathbb{E}[-YR(b)+R(b)bYR(b)+R(b)Y-R(b)YbR(b)]=0
\]
which is equivalent to \eqref{eq:first-commutation} because $\mathbb{E}[R(b)Y]=G(b)Y$
and $\mathbb{E}[YR(b)]=YG(b)$.

To prove the second identity in (1), observe that
\[
\mathbb{E}[R(b)(Yb-bY)R(b)]=\mathbb{E}[(R(b)-G(b))(Yb-bY)(R(b)-G(b))]+G(b)(Yb-bY)G(b)
\]
so that \eqref{eq:first-commutation} implies
\[
G(b)Y-YG(b)-G(b)(Yb-bY)G(b)=\mathbb{E}[(R(b)-G(b))(Yb-bY)(R(b)-G(b))].
\]
Relation \eqref{eq:second-commutation}  is now obtained multiplying
this relation by $G(b)^{-1}$ on both sides.

To verify (2), we need to show that $G(b)$ commutes with any matrix
$Y\in\{b\}'$. This follows immediately from (\ref{eq:first-commutation}).
\end{proof}

The preceding lemma shows that $G(b)$ must be of the form $u(b)$ for some rational function 
$u$ of a complex variable, and \eqref{eq:second-commutation} allows us to show that in
fact $G(b)^{-1}$ is close to a matrix of the form $b+wI_{N}$ when the
variance of $R(b)$ is small. This follows from the next result.
\begin{lem}\label{lem:W(T)}
Assume that $\varepsilon>0$, and $T\in M_{N}(\mathbb{C})$
satisfies the inequality
\[
|k^{*}(TY-YT)h|\le\varepsilon\|Y\|
\]
for every rank one matrix $Y\in M_{N}(\mathbb{C})$ and all unit vectors
$h,k\in\mathbb{C}^{N}$. Then for any $w$ in the numerical range 
$W(T)=\{h^{*}Th:\|h\|=1\}$, we have $\|T-wI_{N}\|\le2\varepsilon$.
\end{lem}
\begin{proof}
Given two unit (column) vectors $h,k\in\mathbb{C}^{N}$, consider
the---necessarily rank one---matrix $Y=kh^{*}\in M_{N}(\mathbb{C})$. The hypothesis implies
that
\[
|k^{*}Tk-h^{*}Th|=|k^{*}(TY-YT)h|\le\varepsilon.
\]
We deduce that the numerical range $W(T)=\{h^{*}Th:\|h\|=1\}$ has diameter at most 
$\varepsilon$, and therefore there $W(T-wI_{N})\subset\{\lambda\in\mathbb{C}\colon|\lambda|
\le\varepsilon\}$ for any $w\in W(T)$. Thus any $w\in W(T)$ satisfies the conclusion because the 
norm of an operator is at most twice its numerical radius (see \cite[Theorem 1.3-1]{numran}).
\end{proof}

A further property of eigenvectors of Hermitian matrices which are close in norm to each other
appears in the analysis of the behaviour of the eigenvectors of our matrix models. The 
following lemma appears already in \cite{Capitaine11}; we offer a proof for the reader's convenience. 
\begin{lem}\label{eigenspaces}
Let $X$ and $X_0$  be  Hermitian $N\times N$ matrices.  Assume that $\alpha,\beta,\delta\in\mathbb R$ are such that
$\alpha<\beta$, $\delta>0$, and neither $X$  nor $X_0$ has any eigenvalues in $[\alpha-\delta,\alpha]\cup
[\beta,\beta+\delta]$. Then, $$
\|E_{X}((\alpha,\beta))-E_{X_0}((\alpha,\beta))\|<\frac{4(\beta-\alpha+2\delta)}{\pi\delta^2}\|X-X_0\|.
$$
In particular, for any unit vector $\xi\in E_{X_0}((\alpha,\beta))(\mathbb C^N)$,
$$
\|(I_N-E_{X}((\alpha,\beta)))\xi\|_2<\frac{4(\beta-\alpha+2\delta)}{\pi\delta^2}\|X-X_0\|.
$$
\end{lem}
\begin{proof}
Consider  the rectangle $\gamma$ having as corners the complex points 
$\alpha- (1\pm i)\delta/2$ and $\beta+(1\pm i)\delta/2$. By assumptions, we have  $\sigma(X)\cap([\alpha-\delta,\alpha]\cup
[\beta,\beta+\delta])=\varnothing$ and  $\sigma(X_0)\cap([\alpha-\delta,\alpha]\cup
[\beta,\beta+\delta])=\varnothing$. Thus the spectral projections can be obtained by analytic functional calculus:
$$
E_{X}((\alpha,\beta))-E_{X_0}((\alpha,\beta))=\frac{1}{2\pi i}\int_\gamma\left[
(\lambda-X)^{-1}-(\lambda-X_0)^{-1}\right]\,d\lambda.
$$
An application of the resolvent equation and elementary norm estimates yield
\begin{eqnarray*}
\|E_{X}((\alpha,\beta))-E_{X_0}((\alpha,\beta))\|&=&\frac{1}{2\pi}\left\|\int_\gamma
(\lambda-X)^{-1}(X_0-X)(\lambda-X_0)^{-1}\,d\lambda\right\|\\
&\leq& \frac{1}{2\pi}\int_\gamma\|(\lambda-X)^{-1}(X_0-X)(\lambda-X_0)^{-1}\|\,
d\lambda\\
&\leq&(\beta-\alpha+2\delta)\\
& & \mbox{}\times\frac{\|X-X_0\|}{\pi}\sup_{\lambda\in\gamma}\frac{1}{\|\lambda-X\|}
\sup_{\lambda\in\gamma}\frac{1}{\|\lambda-X_0\|}\\
&<& \frac{4(\beta-\alpha+2\delta)}{\pi\delta^2}\|X-X_0\|.
\end{eqnarray*}
The lemma follows.
\end{proof}

For the following concentration of measure result it is convenient to identify $\mathbb{C}^N$ with the 
subspace of $\mathbb{C}^{N+1}$ consisting of all vectors whose last component is zero.  Similarly, 
$M_N(\mathbb{C})$ is identified with those matrices in $M_{N+1}(\mathbb{C})$ whose last column and 
row are zero.  We use the notation $\mathbb V$ for variance.

\begin{lem}\label{concentration}
Fix a positive integer $r$, a projection $P$ of rank $r$, and a scalar $z\in\mathbb{C}\setminus\mathbb{R}$. Then{\rm:}
\begin{enumerate}
\item [(i)] $\lim_{N\to\infty}(PR_{N}(z)\,P^*-P\mathbb{E}[R_{N}(z)]\,P^*)=0$
almost surely. 
\item [(ii)]Given unit vectors $h,k\in\mathbb{C}^{N}$, $\mathbb{V}(k^{*}R_{N}(z)h)\leq C/[N|\Im z|^{4}]$.
\end{enumerate}
\end{lem}

\begin{proof}
Assertion  
(i) is equivalent to the statement that, given unit vectors $h,k\in\mathbb{C}^{N}$
\begin{equation}
\lim_{N\to\infty}k^{*}(R_{N}(z)-\mathbb{E}[R_{N}(z)])h=0\label{eq:55}
\end{equation}
 almost surely. The random variable $k^{*}R_{N}(z)h$
is a  Lipschitz function on the unitary group ${\rm U}({N})$ with Lipschitz
constant ${C}/{|\Im z|^{2}}$. An application of \cite[Corollary 4.4.28]{AGZ10}
yields the inequality 
\[
\mathbb{P}\left(\vert k^{*}\left(R_{N}(z)-\mathbb{E}[R_{N}(z)]\right)h\vert>\frac{\varepsilon}{N^{\frac{1}{2}-\alpha}}\right)\leq2\exp\left(-CN^{2\alpha}\vert\Im z\vert^{4}\varepsilon^{2}\right).
\]
for any $\alpha\in(0,1/2)$, and \eqref{eq:55} follows by an application
of the Borel-Cantelli lemma. To prove (ii), apply the same inequality in the usual formula
$\mathbb{E}[X]=\int_{0}^{+\infty}\mathbb{P}(X>t)\, dt$ for a positive
random variable $X$.
\end{proof}

In the following result the coefficient $t^{4}$
can be replaced by $t^{2}$ if we estimate the operator norm of a matrix by its Hilbert-Schmidt norm.

\begin{cor}
\label{cor:variance-estimate}Fix a positive integer $t$, matrices $Y,Z$ of rank at most $t$, and a scalar 
$z\in\mathbb{C}\setminus\mathbb{R}$. Then{\rm:}
\[
\mathbb{E}[\|Y(R_{N}(z)-\mathbb{E}[R_{N}(z)])Z\|^{2}]\le Ct^{4}\|Y\|^{2}\|Z\|^{2}/[N|\Im z|^{4}].
\]
\end{cor}
\begin{proof}
Choose orthonormal vectors $h_{1},\dots,h_{t}$ whose span contains
the range of $Z$ and orthonormal vectors $k_{1},\dots,k_{t}$ whose
span contains the range of $Y^{*}$. The corollary follows from the
inequality
\[
\|Y(R_{N}(z)-\mathbb{E}[R_{N}(z)])Z\|\le\sum_{i,j=1}^{t}\|Y\|\|Z\|\vert k_{j}^{*}\left(R_{N}(z)-\mathbb{E}[R_{N}(z)]\right)h_{i}\vert
\]
and part (ii) of the preceding lemma.
\end{proof}

\begin{rem}\label{4.13}
We note for further use that Lemma \ref{concentration} and Corollary \ref{cor:variance-estimate}
apply to the resolvent of any selfadjoint polynomial in $m+1$ noncommuting variables 
$P(A_N^{(1)},\dots,A_N^{(m)},U_N^*B_NU_N)$
as long as the norms of $A_N^{(j)}$ and $B_N$ are uniformly bounded in $N$.
\end{rem}

\section{Proofs of the main results}
The three subsections below provide a parallel treatment of the three models under consideration.

\subsection{The additive model $X_N=A_N+U_N^*B_NU_N$\label{PLUS}}
We use the notation from Subsection \ref{subsec:+perturb}.
Fix $\alpha\in\text{supp}(\mu)$
and $\beta\in\text{supp}(\nu)$.
Due to the left and right invariance of the Haar measure on ${\rm U}(N)$ we may, and do, assume without loss of 
generality that both $A_{N}$ and $B_{N}$ are diagonal matrices. More precisely, we let $A_N$ be the diagonal matrix
$$
A_N=\text{ Diag}
(\theta_1,\dots,\theta_p,\alpha_1^{(N)},\dots,\alpha_{N-p}^{(N)}),
$$
where  $\alpha_1^{(N)}\ge\cdots\ge\alpha_{N-p}^{(N)}$.
We also have $\theta_1\ge\cdots\ge\theta_p$, but no order relation is assumed between $\theta_i$ and $\alpha_j^{(N)}$ other than $\theta_i\ne\alpha_j^{(N)}$.
For $N\ge p$, we write  $A_{N}=A_{N}'+A_{N}'',$ 
where 
\[
A_{N}'=\text{Diag}(\underbrace{\alpha,\ldots,\alpha}_p,\alpha_{1}^{(N)},\ldots,\alpha_{N-p}^{(N)}),
\]
 and
\[
A_{N}''= {\text{Diag}(\theta_{1}-\alpha,\dots\theta_{p}-\alpha,\underbrace{0,\ldots,0}_{N-p})}.
\]
We have
$A_{N}''=P_N^*\Theta P_N$, where $P_N$ is the $p\times N$ matrix representing the usual projection 
$\mathbb{C}^N\to\mathbb{C}^p$ onto the first $p$ coordinates, and 
\[
\Theta= {\text{Diag}(\theta_{1}-\alpha,\dots,\theta_{p}-\alpha)}.
\]
The operator $P_N$ is precisely $E_{A_N}(\{\theta_1,\dots,\theta_p\})$ corestricted to its range.
Similarly, $B_{N}=B_{N}'+B_{N}'',$ where 
\[
B_{N}'=\text{Diag}(\underbrace{\beta,\ldots,\beta}_q,\beta_{1}^{(N)},\ldots,\beta_{N-q}^{(N)}),
\]

\[
B_{N}''=\text{Diag} ({\tau_{1}-\beta,\dots,\tau_{q}-\beta,\underbrace{0,\ldots,0}_{N-q})}.=Q_N^*TQ_N,
\]
\[
{T}={\text{Diag}(\tau_{1}-\beta,\dots,\tau_{q}-\beta)},
\]
and $Q_N$ is the $q\times N$ matrix representing the usual projection $\mathbb{C}^N\to\mathbb{C}^q$.

\subsubsection{Reduction to the almost sure convergence of a $p\times p$ matrix\label{551}}
Here we explain how to reduce, in the spirit of \cite{BGRao09},
the problem of locating outliers of $A_{N}+U_{N}^{*}B_{N}'U_{N}$ to a  convergence problem
for a random matrix of fixed size $p\times p$. The matrices $A_{N}'$ and $B_{N}'$ have no spikes, 
and therefore \cite[Corollary 3.1]{ColMal11}
applies to the matrix $X'_N=A_{N}'+U_{N}^{*}B_{N}'U_{N}$. Recall that $K=\text{supp}(\mu\boxplus\nu)$.
We reformulate this result as follows: there exist positive random variables $(\delta_N)_{N\in\mathbb{N}}$, such that
\[
\sigma(A_{N}'+U_{N}^{*}B_{N}'U_{N})\subseteq K_{\delta_{N}},\quad N\in\mathbb{N},
\] 
and $\lim_{N\to\infty}\delta_N=0$ almost surely.
Given $z\in\mathbb{C}\setminus K_{\delta_{N}}$, we have
\[
zI_{N}-(A_{N}+U_{N}^{*}B_{N}'U_{N})=zI_{N}-X'_{N}-A_{N}''=(zI_{N}-X'_{N})(I_N-(zI_{N}-X'_{N})^{-1}A''_{N}),
\]
and therefore 
\[
\det(zI_{N}-(A_{N}+U_{N}^{*}B_{N}'U_{N}))=\det(zI_{N}-X_{N}')\det(I_{N}-\left(zI_{N}-X_{N}'\right)^{-1}\,P_N^*\Theta P_N).
\]
 Using the fact that $\det(I-XY)=\det(I-YX)$ when  $XY$
and $YX$ are square, we obtain $\det(I_{N}-\left(zI_{N}-X_{N}'\right)^{-1}\,P_N^*\Theta P_N)
=\det(I_{p}-P_N\left(zI_{N}-X_{N}'\right)^{-1}\,P^*_N\Theta),$ so
\[
\det(zI_{N}-(A_{N}+U_{N}^{*}B_{N}'U_{N}))=\det(zI_{N}-X_{N}')\det(I_{p}-P_N\left(zI_{N}-X_{N}'\right)^{-1}\,P_N^*\Theta).
\]
 We conclude that  the eigenvalues of $A_{N}+U_{N}^{*}B_{N}'U_{N}$
outside $K_{\delta_{N}}$ are precisely the zeros of the function $\det(F_N(z))$, where
\begin{equation}
F_{N}(z)=I_{p}-P_N\left(zI_{N}-X_{N}'\right)^{-1}\,P_N^*\Theta\label{MN}
\end{equation}
 in that open set. This is a random analytic function defined on $\overline{\mathbb{C}}\setminus K_{\delta_{N}}$,
with values in $M_{p}(\mathbb{C})$. We argue next that the sequence  $\{F_N(z)\}_N$ converges almost 
surely to the deterministic diagonal matrix function
\[
F(z)=\mathrm{Diag}\left(1-\frac{\theta_{1}-\alpha}{\omega_1(z)-\alpha},
\dots,1-\frac{\theta_{p}-\alpha}{\omega_1(z)-\alpha}\right),
\]
where $\omega_1$ is the subordination function from (\ref{subord1}).

\subsubsection{Convergence of $F_{N}$ }
We begin with a somewhat more general result. 
\begin{prop}
\label{estimationinmean}
Fix a positive integer $p$, and let $C_N$ and $D_N$ be deterministic real diagonal $N \times N$ matrices 
whose norms are uniformly bounded and such that the limits
$$\eta_i=\lim_{N\to\infty}(C_N)_{ii}$$
exist for $i=1,2,\dots,p. $ Suppose
that the empirical eigenvalue distributions of $C_N$ and $D_N$ converge weakly to $\mu$ and $\nu$, respectively.
Then the resolvent $${R}_N(z)=(zI_N-C_N-U_N^*D_NU_N)^{-1},\quad z\in\mathbb{C}\setminus\mathbb{R},$$ satisfies
\begin{equation}
\lim_{N\to\infty}P_N\mathbb{E}[{R}_N(z)]P_N^*=\mathrm{Diag}\left(
\frac{1}{\omega_1(z)-\displaystyle\eta_{1}},\ldots,
\frac{1}{\displaystyle\omega_1(z)-\eta_{p}}\right),\label{secondapproximation}
\end{equation}
where $\omega_1$ is the subordination function from \eqref{subord1}.
\end{prop}

\begin{proof}
Since all functions involved satisfy $f(\overline{z})=f(z)^*$, it suffices to consider the case of
$z\in\mathbb C^+$. Fix such a scalar $z$ and apply Lemma \ref{bicommutant}(2)  to 
$b=zI_N-C_N$ to conclude that the $N\times N$ matrix $\mathbb{E}[{R}_{N}(z)]$ is
diagonal. Set 
\begin{equation}\label{approxsubordfcn}
\omega_{N,i}(z)=\frac{1}{\mathbb{E}[{R}_{N}(z)]_{ii}}+(C_N)_{ii},\quad 1\leq i\leq p.
\end{equation}
and observe that $\Im \omega_{N,i}(z)\geq \Im z$ for $z\in \mathbb{C}^+$ by Lemma \ref{l}.
We proceed to  show that this function satisfies an approximate subordination relation. We state this separately for future reference.
\begin{lem}
\label{approximatesubordination}
We have
\[
\lim_{N\to\infty}\Vert\mathbb{E}[{R}_{N}(z)]-(\omega_{N,i}(z)I_{N}-C_N)^{-1}\Vert=0,\quad z\in\mathbb{C}^{+},1\leq i\leq p.
\]
\end{lem}

\begin{proof}
The existence of
\[
\Omega_{N}(z)=\mathbb{E}[{R}_{N}(z)]^{-1}+C_N,\quad z\in\mathbb{C}^{+},
\]
is guaranteed by Lemma \ref{l}. We apply now Lemma \ref{bicommutant} with  $zI_N-C_N$ in place of $b$ and
$D_N$ in place of $B_N$, so $\mathbb{E}[{R}_{N}(z)]=G(zI_N-C_N)$.
Relation (\ref{eq:second-commutation}) 
shows that
\begin{eqnarray*} \lefteqn{Y{\Omega_{N}(z)-\Omega_{N}(z)Y}=
\mathbb{E}[{R}_{N}(z)]^{-1}
\mathbb{E}\big[({R}_{N}(z)-\mathbb{E}[{R}_{N}(z)])}
\\
&\quad &\times(YC_N-C_NY)({R}_{N}(z)-\mathbb{E}[{R}_{N}(z)])\big]\mathbb{E}[{R}_{N}(z)]^{-1},\quad Y\in M_{N}(\mathbb{C}).
\end{eqnarray*}
Suppose that $Y$ has rank
one and $h,k\in\mathbb{C}^{N}$ are unit vectors. In this case, there exist
rank one projections $p_{1},p_{2}$ and rank two projections $q_{1},q_{2}$ (depending on $z$)
such that
\begin{eqnarray*}
\lefteqn{k^{*}(Y\Omega_{N}(z)-\Omega_{N}(z)Y)h=
k^{*}\mathbb{E}[{R}_{N}(z)]^{-1}}\\
& & \mbox{}\times\mathbb{E}[p_{1}({R}_{N}(z)-\mathbb{E}[{R}_{N}(z)])q_{1}(YC_N-C_NY)q_2({R}_{N}
(z)-\mathbb{E}[{R}_{N}(z)])p_{2}]\\
& & \mbox{}\times\mathbb{E}[{R}_{N}(z)]^{-1}h.
\end{eqnarray*}
Indeed, the third and first factors in the product above have rank one, while the rank of $YC_N-
C_NY$ is at most two.
We deduce that
\begin{eqnarray*}
|k^{*}(\Omega_N(z)Y-Y\Omega_N(z))h| & \le & \|\mathbb{E}[{R}_{N}(z)]^{-1}\|^2\|YC_N-C_NY\|\\
 &  & \times \mathbb{E}[\|p_{1}({R}_{N}(z)-\mathbb{E}[{R}_{N}(z)])q_{1}\|^{2}]^{1/2}\\
 &  & \times \mathbb{E}[\|q_{2}({R}_{N}(z)-\mathbb{E}[{R}_{N}(z)])p_{2}\|^{2}]^{1/2}.
\end{eqnarray*}
We use now the estimates from Lemma \ref{l} and  Corollary \ref{cor:variance-estimate} along with the 
inequality $\|YC_N-C_NY\|\le 2\|C_N\|\|Y\|$  to obtain a constant $C>0$ (independent of $N$ and $z$) 
such that
\[
|k^{*}(\Omega_{N}(z)Y-Y\Omega_{N}(z))h|\le C\frac{(|z|+1+(1/\Im z))^{2}}{N|\Im z|^{4}}\|Y\|.
\]
The number $\omega_{N,i}(z)$ is precisely the $(i,i)$ entry of the matrix $\Omega_{N}(z)$ and thus it 
belongs to the numerical range $W(\Omega_{N}(z))$; indeed it equals
 $e_i^*\Omega_{N}(z)e_i$, where $e_1,\dots,e_N$ is the basis in which
$C_N$ is diagonal. Lemma \ref{lem:W(T)} yields the estimate
\[
\|\Omega_{N}(z)-\omega_{N,i}(z)I_{N}\|\le2C\frac{(|z|+1+(1/\Im z))^{2}}{N|\Im z|^{4}},
\]
which gives the desired result as $N\to\infty$.
\end{proof}
Fix now $i\in\{1,\ldots,p\}$ and observe that the family of functions $(\omega_{N,i})_N$ is normal on
$\mathbb{C}^+$. Voiculescu's asymptotic freeness result shows that
$$\lim_{N\to\infty}\text{tr}_N(\mathbb{E}[{R}_{N}(z)])=G_{\mu\boxplus\nu}(z)
=G_\nu(\omega_1(z)),$$
so Lemma \ref{approximatesubordination}  implies that
$\omega_{N,i}$ converges  uniformly on compact subsets of $\mathbb{C}^+$.
This, together with a second application of Lemma \ref{approximatesubordination}, implies \eqref{secondapproximation}
and concludes the proof of Proposition 
\ref{estimationinmean}.
\end{proof}
The convergence result for the functions $F_N$ also uses a normal family argument, more specifically the fact that
a normal sequence converges uniformly on compact sets if it converges pointwise on a set with an accumulation point which belongs to the domain. 

\begin{prop}
\label{uniformconvergence}Almost surely, the sequence
$\{F_N\}_{N}$ converges uniformly on compact subsets of $\mathbb{C}\setminus K$
to the analytic function $F$ defined  by
\begin{equation}\label{F(z)}
F(z)=\mathrm{Diag}\left(1-\frac{\theta_{1}-\alpha}{\omega_1(z)-\alpha},
\ldots,1-\frac{\theta_{p}-\alpha}{\omega_1(z)-\alpha}\right),\quad z\in\mathbb C\setminus K.
\end{equation}
\end{prop}

\begin{proof}
Lemma \ref{lem:extension}, part (b), and the hypothesis 
on $\alpha$, show that the function $z\mapsto{1}/({\omega_1(z)-\alpha})$ is analytic on 
$\overline{\mathbb C}\setminus K$. Define 
\[
\mathcal{D}=\{z\in\mathbb{C}\setminus K\colon\Re z\in\mathbb{Q},\Im z\in\mathbb{Q}\setminus\{0\}\}.
\]
The first $p$ diagonal elements of $A_{N}'$ are all equal to $\alpha$ and thus
 Lemma \ref{concentration}(i) and \eqref{secondapproximation} 
(applied to $C_N=A_{N}'$ and $D_N= B_{N}'$) show that given $z\in\mathcal{D}$,
the sequence $P_N(zI_{N}-X_{N}')^{-1}P_N^*$ converges almost surely to $(1/(\omega_1 (z)-\alpha))I_p.$
Moreover, by \cite{ColMal11},  these functions are almost surely uniformly bounded on any compact subset of
$\mathbb{C}\setminus K$. Uniform boundedness on some neighbourhood of infinity in
$\mathbb C\cup\{\infty\}$ is automatic.
We deduce that, almost surely, this sequence converges
uniformly on compact subsets of $\overline{\mathbb{C}}\setminus K$
to the function $({1}/({\omega_1 -\alpha})) I_{p}$. The Proposition follows immediately
from these facts and \eqref{MN}.
\end{proof}

\subsubsection{Proofs of the main results for the additive model.}
\begin{proof}
[Proof of Theorem {\rm\ref{Main+}}, parts {\rm (1)} and {\rm(2)}---eigenvalue behaviour.]
 We proceed in two steps.
\newline\noindent{\bf Step 1.}  We investigate first the case in which $q=0$, that is, $B_N=B_N'$ has no spikes.
Equivalently, we prove our result for the simpler model $A_N+U_N^*B'_NU_N$.
We can work on the almost
sure event, whose existence is guaranteed by Proposition \ref{uniformconvergence},
on which
\begin{itemize}
\item $\lim_{N\to\infty}\delta_N=0$,
\item $\sigma(A_{N}'+U_{N}^{*}B_{N}'U_{N})\subseteq K_{\delta_{N}}$
for all $N$, and,
\item 
the sequence $(F_{N})_{N\geq p}$ converges to the function $F$ defined by \eqref{F(z)}
uniformly on the compact subsets of $\overline{\mathbb{C}}\setminus K$.
\end{itemize}

 We  apply Lemma \ref{alt-Benaych-Rao} on this event, with $\gamma=\mathbb R$. We first argue that
the hypotheses of that lemma are satisfied. The values of $F$  are clearly diagonal matrices and $F(\infty 
)=I_p$. 
We show that the zeros of $(F(z))_{ii}$ are simple. Indeed, 
$$
(F'(z))_{ii}=\frac{\omega_1'(z)(\theta_i-\alpha)}{(\omega_1(z)-\alpha)^2},
$$
and the zeros of $\omega_1'$ are simple by the Julia-Carath\'eodory Theorem \cite[Chapter I, 
Exercises 6 and 7]{GarnettBook} because $\omega_1(\mathbb{C}^+)\subset\mathbb{C}^+$, 

Hypotheses (1) and (3) of  Lemma  \ref{alt-Benaych-Rao} follow from Proposition 
\ref{uniformconvergence}. To verify hypothesis (2) of Lemma \ref{alt-Benaych-Rao}, 
observe that if $F_N(z)$ is not invertible then $z$ is an eigenvalue of the selfadjoint matrix $A_N+U_N^*
B_N'U_N$, hence real.
There are arbitrarily small numbers  $\delta>0$ such that the boundary points of $K_{\delta}$ are 
not zeroes of $\det(F)$. When this condition is satisfied, Lemma \ref{alt-Benaych-Rao} yields precisely
the conclusion of Theorem \ref{Main+}(1)--(2), when $q=0$. Indeed, as explained
in Subsection \ref{551}, the eigenvalues of $A_{N}+U_{N}^{*}B_{N}'U_{N}$ 
in $\mathbb{C}\setminus K_{\delta}$
are exactly the zeroes of $\det(F_{N})$, and the set of points $z$
such that $F(z)$ is not invertible is precisely $\bigcup_{i=1}^p\omega_1^{-1}(\{
\theta_i\})$. This completes the first step.
\newline\noindent{\bf Step 2.} 
Suppose now that $q>0$ and use Step 1 above to obtain the existence of a sequence of positive random 
variables $(\delta_{N})_{N\in\mathbb N}$  such that $\lim_{N\to\infty}\delta_N=0$ almost surely and
$
\sigma(A_{N}+U_{N}^{*}B_{N}'U_{N})\subseteq  K^{\prime\prime}_{\delta_{N}},
$
where
$$K''=K\cup\bigcup_{i=1}^p\omega_1^{-1}(\{\theta_i\}).$$
We proceed as in Step 1 (switching  the roles of $A_N$ and $B_N$) in order to conclude that the 
eigenvalues of $X_N$ outside $K^{\prime\prime}_{\delta_{N}}$ are precisely the zeros of the function 
$\det(I_{q}-Q_N\left(zI_{N}-U_{N}A_{N}U_{N}^{*}-B_{N}'\right)^{-1}\,Q_N^*{ T})$ in that open set, where
\[
{T}=\text{Diag}(\tau_{1}-\beta,\dots,\tau_{q}-\beta).
\] 
and $Q_N$ is the orthogonal projection $\mathbb{C}^N\to\mathbb{C}^q$. 
Lemma \ref{alt-Benaych-Rao} is applied now to the functions 
$$
\widetilde{F}_N(z)=I_{s}-Q_N\left(zI_{N}-U_{N}A_{N}U_{N}^{*}-B_{N}'\right)^{-1}\,Q_N^*{T},\quad N\ge q,
$$
$$
\widetilde{F}(z)=\mathrm{Diag}\left(1-\frac{\tau_{1}-\beta}{\omega_2(z)-\beta},\dots,
1-\frac{\tau_{q}-\beta}{\omega_2(z)-\beta}\right),
$$
and the compact set $K'$. The convergence of $\{\widetilde{F}_N\}_N$ to $\widetilde{F}$ follows by
an adaptation of Proposition \ref{uniformconvergence}. This completes the proof of parts (1) and (2) of 
Theorem \ref{Main+} in the general case $q>0$, provided that $k=0$. By symmetry, we have also proved 
these assertions in case $\ell=0$. 
To prove part (2) in case $k\ne0\ne\ell$, we use a
perturbation argument. Fix $\rho\in\mathbb{R}\setminus K$ such that $\omega_1(\rho)=\theta_{i_0}$  
for some $i_0\in\{1,\dots,p\}$ and $\omega_2(\rho)=\tau_{j_0}$ for some $j_0\in\{1,\dots,q\}$, and fix 
$\varepsilon>0$ as in the statement of (2).  Choose $\delta\in(0,\varepsilon/3)$ so small that 
$\omega_1((\rho-3\delta,\rho+3\delta))$ contains no spikes $\theta_{i}\ne\theta_{i_0}$  and
$\omega_2((\rho-3\delta,\rho+3\delta))$ contains no spikes $\tau_{j}\ne\theta_{j_0}$. Since 
$\omega'_1$ is strictly increasing on $(\rho-3\delta,\rho+3\delta)$, we have 
$\omega_1(\rho+2\delta)=\theta_1+\eta$, with $\eta>0$. We use the already established part (2) of the 
theorem to conclude that, almost surely for large $N$, the perturbed matrix
$$X'_N=X_N+\eta E_{A_N}(\{\theta_i \})$$
has $k$ eigenvalues in $(\rho-\delta,\rho+\delta)$ and another $\ell$  eigenvalues in the disjoint interval 
$(\rho+\delta,\rho+3\delta)$.  An application of Lemma \ref{eigenspaces} for sufficiently small $\delta$ 
shows that $X_N$ has $k+\ell$ eigenvalues in $(\rho-\varepsilon,\rho+\varepsilon)$.
\end{proof}

\begin{proof}
[Proof of Theorem {\rm\ref{Main+}}, parts {\rm (3)} and {\rm (4)}---eigenspace behaviour.] 
We borrow heavily from the techniques of 
\cite{Capitaine11}. There are again two steps.
\newline\noindent{\bf Step A.} We prove parts (3) and (4) of  Theorem \ref{Main+} under the additional 
assumption that $\theta_1>\dots>\theta_p$, $\tau_1>\dots>\tau_q$, $k=1$, and $\ell=0$. Thus  
$\omega_1 (\rho)=\theta_{i_0}$ for some $i_0\in\{1,\dots,p\}$, and $\omega_2(\rho)\notin\{\tau_1,
\dots,\tau_q\}$. Assertion (3) of Theorem \ref{Main+} follows if the equalities
$$
\lim_{N\to\infty}
\left\|E_{A_N}(\{\theta_i\})E_{X_N}((\rho-\varepsilon,\rho+\varepsilon))E_{A_N}(\{\theta_i\})
-\frac{\delta_{i_0i}}{\omega'_1(\rho)}E_{A_N}(\{\theta_i\})\right\|=0,
$$
and
$$
\lim_{N\to\infty}\|E_{B_N}(\{\tau_j\})E_{X_N}((\rho-\varepsilon,\rho+\varepsilon))E_{B_N}(\{\tau_j\})\|=0,
$$
are shown to hold almost surely for all $i=1,\dots,p,j=1,\dots,q.$ The Hermitian matrices in these 
equations have rank one, so their norm is equal to the absolute value of their trace.  Thus we need to 
show that
\begin{equation}
\label{pr-a}
\lim_{N\to\infty}{\rm Tr}_N[E_{A_N}(\{\theta_i\})E_{X_N}((\rho-\varepsilon,\rho+\varepsilon))]=\frac{\delta_{i_0i}}{\omega_1'(\rho)},\quad i=1,\dots,p,
\end{equation}
and
\begin{equation}
\label{pr-b}
\lim_{N\to\infty}{\rm Tr}_N[E_{B_N}(\{\tau_j\})E_{X_N}((\rho-\varepsilon,\rho+\varepsilon))]=0,\quad j=1,\dots,q,
\end{equation}
almost surely.
It is useful to write the random variable in \eqref{pr-a} in terms of  functional calculus with \emph{continuous} rather than indicator functions. Choose $\delta>0$ so small that each interval $[\theta_i-\delta,\theta_i+\delta]$ contains exactly one point of $\sigma(A_N)$ (namely, $\theta_i$) for $i=1,\dots,p$ and for large $N$. For each $i=1,\dots,p$, choose a function $f_i\in\mathcal{D}(\mathbb{R})$ with support in $[\theta_i-\delta,\theta_i+\delta]$ such that $0\le f_i\le 1$ and $f_i(\theta_i)=1$. Also choose a function $h\in\mathcal{D}(\mathbb{R})$ with 
support in $[\rho -\varepsilon,\rho+\varepsilon]$ such that $0\le h\le1$ and $h(x)=1$ for $x\in[\rho -\varepsilon/2,\rho+\varepsilon/2]$. For sufficiently large $N$, we have $E_{A_N}(\{\theta_i\})=f_i(A_N)$.  Also, by the already established assertion (1) of the theorem, we have $E_{X_N}((\rho-\varepsilon,\rho+\varepsilon))=h(X_N)$ almost surely for $N$ sufficiently large. Thus we see that, almost surely for sufficiently large $N$,
\begin{equation}\label{egalite}
{\rm Tr}_N[E_{A_N}(\{\theta_i\})E_{X_N}((\rho-\varepsilon,\rho+\varepsilon))]=
 {\rm Tr}_N\left[h(X_N) f_i(A_N)\right],\quad i=1,\dots,p.
\end{equation}
To conclude the proof of \eqref{pr-a}, we obtain as in Lemma \ref{concentration} a concentration inequality for the right-hand side
of \eqref{egalite} and then we estimate the expected value. In the following argument we use the fact that a Lipschitz function on $\mathbb R$ remains Lipschitz, with the same constant, when considered as a function on the selfadjoint matrices endowed with the Hilbert-Schmidt norm. See \cite[Lemma A.2]{Capitaine11} for a simple proof of this fact, first observed in \cite{birman}.

\begin{lem}\label{inegconc}Fix $i\in\{1,\dots,p\}$, denote by $\gamma$ the Lipschitz constant of the function $h$, and set $C=\sup_N \Vert B_N \Vert$.
For $N$ sufficiently large, the random variable $Z_N={\rm Tr} \left[ h(X_N) f_i(A_N)\right]$
satisfies the concentration inequality
$$ \mathbb{P}\left( \vert Z_N-\mathbb {E}(Z_N) \vert > \eta \right) \leq 2 \exp\left(-\frac{\eta^2 N}{4C ^2\gamma^2}\right),\quad \eta>0.$$

\end{lem}
\begin{proof} The lemma follows from  \cite[Corollary 4.4.28]{AGZ10} once we establish that the Lipschitz constant of the function $$g(U)= {\rm Tr}_N\left[ h (A_N+U^*B_NU)f_n(A_N)\right],\quad U\in{\rm U}(N),$$
is at most $2C\gamma$. For any $U$ and $V$ in ${\rm U}(N)$ we have 
\begin{eqnarray*}
\vert g(U)-g(V) \vert &=&|{\rm Tr}_N\left[f_i(A_N)\left(h (A_N+U^*B_NU)
-h(A_N+V^*B_NV)\right)\right]|\\
&\leq& \|h (A_N+U^*B_NU)
-h (A_N+V^*B_NV)\|_2\\
&\leq&\gamma \Vert U^* B_N U -V^* B_N V \Vert_2 ,
\end{eqnarray*}
where we used the Cauchy-Schwarz inequality for the Hilbert-Schmidt norm and the fact that 
$\|f_i (A_N)\|_2\le1$. Since
$$\Vert U^* B_N U -V^* B_N V \Vert_2\le\Vert U^* B_N( U -V) \Vert_2+\|(U^*-V^*)B_NV\|_2\le2\|B_N\|\|U-V\|_2,$$
we conclude that
$
\vert g(U)-g(V) \vert \leq 2C\gamma \Vert U-V\Vert_2,
$
as desired.
\end{proof}

The above result, combined with the Borel-Cantelli lemma, yields immediately
\begin{equation*}
\label{concentre}
\lim_{N\to\infty}\left({\rm Tr}_N\left[ h (X_N) f_i(A_N)\right]-\mathbb{E}\left[ {\rm Tr}_N \left[ h (X_N) f_i(A_N)\right]\right]\right)=0,\quad i=1,\dots,p,
\end{equation*}
almost surely. We conclude the proof of \eqref{pr-a} by showing that
\begin{equation*}
\lim_{N\to\infty}\mathbb{E}\left[ {\rm Tr}_N \left[ h (X_N)f_i(A_N)\right]\right]=\frac{\delta_{i_0i}}{\omega'_1(\rho)}, \quad i=1,\dots,p.\end{equation*}
Lemma \ref{approxcauchy} with $C_N=f_i(A_N)$ allows us to rewrite this as
\begin{equation*}
\lim_{N\to\infty}\lim_{y\downarrow0}\Im\frac1\pi\int_{\rho-\varepsilon}^{\rho+\varepsilon}
\mathbb{E}\left[
{\rm Tr}_N[R_N(x+iy)f_i(A_N)]
\right]h(t)\,dt=-
\frac{\delta_{i_0i}}{\omega'_1(\rho)},\quad i=1,\dots,p,
\end{equation*}
or more simply, because $f_i(A_N)$ is the projection of $\mathbb{C}^N$ onto the $i$th coordinate,
\begin{equation}
\label{sperantze}
\lim_{N\to\infty}\lim_{y\downarrow0}\Im\frac1\pi\int_{\rho-\varepsilon}^{\rho+\varepsilon}
\mathbb{E}\left[
[R_N(x+iy)]_{ii}
\right]h(t)\,dt=-
\frac{\delta_{i_0i}}{\omega'_1(\rho)},\quad i=1,\dots,p.
\end{equation}
Lemma \ref{approximatesubordination} suggests writing
\begin{equation}\label{fund} 
\mathbb{E} [ R_N(z)_{ii}] = \frac{1}{\omega_1(z)-\theta_i}
+\Delta_{i,N}(z),\quad i=1,\dots,p,z\in\mathbb{C}^+.
\end{equation}
We proceed to estimate the functions $\Delta_{i,N}$.

\begin{prop}\label{estimfonda}
There exist positive numbers $\{a_N\}_N$ such that $\lim_{N\to\infty}a_N=0$ and   
 $$\left| \Delta_{i,N} (z)\right| \leq a_N (1+\vert z\vert)^4(1+|\Im z|^{-1})^4, \quad z\in \mathbb{C}\setminus \mathbb{R},i=1,\dots,p.$$

\end{prop}
\begin{proof}   
Define analytic functions $\omega_{N,i}$ for $i=1,\dots,p$ using \eqref{approxsubordfcn} with $i,A_N,B_N$ in place of $k,C_N,D_N$, respectively. 
These functions are analytic  outside the interval $[-\|A_N\|-\|B_N\|,\|A_N\|
+\|B_N\|]$, and the  hypothesis that $\|A_N\|$ and $\|B_N\|$ are uniformly bounded  implies their analyticity on $\mathbb C\setminus
[-m,m]$ for some $m>0$ independent of $N$. The operator $\mathbb E[R_N(z)]$ belongs to $\{A_N\}''$  by Lemma \ref{bicommutant}, and is therefore
a diagonal operator, so
$$\omega_{N,i}(z)=\left(\mathbb E[R_N(z)]^{-1}\right)_{ii}
+\theta_i,\quad i=1,\dots,p.$$
 By Lemma \ref{l} and the considerations preceding it (especially \eqref{4.6} and 
\eqref{4.7}),
\begin{eqnarray*}
\lim_{z\to\infty}\frac{\omega_{N,i}(z)}z&=&(I_N)_{ii}=1,\\
\lim_{z\to\infty}\omega_{N,i}(z)-z&=&-(A_N+\mathbb E[U_N^*B_NU_N])_{ii}+\theta_i=-{\rm tr}_N(B_N),
\end{eqnarray*}
and
\begin{eqnarray*}
{\lim_{z\to\infty}z(\omega_{N,i}(z)-z+\mathrm{tr}_N(B_N))}
&=&-\mathbb E[(X_N-\mathbb E[(X_N)])^2]_{ii}\\
&=&-(\mathrm{tr}_N(B_N^2)-\mathrm{tr}_N(B_N)^2),
\end{eqnarray*}
for  $N\ge p$ and $i=1,\dots,p$. 
It follows that we have, as in \eqref{Nevanlinna}, Nevanlinna representations of the form
$$
\omega_{N,i}(z)=z-\mathrm{tr}_N(B_N)-\int_{[-m,m]}\frac{1}{z-t}\,d\sigma_{N,i}(t),\quad z\in\mathbb C\setminus[-m,m],
$$
where $\sigma_{N,i}$ is a positive measure on $[-m,m]$, with total mass $\mathrm{tr}_N(B_N^2)-\mathrm{tr}_N(B_N)^2$.  Similarly, the subordination function $\omega_1$ from \eqref{EqSubord+} can be written as
$$
\omega_1(z)=z-\int_\mathbb Rt\,d\nu(t)-\int_\mathbb R\frac{1}{z-t}\,d\sigma(t),\quad z\in\mathbb C^+.
$$
The hypothesis that the empirical eigenvalue distribution of $B_N$ converges to $\nu$ implies in particular
$\lim_{N\to\infty}\mathrm{tr}_N(B_N)=\int_\mathbb Rt\,d\nu(t).$ In addition, the fact that 
 $\lim_{N\to\infty}\omega_{N,i}=\omega_1$
uniformly  on compact subsets of $\mathbb{C}\setminus[-m,m]$ implies that $\sigma$ is supported in $[-m,m]$ and that $\lim_{N\to\infty}\sigma_{N,i}=\sigma$ in the weak${}^*$-topology. We also have
$$\lim_{N\to\infty}\|\sigma_{N,i}\|=
\int_\mathbb Rt^2\,d\nu(t)-\left(\int_\mathbb Rt\,d\nu(t)\right)^2.$$
 Lemma 
\ref{conv}, applied  to the sequence $\rho_{N,i}=\sigma_{N,i}-\sigma$ yields positive numbers  
$\{v_{N,i}\}_{N\ge s}$ such that $\lim_{N\to\infty}v_{N,i}=0$ and
\begin{eqnarray}\label{for-e-vectors}
|\omega_{N,i}(z)-\omega_1(z)|&<&v_{N,i}\frac{|z|^2+2\sup_{N\in\mathbb N}(\mathrm{tr}_N(B_N^2)-\mathrm{tr}_N(B_N)^2)}{(\Im z)^2}\\
& & \mbox{}+\left|\int_\mathbb R t\,d\nu(t)-\mathrm{tr}_N(B_N)\right|,\quad z\in\mathbb{C}^+.\nonumber
\end{eqnarray}
We can now estimate 
\begin{eqnarray*}
\left|\mathbb E\left[R_N(z)_{ii}\right]-\frac{1}{\omega_1(z)-\theta_i}\right|&=&
\left|\frac{1}{\omega_{N,i}(z)-\theta_i}-\frac{1}{\omega_1(z)-\theta_i}\right|\\
&=&\frac{|\omega_{N,i}(z)-\omega_1(z)|}{|\omega_{N,i}(z)-\theta_i||\omega_{1}(z)-\theta_i|}\\
&<&\frac{|\omega_{N,i}(z)-\omega_1(z)|}{|\Im z|^2}\\
&\leq&a_N(1+|z|)^4\cdot(1+|\Im z|^{-4}),
\end{eqnarray*}
where
$$a_N=b_N(1+2\sup_{N\in\mathbb N}(\mathrm{tr}_N(B_N^2)-\mathrm{tr}_N(B_N)^2))^2,$$
and
$$b_N=\max\left\{v_{N,1},\dots,v_{N,p},\left|\int_\mathbb{R}t\,d\nu(t)-{\rm{tr}}_N(B_N)\right|\right\}.$$
The proposition follows.
\end{proof}
\begin{cor}We have
$$\lim_{N\to\infty}\limsup_{y\to0}\left|\int_\mathbb{R}\Delta_{i,N}(t+iy)h(t)\,dt\right|=0,\quad i=1,\dots,p.$$
\end{cor}
\begin{proof}
The preceding proposition allows us to apply Lemma \ref{HT} to obtain a positive constant $c$ such that
$$
\limsup_{y\downarrow0}\left|\int_{\rho-\varepsilon}^{\rho+\varepsilon}\Delta_{i,N}(t+iy)h(t)\,dt\right|\le ca_N$$
for $N\ge p$ and $i=1,\dots,p$. The corollary follows.
\end{proof}
The preceding  result, combined with \eqref{fund}, shows that \eqref{sperantze} is equivalent to
\begin{equation}\label{inca una}
\lim_{y\downarrow0}\Im\frac1\pi\int_{\rho-\varepsilon}^{\rho+\varepsilon}
\frac{h(t)}{\omega_1(t+iy)-\theta_i}\,dt=-
\frac{\delta_{i_0i}}{\omega'_1(\rho)},\quad i=1,\dots,p.
\end{equation}
This is easily verified.  Indeed, denote by $\Omega_y$, $y>0$, the rectangle with vertices $\rho\pm\varepsilon/2\pm iy$.  Calculus of residues yields
$$\frac1{2\pi i}\int_{\partial\Omega_y}\frac1{\omega_1(z)-\theta_i}\,dz=
\frac{\delta_{i_0i}}{\omega'_1(\rho)},\quad i=1,\dots,p.
$$
On the other hand,
$$
\Im\frac1\pi\int_{\rho-\varepsilon}^{\rho+\varepsilon}
\frac{h(t)}{\omega_1(t+iy)-\theta_i}\,dt=\frac1{2\pi i}\int_{\rho-\varepsilon}^{\rho+\varepsilon}
\left[
\frac{h(t)}{\omega_1(t+iy)-\theta_i}-\frac{h(t)}{\omega_1(t-iy)-\theta_i}
\right]\,dt.
$$
Now we use the fact that $h=1$ on $(\rho-\varepsilon/2,\rho+\varepsilon/2)$ to conclude that
$$
\Im\frac1\pi\int_{\rho-\varepsilon}^{\rho+\varepsilon}
\frac{h(t)}{\omega_1(t+iy)-\theta_i}\,dt+\frac1{2\pi i}\int_{\partial\Omega_y}\frac1{\omega_1(z)-\theta_i}\,dz.
$$
is a sum of the following four integrals:
$$\Im\frac1\pi\int_{\rho-\varepsilon}^{\rho-\varepsilon/2}
\frac{h(t)}{\omega_1(t+iy)-\theta_i}\,dt,\quad
\Im\frac1\pi\int_{\rho+\varepsilon/2}^{\rho+\varepsilon}
\frac{h(t)}{\omega_1(t+iy)-\theta_i}\,dt,
$$
$$
\frac1{2\pi i}\int_{\rho-\varepsilon/2-iy}^{\rho-\varepsilon/2+iy}\frac1{\omega_1(z)-\theta_i}\,dz,\quad
\frac1{2\pi i}\int_{\rho+\varepsilon/2-iy}^{\rho+\varepsilon/2+iy}\frac1{\omega_1(z)-\theta_i}\,dz,
$$
all of which are easily seen to tend to zero as $y\downarrow0$. This completes the proof of 
\eqref{inca una} and therefore of \eqref{pr-a}.
We observe now that the proof of \eqref{pr-a} for $i\ne i_0$ uses only the fact that $\omega_1(\rho)
\ne\theta_i$.  Therefore switching the roles of $A_N$ and $B_N$ in this argument yields a proof of  
\eqref{pr-b} and completes the proof of part (3) of Theorem \ref{Main+} in this case if $k=1$ and
$\ell=0$. The case $\ell=1$ and $k=0$ follows by symmetry.

Assertion (4) of the Theorem follows from (3) simply because $E_{X_N}((\rho-\varepsilon,\rho+\varepsilon
))$ is a projection of rank one. Indeed, denote by $\{e_i\}_{i=1}^N$ the canonical basis in $\mathbb C^N$, 
so $A_Ne_i=\theta_ie_i$. Let $\xi$ be a unit vector in the range of 
$E_{X_N}((\rho-\varepsilon,\rho+\varepsilon))$, so $E_{X_N}((\rho-\varepsilon,\rho+\varepsilon))h=
\langle h,\xi\rangle\xi$ for every $h\in\mathbb C^N$. Direct calculation shows that 
$$
P_NE_{X_N}((\rho-\varepsilon,\rho+\varepsilon))P_Ne_{i_0}=\sum_{i=1}^p\langle e_{i_0},\xi\rangle
\langle\xi, e_i\rangle e_i
$$
and thus, almost surely for large $N$,
$$
\left\|\sum_{i=1}^p\langle e_{i_0},\xi\rangle
\langle\xi, e_i\rangle e_i-\frac{1}{\omega_1'(\rho)}e_{i_0}\right\|<\varepsilon.
$$
In particular, we obtain $\left||\langle e_{i_0},\xi\rangle|^2-1/\omega_1'(\rho)\right|<\varepsilon$,
which is precisely the first relation in (4). The case $k=0,\ell=1$ follows by symmetry.
\newline\noindent{\bf Step B:} In this step we prove (3) and (4) in the general case of spikes with higher multiplicities and arbitrary values for $k$ and $\ell$. We use an idea
 from  \cite{Capitaine11} to reduce the problem to the case  considered in Step A.
Given positive numbers $\eta$ and $\delta$, set
\begin{eqnarray}
A_{N,\eta}&=&A_N+\mathrm{Diag}(p\eta,(p-1)\eta,\dots,\eta,\underbrace
{
0,\dots,0
}_{N-p}),\label{perturbA}\\
B_{N,\delta}&=&B_N+\mathrm{Diag}(q\delta,(q-1)\delta,\dots,\delta,\underbrace
{
0,\dots,0
}_{N-q})\label{perturbB}
\end{eqnarray}
for $N\ge p+q$. These matrices have distinct spikes $\theta_i(\eta)=\theta_i+(p-i+1)\eta$ and 
$\tau_j(\delta)=\tau_j+(q-j+1)\delta$, respectively. The fact that $\omega_1$ is increasing and 
continuous at $\rho$ implies that, for sufficiently small $\eta$, there  exist exactly $k$ indices $i_1,
\dots,i_k$ such that the equations $\omega_1(t)=\theta_{i_n}(\eta)$ each have a solution $\rho_n=\rho_n(\eta)\in
(\rho-\varepsilon,\rho+\varepsilon)$, $n=1,2\ldots k$. Similarly, for sufficiently small $\delta$ there exist  $\ell$ indices 
$j_1,\dots,j_\ell$ and $\ell$ values $\rho_{k+n}=\rho_{k+n}(\delta)\in(\rho-\varepsilon,
\rho+\varepsilon)$ such that $\omega_2(\rho_{k+n}(\delta))=\tau_{j_n}(\delta)$, $n=1,\ldots,\ell$. The 
numbers $\eta$ and $\delta$ 
can be chosen such that the intervals $(\rho_n-2\eta,\rho_n+2\eta),$
$n=1,2,\ldots,k+\ell,$ are pairwise disjoint and contained in $(\rho-\varepsilon,\rho+\varepsilon)$.  We 
conclude that the arguments of Step A hold with $X_{N,\eta,\delta}=A_{N,\eta}+U_N^*B_{N,\delta}U_N$, 
$\rho_n$, and $\eta$ in place of $X_N$, $\rho$, and $\varepsilon$, respectively. Thus,
$$
\lim_{N\to\infty}
\left\|
P_NE_{X_{N,\eta,\delta}}((\rho_n-\eta,\rho_n+\eta))P_N-\frac1{\omega'_1(\rho_n)}E_{A_{N,\eta}}(\{\omega_1(\rho_n)\})
\right\|=0
$$
almost surely for $n=1,\dots,k+\ell$. We have
$$
\sum_{n=1}^{k+\ell}E_{X_{N,\eta,\delta}}((\rho_n-\eta,\rho_n+\eta))=E_{X_{N,\eta,\delta}}
((\rho-\varepsilon,\rho+\varepsilon))
$$
and also, noting that $E_{A_N,\eta}(\{\omega_1(\rho_n)\})=0$ for $n=k+1,\dots,k+\ell$,
$$
\sum_{n=1}^{k+\ell}E_{A_N,\eta}(\{\omega_1(\rho_n)\})=E_{A_N}(\omega_1(\rho))
$$
for small $\eta$. In addition, $1/\omega_1'(\rho_n)$ can be made arbitrarily close to 
$1/\omega_1'(\rho)$ by making $\eta$ sufficiently small. We conclude that
$$
\left\|
P_NE_{X_{N,\eta,\delta}}((\rho-\varepsilon,\rho+\varepsilon))P_N-\frac1{\omega'_1(\rho)}E_{A_{N}}
(\{\omega_{1}(\rho)\})\right\|<\varepsilon
$$
almost surely for large $N$ if $\eta$ is sufficiently small.  Clearly,
$$\|X_N-X_{N,\eta,\delta}\|\le p\eta+q\delta.$$
An application of  Lemma \ref{eigenspaces} shows that
\begin{equation}\label{diffofspectrproj}
\|E_{X_{N,\eta,\delta}}((\rho-\varepsilon,\rho+\varepsilon)-
E_{X_{N}}((\rho-\varepsilon,\rho+\varepsilon))\|
\end{equation} 
can be made arbitrarily small for appropriate choices of $\eta$ and $\delta$, uniformly in $N$.  The 
first inequality in (3) follows at once, and the second one is proved similarly.

We now verify assertion (4) when $\ell=0$. Let $\xi$ be a unit vector in the range of 
$E_{X_{N}}((\rho-\varepsilon,\rho+\varepsilon))$. 
Since the quantity in \eqref{diffofspectrproj} is small, we can find, almost surely for large 
$N$, unit vectors
$\xi_{\eta,\delta}^{(N)}\in E_{X_{N,\eta,\delta}}((\rho-\varepsilon,\rho+\varepsilon))$ such that 
$\lim_{\delta+\eta\to0}\|\xi^{(N)}-\xi_{\eta,\delta}^{(N)}\|=0$, uniformly in $N$. It suffices therefore to 
prove that
$$
\limsup_{N\to\infty}\left|\|E_{A_N}(\{\omega_1(\rho)\})\xi_{\eta,\delta}^{(N)}\|^2-
\frac{1}{\omega'_1(\rho)}\right|$$
can be made arbitrarily small for appropriate choices of $\eta$ and $\delta$. Write $\xi_{\eta,\delta}^{(N)}=
\xi_{1}^{(N)}+\cdots+\xi_{k}^{(N)}$ with $\xi_{n}^{(N)}$ in the range of $E_{X_{N,\eta,\delta}}(\{\rho_n\})$,
$n=1,\ldots,k$. The case of assertion (4) proved in Step A shows that for any $ \eta,\delta>0$ sufficiently
small,
$$
\lim_{N\to\infty}\|E_{A_{N,\eta}}(\{\omega_1(\rho_n)\})\xi_{n}^{(N)}\|^2-
\frac{\|\xi_{n}^{(N)}\|^2}{\omega_1'(\rho_n)}=0,\quad n=1,\dots,k.
$$
We also have $\lim_{N\to\infty}\|E_{A_{N,\eta}}(\{\theta_{i}(\eta)\})\xi_n^{(N)}\|=0$ for $i\not\in\{i_1,
\dots, i_k\}$. Since 
$$
E_{A_{N}}(\{\omega_1(\rho)\})=E_{A_{N,\eta}}(\{\omega_1(\rho_1),\dots,\omega_1(\rho_k)\}),
$$
the relation
\begin{eqnarray*}
\left|\|E_{A_{N}}(\{\omega_1(\rho)\})\xi_{\eta,\delta}^{(N)}\|^2-
\frac{1}{\omega_1'(\rho)}\right|
& \leq&
\left|\sum_{n=1}^k\|E_{A_{N,\eta}}(\{\omega_1(\rho_n)\})\xi_{n}^{(N)}\|^2-
\frac{\|\xi_{n}^{(N)}\|^2}{\omega_1'(\rho_n)}\right|\\
& &\mbox{}+
\sum_{n=1}^k\|\xi_{n}^{(N)}\|^2\left|\frac{1}{\omega_1'(\rho_n)}-\frac{1}{\omega_1'(\rho)}\right|
\end{eqnarray*}
implies
$$
\limsup_{N\to\infty}
\left|\|E_{A_{N}}(\{\omega_1(\rho)\})\xi_{\eta,\delta}^{(N)}\|^2-
\frac{1}{\omega_1'(\rho)}\right|\leq
\max_{1\leq n\leq k}\left|\frac{1}{\omega_1'(\rho_n)}-\frac{1}{\omega_1'(\rho)}\right|.
$$
The desired conclusion follows
by noting that $\omega_1'(\rho_n)=\omega_1'(\rho_n(\eta))$ can be made 
arbitrarily close to $\omega_1'(\rho)$. 
The second part of 
assertion (4) follows by symmetry.
\end{proof}

The proofs of the versions of Theorem \ref{MainX+} for the positive line and for the unit circle follow
the same outline. We avoid excessive repetition and only indicate the differences in the tools used 
throughout the proof.

\subsection{The multiplicative model $X_N=A_N^{1/2}U_N^*B_NU_NA_N^{1/2}$\label{sec:5.2}}
We use the notations from Subsection \ref{subsec:Xperturb+}.
As in the previous section, we assume that both $A_N$ and $B_N$ 
are diagonal matrices: 
$$
A_N=\displaystyle{\text{ Diag}(\theta_1,\ldots,\theta_p, \alpha _1^{(N)}, \ldots , \alpha_{N-p}^{(N)})},
$$
$$
B_N=\displaystyle{\text{Diag}(\tau_1,\ldots,\tau_q,\beta_1^{(N)},\ldots,\beta_{N-q}^{(N)})}.
$$
Since $\mu, \nu \neq \delta_0$, fix $\alpha \in \mathrm{supp}(\mu )\setminus \{0\}$ and
$\beta \in \mathrm{supp}(\nu )\setminus \{0\}$.
We use the following multiplicative decompositions: 
$$
A_N=A_N'A_N''=A_N''A_N',
$$
$$
A_N'=\text{ Diag}(\alpha , \ldots , \alpha , \alpha_1^{(N)}, \ldots , \alpha_{N-p}^{(N)}),
$$
$$
A_N''=\displaystyle{\text{ Diag}(\theta_1/\alpha,\ldots,\theta_p/\alpha, 1, \ldots , 1)},$$
$$B_N=B_N'B_N''=B_N''B_N',$$
$$B_N'=\text{ Diag}(\beta , \ldots , \beta , \beta_1^{(N)}, \ldots , \beta_{N-q}^{(N)}),$$
$$
B_N''=\displaystyle{\text{ Diag}({\tau_1}/\beta,\ldots,{\tau_q}/{\beta}, 1, \ldots , 1)}.
$$

As before, we write
$A_{N}''=P_N^*\Theta P_N+I_N-P_N^*P_N$, where $P_N$ is defined as in Section \ref{PLUS} and
\[
\Theta={\displaystyle {\text{ Diag}(\theta_{1}/\alpha,\ldots,\theta_{p}/\alpha)}.}
\]
Similarly, $B_{N}''=Q_N^*{T}Q_N+I_N-Q_N^*Q_N$, where $Q_N$ is defined as in Section \ref{PLUS} and
\[
{T}={\displaystyle{\text{Diag}(\tau_{1}/\beta,\ldots,\tau_{q}/\beta)}.}
\]

We discuss first the behavior of the model $A_N^{1/2}U_N^*B_N'U_NA_N^{1/2}$ with spikes only on 
$A_N$. The essential step is a reduction to a convergence problem for a sequence of matrices of fixed 
size.

\subsubsection{Reduction to the almost sure convergence of a $p\times p$ matrix\label{521}}
Recall that $K=\mathrm{supp}(\mu\boxtimes\nu)$.
Corollary 3.1 of \cite{ColMal11} yields the existence of positive random variables $\{\delta_N\}_{N
\in\mathbb N}$ such that
$$ 
\sigma((A_N^{\prime})^{1/2}U_N^*B_N'U_N(A_N')^{1/2})\subseteq K_{\delta_N}
$$ 
and $\lim_{N\to\infty}\delta_N=0$ almost surely.

We argue first that, in case $0\not\in\text{supp}(\mu\boxtimes\nu)$, it follows that $X_N$ is 
almost surely invertible for large $N$. Indeed, in this case, $0\not\in\text{supp}(\mu)\cup
\mathrm{supp}(\nu)$. Therefore, $A_N$ and $B_N$ are invertible for large $N$, and thus so is $X_N$.
This observation allows us to restrict the analysis to nonzero eigenvalues of $X_N$.

Denote $X_N'=A_N^{1/2}U_N^*B_N'U_NA_N^{1/2}.$
Fix $z\in\mathbb{C}\setminus( K_{\delta _N}\cup\{0\})$ so that the matrix $zI_N-
(A_N')^{1/2}U_N^*B_N'U_N(A_N')^{1/2}$ is invertible. Using Sylvester's identity $\det(I-XY)=
\det(I-YX)$, we obtain for large $N$
\begin{eqnarray*}
\det(zI_N-X_N') &=& z^N\det(I_N-{z}^{-1}A_N''(A_N')^{1/2}U_N^*B_N'U_N(A_N')^{1/2})\\
&=& z^N\det(I_N-A_N''+A_N''(I_N-z^{-1}(A_N')^{1/2}U_N^*B_N'U_N(A_N')^{1/2}))\\
&=& \det((I_N-A_N'')(I_N-{z}^{-1}(A_N')^{1/2}U_N^*B_N'U_N(A_N')^{1/2})^{-1}+A_N'')\\
&  & \times \det(zI_N-(A_N')^{1/2}U_N^*B_N'U_N(A_N')^{1/2}).
\end{eqnarray*}
The matrix $(I_N-A_N'')(I_N-\frac{1}{z}(A_N')^{1/2}U_N^*B_N'U_N(A_N')^{1/2})^{-1}+A_N''$ is of 
the form:
$$\left[
\begin{array}{cc} F_N(z) & * \\ 
0 & I_{N-p} \end{array}\right],
$$
where $F_N$ is the analytic function with values in $M_p(\mathbb{C})$
defined on $\mathbb{C}\setminus (K_{\delta_N}\cup\{0\})$ by 
\begin{equation}\label{MN'}
F_N(z):=(I_p-
\Theta )P_N\left(I_N-\frac{1}{z}(A_N')^{1/2}U_N^*B_N'U_N(A_N')^{1/2}\right)^{-1}\,P^*_N+\Theta ,
\end{equation}
and $\Theta$ is the diagonal $p\times p$ matrix defined earlier.
Thus, for large $N$, the nonzero eigenvalues of $X_N'$ outside $K_{\delta_N}$
are precisely the zeros of $\det (F_N)$ in that open set. As in Subsection \ref{PLUS}, the random
matrix functions sequence $\{F_N\}_{N}$ converges a.s. to a diagonal deterministic $p\times p$
matrix function:

\subsubsection{Convergence of $F_{N}$ } We start with the analogue of Proposition
\ref{estimationinmean}.

\begin{prop} \label{uniformconvergence'}
Fix a positive integer $p$, and let $C_N$ and $D_N$ be deterministic nonnegative diagonal $N\times N$ 
matrices with uniformly bounded norms such that the limits
$$
\eta_i=\lim_{N\to\infty}(C_N)_{ii}
$$
exist for $i=1,\ldots,p$. Suppose that the empirical eigenvalue distributions of $C_N$ and $D_N$ 
converge weakly to $\mu$ and $\nu$, respectively.
Then the resolvent 
$$
R_N(z)=\left(zI_N-C_N^{1/2}U_N^*D_NU_NC_N^{1/2}\right)^{-1},\quad z\in\mathbb C\setminus\mathbb R
$$
satisfies
$$
\lim_{N\to\infty}P_N\mathbb E\left[zR_N(z)\right]P_N^*=\mathrm{Diag}\left(\frac{1}{1-\eta_1
\omega_1(z^{-1})},\ldots,\frac{1}{1-\eta_p
\omega_1(z^{-1})}\right).
$$
\end{prop}
\begin{proof}
We consider without loss of generality elements $z\in\mathbb C^+$. If $C_N$ is invertible, then 
Lemma \ref{bicommutant}, part (2), applied to $b=zC_N^{-1}$ implies that 
$\mathbb{E}[{ R}_N(z)]$ is diagonal. If $C_N$ is not invertible, then 
$$
\lim_{\varepsilon\downarrow0}\mathbb{E}[(zI_N-(C_N+\varepsilon I_N)^{1/2}U_N^*D_NU_N
(C_N+\varepsilon I_N)^{1/2})^{-1}]=\mathbb{E}[{ R}_N(z)].
$$
The limit of diagonal matrices is diagonal, so $\mathbb{E}[{R}_N(z)]$ is diagonal.
Define
\begin{equation}\label{omegaN'}
\omega _{N,i}(z):=\frac{1}{(C_N)_{ii}}\left(1-\frac{z}{\mathbb E\left[{R}_N\left(z^{-1}\right)\right]_{ii}}\right),\quad1\leq i\leq p.
\end{equation}
We prove the uniform convergence on compact subsets of $\mathbb{C}\setminus 
\mathbb{R}^+$ of the sequences $\{\omega _{N,i}\}_{N\geq p}$ of analytic functions to
$\omega_1$. The multiplicative counterpart of Lemma \ref{approximatesubordination} is as follows:
\begin{lem} \label{approximatesubordination'}
Assume that $C_N\ge\varepsilon I_N$ for some $\varepsilon>0$.
We have: 
$$\lim_{N\to\infty}\left\| z\mathbb{E}[{R}_N(z)]-\left(I_N-\omega _{N,i}\left(z^{-1}\right)C_N\right)^{-1}\right\| 
=0,\quad z\in \mathbb{C}\setminus \mathbb{R},i\in\{1,\dots,p\}.$$
\end{lem}

\begin{proof}
For $z\in \mathbb{C}^+$, define 
$$
\Omega_N(z)=(C_N)^{-1}\mathbb E\left[{ R}_N\left(z^{-1}\right)\right]^{-1}=
\mathbb E\left[\left((zC_N)^{-1}-U_N^*D_NU_N\right)^{-1}\right]^{-1}.
$$ 
This function is well-defined by Lemma \ref{l}, and the second equality is justified by Lemma
\ref{bicommutant}(2). We apply Lemma \ref{bicommutant}(1) with $b=(zC_N)^{-1}$ to obtain
\begin{eqnarray*}
\lefteqn{Y(\Omega_N(z)-(zC_N)^{-1})-(\Omega_N(z)-(zC_N)^{-1})Y}\\
& = & \Omega_N(z)\mathbb E\left[\left(\left((zC_N)^{-1}-U_N^*D_NU_N\right)^{-1}-\Omega_N(z)^{-1}\right)\right.\\
& & \mbox{}\times(Y(zC_N)^{-1} -(zC_N)^{-1} Y)\\
& & \mbox{}\times\left.\left(\left((zC_N)^{-1}-
U_N^*D_NU_N\right)^{-1}-\Omega_N(z)^{-1}\right)\right]\Omega_N(z).
\end{eqnarray*}
Consider arbitrary vectors of norm one $h,k\in\mathbb C^N$ and an $Y$ of rank one to conclude the 
existence of rank one projections $p_1,p_2$ and rank 2 projections $q_1,q_2$ so that 
\begin{eqnarray*}
\lefteqn{\left|k^*(Y(\Omega_N(z)-(zC_N)^{-1})-(\Omega_N(z)-(zC_N)^{-1})Y)h\right|\leq}\\
& & \|\Omega_N(z)\|^2\|Y(zC_N)^{-1} -(zC_N)^{-1} Y\|\\
& & \mbox{}\times\mathbb E\left[\left\|p_1\left(\left((zC_N)^{-1}-U_N^*D_NU_N\right)^{-1}-
\Omega_N(z)^{-1}\right)q_1\right\|^2\right]^{1/2}\\
& & \mbox{}\times\mathbb E\left[\left\|q_2\left(\left((zC_N)^{-1}-U_N^*D_NU_N\right)^{-1}-\Omega_N(z)^{-1}\right)p_2\right\|^2\right]^{1/2}.
\end{eqnarray*}
Lemma \ref{l} yields $\|\Omega_N(z)\|<\|(zC_N)^{-1}\|+\|B_N'\|+
4c\|(1/\Im(1/z))C_N\|$, with $c\in(0,+\infty)$. Remark 
\ref{4.13} provides estimates
 for the last two factors. The estimate
$\|Y(zC_N)^{-1} -(zC_N)^{-1} Y\|<2\|Y\||z^{-1}|\varepsilon^{-1}$ is obvious. Thus, 
$$
\left|k^*(Y(\Omega_N(z)-(zC_N)^{-1})-(\Omega_N(z)-(zC_N)^{-1})Y)h\right|\leq\frac{C(z,
\varepsilon)}{N}\|Y\|
$$
for some constant $C(z,\varepsilon)$ independent of $N$. The $(i,i)$ entry of the matrix 
$\Omega_N(z)-(zC_N)^{-1}$ is precisely $e_i^*(\Omega_N(z)-(zC_N)^{-1})e_i$, which 
belongs to the numerical range of $\Omega_N(z)-(zC_N)^{-1}$. Lemma 
\ref{lem:W(T)} yields
$$
\|\Omega_N(z)-(zC_N)^{-1}-(e_i^*(\Omega_N(z)-(zC_N)^{-1})e_i)I_N\|<2\frac{C(z,\varepsilon)}{N}.
$$
Since $C_N$ is diagonal, the lemma follows by letting $N\to\infty$.
\end{proof}
The proof of Proposition \ref{uniformconvergence'} $C_N$ is bounded from
below by a positive multiple of $I_N$ is now completed by an application of the above lemma.
Indeed, using Biane's subordination formula \eqref{subord1'} and the 
asymptotic freeness result of Voiculescu \cite{Voiculescu91}, we obtain 
$\lim_{N\to\infty}\text{tr}_N(z\mathbb E[{R}_N(z)])=1+\psi_{\mu\boxtimes\nu}
(1/z)=1+\psi_\mu(\omega_1(1/z))$. Clearly, $\text{tr}_N((I_N-\omega_{N,i}
(1/z)C_N)^{-1})\to1+\psi_\mu(\lim_{N\to\infty}\omega_{N,i}(1/z))$. The result
follows by analytic continuation.
The general case follows by replacing a non-invertible $C_N$ by $C_N+\varepsilon I_N$. The 
approximation is uniform in $N$, so a normal family argument yields the desired result as $\varepsilon
\to0$.
\end{proof}

Observe that 
$$
F_N(z)=(I_p-\Theta )P_N(zR_N'(z))\,P_N^*+\Theta,
$$ 
where $R_N'$ denotes the resolvent of $(A_N')^{1/2}U_N^*B_N'U_N(A_N')^{1/2}$.
An application of Proposition \ref{uniformconvergence'} to $C_N=A_N'$ and $D_N=B_N'$, Remark 
\ref{4.13}, as well as of Lemma \ref{lem:extensionX}, yield the following result. We leave the details, 
similar to the ones in the proof of Proposition \ref{uniformconvergence}, to the reader.

\begin{prop} \label{estimationinmean'}
Almost surely, the sequence $\{F_N\}_N$ converges uniformly on the compact subsets of 
$\overline{\mathbb C}\setminus K$ to the analytic function $F$ defined on $\overline{\mathbb C}\setminus K$ by 
$$
F(z)=\mathrm{Diag}\left(\left(1-\frac{\theta_1}{\alpha}\right)\frac{1}{1-\alpha\omega_1(z^{-1})}+
\frac{\theta_1}{\alpha},\ldots,\left(1-\frac{\theta_p}{\alpha}\right)\frac{1}{1-\alpha\omega_1(z^{-1})}+
\frac{\theta_p}{\alpha}\right).
$$
\end{prop}

\subsubsection{Proofs of the main results for the positive multiplicative model.}
\begin{proof}
[Proof of Theorem {\rm\ref{MainX+}}, parts {\rm (1)} and {\rm(2)}---eigenvalue behaviour.] 
$\frac{}{}$
\newline\noindent{\bf Step 1.} We prove our result for the model $A_N^{1/2}U_N^*B_N'U_N
A_N^{1/2}$ in which only $A_N$ has spikes. 
We consider the almost sure event, whose existence is guaranteed by Proposition 
\ref{estimationinmean'}, on which there exist a random sequence 
$\{\delta_N\}_{N}\subset(0,+\infty)$ converging to zero so that:
\begin{itemize}
\item $\sigma(A_N'^{1/2}U_N^*B_N'U_NA_N'^{1/2})\subseteq K_{\delta_N}$, 
and
\item the sequence $\{F_N(z)\}_{N\ge p}$ converges to
$$
F(z) = 
\mathrm{diag}\left(\left(1-\frac{\theta_1}{\alpha}\right)\frac{1}{1-\alpha\omega_1\left(z^{-1}\right)}+\frac{\theta_1}{\alpha}, \ldots ,\left(1 - \frac{\theta_p}{\alpha}\right)\frac{1}{1-\alpha \omega_1 \left(z^{-1}\right)}+\frac{\theta_p}{\alpha}\right).
$$ 
uniformly on the compact subsets of $\mathbb C\setminus K$.
\end{itemize}
On this event, we apply Lemma \ref{alt-Benaych-Rao} with $\gamma =\mathbb R$, the 
sequence $\{F_N\}_{N\ge p}$ and its uniform on compacts limit $F$. We argue first that the function 
$F_N(z)$ given by equation \eqref{MN'} is invertible for $z\not\in\mathbb R$. Indeed, the relations 
preceding \eqref{MN'} imply that, if $F_N(z)$ is not invertible, then $z$ is an eigenvalue of the
selfadjoint matrix $X_N'$, and hence it is a real number. This verifies hypothesis (2). Hypotheses (1)
and (3) follow from Proposition \ref{estimationinmean'}. Finally, $F(\infty)=I_p$,
$$
(F'(z))_{jj}=\frac{\omega_1'(1/z)(\theta_j-\alpha)}{z^2(1-\alpha\omega_1(1/z))^2},
$$
and the zeros of $\omega_1'$ are simple by the Julia-Carath\'eodory Theorem.
Thus,
Lemma \ref{alt-Benaych-Rao} applies to $F_N$ and $F$.

For almost every $\delta>0$,
the boundary points of $K_{\delta}$ are not zeroes of
$\det(F)$. When this condition is satisfied, Lemma \ref{alt-Benaych-Rao} yields precisely
the conclusion of Theorem \ref{MainX+} (1)--(2), when $q=0$. Indeed, as noted above, 
the nonzero eigenvalues of $X_N'$ 
in $\mathbb{C}\setminus K_{\delta}$ are exactly the zeroes of $\det(F_{N})$, and the set of points $z$
such that $F(z)$ is not invertible is precisely $\bigcup_{i=1}^pv_1^{-1}(\{1/
\theta_i\})$. This completes the first step.
\newline\noindent{\bf Step 2.} This is completely analogous to the reasoning
from the second step of the proof of Theorem \ref{Main+}(1)--(2). We omit the details.
\end{proof}

\begin{proof}[Proof of Theorem {\rm\ref{MainX+}}, parts {\rm (3)} and {\rm(4)}---eigenvspace behaviour.] 
$\frac{}{}$
\newline\noindent{\bf Step A.} We assume first that $\theta_1>\cdots>\theta_p>0$,
$\tau_1>\cdots>\tau_q>0$, $\ell=0$ and $k=1$.  Step A of the proof of Theorem \ref{Main+} is 
modified as follows: $X_N$ is replaced by $A_N^{1/2}U_N^*B_NU_NA_N^{1/2}$, the analogue of Lemma 
\ref{inegconc} holds with the constant $C$ replaced by $\sup_N\|A_N\|\|B_N\|$, and Proposition 
\ref{estimfonda} is replaced by the following statement.

\begin{prop}\label{estimfondaX}
There is a polynomial $P$ with nonnegative coefficients,  a sequence $\{a_N\}_N$ of nonnegative
real numbers converging to zero when $N$ goes to infinity and some nonnegative integer number 
$t$, such that for every $i=1,\ldots,p$ and  $z\in\mathbb{C}\setminus \mathbb{R}$,
\begin{equation}\label{fundX} 
\mathbb{E} [ R_N(z)_{ii}] = \frac{1}{z} \frac{1}{1-\theta_i\omega_1(1/z)}
+\Delta_{i,N}(z),
\end{equation}
with $$\left| \Delta_{i,N} (z)\right| \leq (1+\vert z\vert)^t P(\vert \Im z \vert^{-1})a_N$$

\end{prop}
\begin{proof}
We set
$$
\omega_{N,i}(z)=
\frac{1}{\theta_i}\left(1-\frac{z}{\mathbb E\left[R_N\left(\frac1z\right)\right]_{ii}}\right),\quad z\in
\mathbb C\setminus[0,+\infty).
$$
As established in Proposition \ref{uniformconvergence'}, $\lim_{N\to\infty}\mathbb E[zR_N(z)]_{ii}=
(1-\theta_{i}\omega_1(1/z))^{-1}$. It follows that $\omega_{N,i}$ converges to $\omega_1$ uniformly 
on compacts of $\mathbb C\setminus[0,+\infty).$ Clearly, $\omega_{N,i}$ is also defined 
on a neighbourhood of zero. Note that
$$
\lim_{y\to+\infty}\omega_{N,i}(-1/iy)=\frac{1}{\theta_i}\left(1-
\frac{1}{\displaystyle\lim_{y\to+\infty}\mathbb E\left[-iyR_N\left({-iy}\right)\right]_{ii}}\right)=
\frac{1}{\theta_i}\left(1-\frac11\right)=0,
$$
and 
\begin{eqnarray*}
\lim_{y\to+\infty}iy\omega_{N,i}\left(\frac{-1}{iy}\right)&=&-\frac{1}{\theta_i}\lim_{y\to+\infty}
\frac{\mathbb E\left[iyX_N(iy+X_N)^{-1}\right]_{ii}}{\mathbb E\left[iy(iy+X_N)^{-1}\right]_{ii}}\\
&=& 
-\frac{1}{\theta_i}
\mathbb E[X_N]_{ii}.
\end{eqnarray*}
In addition, since $\|X_N\|\leq\|A_N\|\|B_N\|$ which is uniformly bounded, the map
$z\mapsto\omega_{N,i}(-1/z)$ is analytic and real on the complement of an interval $[-m,0]$,
with $m=\sup_N\|A_N\|\|B_N\|$. Thus, the maps $z\mapsto\omega_{N,i}(-1/z)$ and 
$z\mapsto\omega_1(-1/z)$ are Nevanlinna maps \eqref{Nevanlinna}, and hence can pe represented as
$$
\omega_{N,i}\left(\frac{-1}{z}\right)=\int_{[0,m]}\frac{1}{t-z}\,d\Phi_{N,i}(t),\quad z\in\mathbb C^+,
$$
and
$$
\omega_1\left(\frac{-1}{z}\right)=\int_{[0,m]}\frac{1}{t-z}\,d\Phi(t),\quad z\in\mathbb C^+.
$$
Here $\Phi_{N,i},\Phi$ are positive measures on $[0,m]$ $\Phi_{N,i}([0,m])=
\frac{1}{\theta_i}\mathbb E\left[X_N-\chi_{\{0\}}(X_N)\right]_{ii}$, and
$\Phi([0,m])=\frac{\int_\mathbb R t\,d(\mu\boxtimes\nu)(t)}{\int_\mathbb R t\,d\mu(t)}=\int_\mathbb 
Rt\,d\nu(t)$. Thus, Lemma \ref{conv}
applies to $\rho_{N,i}=\Phi_{N,i}-\Phi$ to allow the estimate
$$
\left|\omega_{N,i}\left(\frac{-1}{z}\right)-\omega_1\left(\frac{-1}{z}\right)\right|<
\frac{|z|^2+C}{(\Im z)^2}v_{N,i},
$$
with $C_i=\Phi([0,m])+\sup_N\Phi_{N,i}([0,m])$. We have
\begin{eqnarray*}
\left|\mathbb E\left[R_N(z)\right]_{ii}-\frac{1}{z}\cdot\frac{1}{1-\theta_i\omega_1\left(\frac{1}{z}\right)}\right|&=&\left|\frac{1}{z}\cdot\frac{1}{1-\theta_i\omega_{N,i}\left(\frac{1}{z}\right)}-\frac{1}{z}\cdot\frac{1}{1-\theta_i\omega_1\left(\frac{1}{z}\right)}\right|\\
&=&\frac{\theta_i}{|z|}\frac{\left|\omega_{N,i}\left(\frac{1}{z}\right)-\omega_1\left(\frac{1}{z}
\right)\right|}{\left|(1-\theta_i\omega_1\left(\frac{1}{z}\right))(1-\theta_l\omega_{N,i}\left(\frac{1}{z}\right))\right|}\\
&<&\frac{1}{|z|}\frac{|z|^2+C_i}{(\Im z)^4}\frac{(|z|+m)^4}{\theta_i\Phi_{N,i}([0,m])
\Phi([0,m])}v_{N,i}.
\end{eqnarray*}
The proposition follows.
\end{proof}
To complete the argument of Step A, it suffices now to observe that the residue of the function
$1/(z(1-\theta_i\omega_1(z^{-1})))$ at $\rho$ is equal to 
$$
\delta_{\omega_1(1/\rho),1/\theta_i}\frac{\omega_1(1/\rho) \rho}{\omega_1'(1/\rho)},\quad
i=1,\ldots,p.
$$
\newline\noindent{\bf Step B:} We use the same perturbation argument as in Step B of the proof of 
Theorem \ref{Main+}. We reduce the problem to the case of a spike with multiplicity one, to which we 
apply Step A. The only change from the argument in Step B of Theorem \ref{Main+} comes from 
the form of the subordination functions. We use perturbations \eqref{perturbA} and \eqref{perturbB} and 
define $X_{N,\delta,\eta}=A_{N,\delta}^{1/2}U_N^*B_{N,\eta}U_NA_{N,\delta}^{1/2}$. The 
quantity $\|X_{N,\delta,\eta}-X_N\|$ tends to zero uniformly in $N$ as $\delta+\eta\to0$.
The details are omitted.
\end{proof}

\subsection{The unitary multiplicative model $X_N=A_NU_N^*B_NU_N$}
We use the notation from Subsection \ref{subsec:XperturbT}. 
The tools used are identical to the ones used in 
the analysis of the positive model $X_N=A_N^{1/2}U_N^*B_NU_NA_N^{1/2}$. However, the domains of 
definition of the analytic transforms involved are different. We indicate the relevant differences. Choose 
$\alpha,\beta\in\mathbb T$ such that $1/\alpha\in\mathrm{supp}(\mu)$ and $1/\beta\in\mathrm{supp}
(\nu)$. The reduction to the almost sure convergence of a $p\times p$ matrix is performed the
same way, and the same concentration
inequality holds (this time with Lipschitz constant $\frac{2}{(1-|z|)^2}$) in
Lemma \ref{concentration}. The counterparts of Propositions \ref{uniformconvergence'} and 
\ref{estimationinmean'} hold, but
in Proposition \ref{estimationinmean'} we must consider $z\in\mathbb 
C\setminus\mathbb T$. The resolvent $R_N$ is defined by $R_N(z)=\left(zI_N-A_NU_N^*B_NU_N\right)^{
-1}$. The 
function $\omega_{N,i}$ defined by 
$$
\omega _{N,i}(z)=\frac{1}{(A_N)_{ii}}\left(1-\frac{z}{\mathbb E\left[{R}_N\left(z^{-1}\right)\right]_{ii}}\right)
,\quad z\in\mathbb D,
$$
is easily
seen to map $\mathbb D$ into itself and fix the origin. Indeed, $|(A_N)_{ii}|=|(A'_N)_{ii}|=1$. 
In the unitary version of Lemma \ref{approximatesubordination'},
no supplementary condition on $A_N'$ is required, and for $\Omega_N$ defined as in the
proof of  Lemma \ref{approximatesubordination'}, the estimate becomes $\|\Omega_N(z)\|
<2/|z|$ if $|z|<1$. The estimates for the corresponding resolvents are provided by 
Lemma \ref{l}.

\subsubsection{Proofs of the main results for the unitary multiplicative model.}
\begin{proof}
[Proof of Theorem {\rm\ref{MainX+}} for $\mathbb T$, parts {\rm (1)} and {\rm(2)}---eigenvalue behaviour.] 
We must now apply Lemma \ref{alt-Benaych-Rao} with $\gamma=\mathbb T$. It will be applied
to $\gamma=\mathbb T$, the sequence $\{F_N(z)\}_{N}$ defined by
$$
F_N(z)=z(I_p-\Theta)P_N\left(zI_N-A'_NU_N^*B_N'U_N\right)^{-1}P_N^*+\Theta,\quad
z\in\mathbb C\setminus\mathbb T,
$$
and the limit $F$ provided by Proposition \ref{estimationinmean'}.
Observe that $F_N(z)$ is invertible for $z\not\in\mathbb T$. Indeed, it is easy to see that,
if $F_N(z)$ is not invertible, $z$ belongs to the spectrum of the unitary operator $A_NU_N^*B_N'U_N$. 
The convergence of $F_N$
to $F$ follows from the appropriate version of Proposition \ref{estimationinmean'}.
Clearly, $F(z)$ is diagonal and, again by the Julia-Carath\'eodory Theorem, this time applied to the disk, 
its diagonal entries have only simple zeros. The remainder of the argument
requires no further adjustments.
\end{proof}
\begin{proof}[Proof of Theorem {\rm\ref{MainX+}}, parts {\rm (3)} and {\rm(4)}---eigenspace behaviour.] 

The relevant chan\-ges for this part of the proof occur in
Proposition \ref{estimfondaX}, where $(1-|z|)^{-1}$ must be used instead of $|\Im z|^{-1}$ and an 
application of Lemma \ref{convT} in place of Lemma \ref{conv}. Also, the perturbations \eqref{perturbA}
and \eqref{perturbA} are applied to the arguments of $\theta_i$ and $\tau_j$, respectively.
\end{proof}

\begin{rem}
It is easy to see that our results hold equally well when $A_N$ is random, independent of $U_N$, and has 
spikes $\theta_1(N),\dots,\theta_p(N)$ with the property that $\lim_{N\to\infty}\theta_i(N)=\theta_i$,
$1\leq i\leq p$, almost surely. Similarly, $B_N$ can be taken to be random, independent of $A_N$ and 
$U_N$, and with spikes $\tau_1(N),\dots,\tau_q(N)$ that converge almost surely to $\tau_1,\dots,
\tau_q$. The proofs use  the general form of Propositions \ref{estimationinmean} and 
\ref{uniformconvergence'}, respectively.
\end{rem}

\begin{rem}
The above remark allows us to treat sums or products of more than two spiked matrices.
More precisely, let $k\ge3$ be an integer, let $A_N^{(1)},\dots,A_N^{(k)}\in M_N(\mathbb C)$ be 
deterministic Hermitian matrices and let $U_N^{(1)},\dots,U_N^{(k)}\in {\rm U}(N)$ be independent 
Haar-distributed random matrices. Suppose
that the eigenvalue distribution of $A_N^{(j)}$ tends weakly to $\mu_j$ and $A_N^{(j)}$ has spikes
subject to the hypotheses
of Subsection \ref{subsec:+perturb}. Then 
$X_N^{(k)}=U_N^{(1){*}}A_N^{(1)}U_N^{(1)}+\cdots+U_N^{(k-1){*}}A_N^{(k-1)}U_N^{(k-1)}+A_N^{(k)}$ has 
asymptotic eigenvalue distribution equal to $\mu_1\boxplus\cdots\boxplus\mu_k$,
and the outliers in the spectrum of $X_N$ are described by an appropriate reformulation of 
Theorem \ref{Main+}. The result can be proved by induction on $k$ if we observe that $X_N^{(k+1)}$
has the same asymptotic eigenvalue distribution as $A_N^{(k+1)}+U_N^{*}B_NU_N$, 
where $B_N=X_N^{(k)}$ and $U_N$ is a Haar-distributed unitary independent from 
$U_N^{(1)},\dots,U_N^{(k-1)}\in {\rm U}(N)$. A similar remark applies to Theorem \ref{MainX+} in the
case of the circle. For the case of the multiplicative model on $[0,+\infty)$, the corresponding
generalization applies to a model of the form
$A_k^{1/2}U_{k-1}^{*}A_{k-1}^{1/2}\cdots U_2^{*}A_2^{1/2}U_1^{*}A_1
U_1A_2^{1/2}U_2\cdots A_{k-1}^{1/2}U_{k-1}A_k^{1/2}$.
\end{rem}

\begin{rem}\label{atoms}
The analogue of Theorem \ref{Main+} when $\mu=\delta_0$ was proved in \cite{BGRao09}
under the additional assumption that all eigenvalues of $A$ except for the spikes are equal to zero.
Our arguments provide a proof of this result without this additional assumption. Similar observations
apply to Theorem \ref{MainX+} when either $\mu$ or $\nu$ is a point mass. The only case in
which one needs to be more careful is that of Theorem \ref{MainX+} for the positive half-line when
$\mu$ or $\nu$ is equal to $\delta_0$. Suppose, for instance, that $\nu=\delta_0\neq\mu$. 
The eigenvalues of $X_N=A_N^{1/2}U_N^{*}B_NU_NA_N^{1/2}$ are uniformly approximated arbitrarily well
by the eigenvalues of 
$$
X_{N,\varepsilon}  =  A_N^{1/2}U_N^{*}(B_N+\varepsilon I_N)U_NA_N^{1/2}
 = X_N+\varepsilon A_N,
$$
and our methods do apply to the perturbed model $X_{N,\varepsilon}$, whose asymptotic eigenvalue 
distribution is $\mu\boxtimes\delta_\varepsilon$. The spikes are calculated explicitly by noting that
$\eta_{\mu\boxtimes\delta_\varepsilon}(z)=\eta_\mu(\varepsilon z)$, so $\omega_1(z)=
\varepsilon z$, $\omega_2(z)=\eta_\mu(\varepsilon z)/\varepsilon$. Thus, $v_1(z)=\varepsilon/z$
and $v_2(z)=\eta_\mu(\varepsilon/z)/\varepsilon$. The spikes of $X_{N,\varepsilon}$ are the solutions
of the equations $v_1(\rho)=1/\theta_i,$ $i=1,\dots,p$ and $v_2(\rho)=1/(\tau_j+\varepsilon)$, $j=1,
\dots,q$. The first set of equations yields the outliers $\varepsilon\theta_i$, $i=1,\dots,p$, while the 
second set of equations can be rewritten as
$$
\rho=(\tau_j+\varepsilon)\left[\frac\rho\varepsilon\eta_\mu\left(\frac\varepsilon\rho\right)\right],
\quad j=1,\dots,q.
$$
As $\varepsilon\to0$, we conclude that the spikes of $X_N$ are the numbers $\tau_j\gamma$, 
$j=1,\dots,q$, where $\gamma=\eta_\mu'(0)=\int_0^\infty t\,d\mu(t)$ is the first moment of $\mu$. 
If $\mu=\nu=\delta_0$, a similar argument shows that $X_N$ has no outliers at all, that is, $\lim_{
N\to\infty}\|X_N\|=0$ almost surely.
\end{rem}

\begin{example}\label{ex:5.14}
The following numerical simulation, due to Charles Bordenave, illustrates the appearance of two outliers 
arising from a single spike.  We take $N=1000$ and $X_N=A_N+
U_N^*B_NU_N$, where $A_N=2p-I_{N}$, with $p$ an orthogonal projection of rank $N/2=500$. The 
matrix $B_N$ is given by the formula
$$
B_N=\left[\begin{array}{cc}
\frac{W}{2(N-1)} & 0_{(N-1)\times 1}\\
0_{1\times (N-1)} & 10
\end{array}\right],
$$
with $W$ being sampled from a standard $999\times999$ GUE.

\noindent\begin{center}
     \includegraphics[width=12cm,height=7cm]{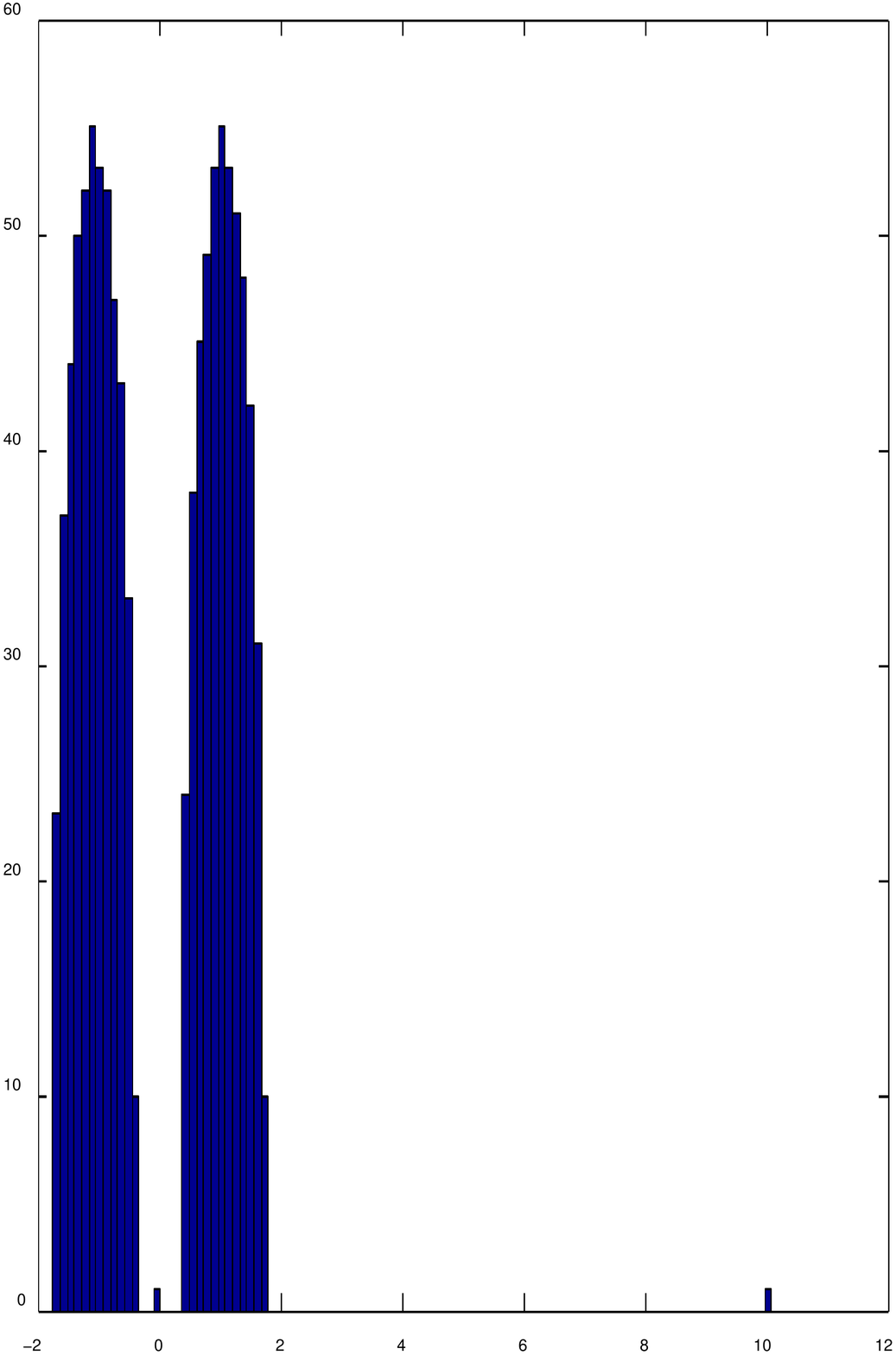}     
\par\end{center}
\end{example}


\begin{thebibliography}{10}

\bibitem{akhieser} Naum Ilich Akhieser, {\em The classical moment problem
and some related questions in analysis.} Translated by N. Kemmer.
Hafner Publishing Co., New York, 1965.


\bibitem{AGZ10} G. W. Anderson, A. Guionnet, and O. Zeitouni, \emph{An
introduction to random matrices}, Cambridge University Press, Cambridge,
2010.


\bibitem{BaiYao08b} Z. D. Bai and J. Yao, On sample eigenvalues in
a generalized spiked population model, \emph{J. Multivariate Anal}.
\textbf{106} (2012), 167--177. 



\bibitem{BBP05} J. Baik, G. Ben Arous, and S. P{\'e}ch{\'e}, Phase
transition of the largest eigenvalue for nonnull complex sample covariance
matrices, {\em Ann. Probab.} \textbf{33} (2005), 1643--1697.

\bibitem{BaikSil06} J. Baik and J. W. Silverstein, Eigenvalues of
large sample covariance matrices of spiked population models, \emph{J.
Multivariate Anal}. \textbf{97}(2006), 1382--1408.


\bibitem{Belinschi08} S. Belinschi, The Lebesgue decomposition of
the free additive convolution of two probability distributions, \emph{Probab.
Theory Related Fields} \textbf{142} (2008), 125--150.

\bibitem{AIHP}\bysame, A note on regularity for free convolutions, \emph{Ann. Inst. H. 
Poincar\'e Probab. Stat.} {\bf 42} (3), 635--648 (2006)

\bibitem{BelBer07} S. T. Belinschi and H. Bercovici, A new approach
to subordination results in free probability, \emph{J. Anal. Math.}
\textbf{101} (2007), 357--365.

\bibitem{B-B-imrn}\bysame, Partially defined semigroups relative
to multiplicative free convolution, \emph{Int. Math. Res. Not.} \textbf{2005},
65--101.

\bibitem{BPV12} S. T. Belinschi, M. Popa, and V. Vinnikov. Infinite
divisibility and a non-commutative Boolean-to-free Bercovici-Pata
bijection, {\em J. Funct. Anal.} \textbf{262} (2012), 94--123.


\bibitem{BercoviciVoiculescu} H. Bercovici and D. Voiculescu, Free
convolution of measures with unbounded support, \emph{Indiana Univ.
Math. J}. \textbf{42} (1993), 733--773.

\bibitem{BGRao09} F. Benaych-Georges and R. R. Nadakuditi, The eigenvalues
and eigenvectors of finite, low rank perturbations of large random
matrices, \emph{Advances in Mathematics} \textbf{227} (2011), 494--521.

\bibitem{Bhatia} Rajendra Bhatia, \emph{Matrix analysis}, Springer Verlag, New York, 1997.


\bibitem{Biane98}P. Biane, Processes with free increments, {\em Math.
Z}. \textbf{227} (1998), 143--174.

\bibitem{birman}M. S. Birman and M. Z. Solomyak, Double Stieltjes operator integrals. III, \emph{Prob. Math. Phys, Leningrad Univ.} \textbf{6} (1973), 27--53.

\bibitem{Capitaine11} M. Capitaine, Additive/multiplicative free
subordination property and limiting eigenvectors of spiked additive
deformations of Wigner matrices and spiked sample covariance matrices,
\emph{Journal of Theoretical Probability}, Volume 26 (3) (2013), 595--648.



\bibitem{Capitaine14} M. Capitaine,
Exact separation phenomenon for the eigenvalues of large Information-Plus-Noise type matrices. 
Application to spiked models,
\emph{Indiana Univ. Math. J.}  {\bf 63} (2014),  1875--1910.

\bibitem{CD07}
M.~Capitaine and C.~Donati-Martin.
\newblock Strong asymptotic freeness for {W}igner and {W}ishart matrices.
\newblock {\em Indiana Univ. Math. J.}, {\bf 56} (2):767--803, 2007.


\bibitem{CDF09} M. Capitaine, C. Donati-Martin, and D. F{\'e}ral,
The largest eigenvalues of finite rank deformation of large Wigner
matrices: convergence and nonuniversality of the fluctuations, \emph{Ann.
Probab}.\textbf{ 37} (2009), 1--47.

\bibitem{CDFF10} M. Capitaine, C. Donati-Martin, D. F{\'e}ral and
M. F\'evrier, Free convolution with a semi-circular distribution
and eigenvalues of spiked deformations of Wigner matrices, \emph{Electronic
Journal of Probability} \textbf{16} (2011), 1750--1792.


\bibitem{ColMal11} B. Collins and C. Male, The strong asymptotic
freeness of Haar and deterministic matrices, 
{\em Ann. Sci. \'Ec. Norm. Sup\'er. (4)}  {\bf 47} (2014), no. 1, 147--163.

\bibitem{FePe}D. F{\'e}ral and S. P{\'e}ch{\'e}, The largest eigenvalue
of rank one deformation of large Wigner matrices, \emph{Comm. Math.
Phys}. \textbf{272} (2007), 185--228.


\bibitem{Fulton98}W. Fulton, Eigenvalues of sums of Hermitian matrices
(after A. Klyachko), Ast\'erisque \textbf{252} (1997), 255--269,
S{\'e}minaire Bourbaki 1997/98.

\bibitem{FurKom81} Z. F\H{u}redi and J. Koml{\'o}s, The eigenvalues
of random symmetric matrices, \emph{Combinatorica} \textbf{1} (1981),
233--241.

\bibitem{GarnettBook} John B. Garnett, \emph{Bounded analytic functions,}
Academic Press , New York, 1981.

\bibitem{gesz-tse}F. Gesztesy and E. Tsekanovskii, On matrix-valued
Herglotz functions, \emph{Math. Nachr}. \textbf{218} (2000), 61--138.


\bibitem{numran}K. E. Gustafson and K. M. D. Rao, \emph{Numerical
range}, Springer-Verlag, New York, 1997.


\bibitem{HT}
U. Haagerup and S. Thorbj\o{}rnsen, 
{A new application of random matrices: ${\rm Ext}(C\sp *\sb {\rm red}(F\sb 2))$ is not a group.}  
\emph{Ann. of Math.} (2)  {\bf 162}  (2005),  no. 2, 711--775.

\bibitem{John} I. Johnstone, On the distribution of the largest eigenvalue
in principal components analysis, \emph{Ann. Stat}. \textbf{29} (2001),
295--327.

\bibitem{Kargin11} V. Kargin, Subordination of the resolvent for
a sum of random matrices,  Preprint (2011) arXiv:1205.0993[math.PR].

\bibitem{LV}P. Loubaton and P. Vallet, Almost sure localization of
the eigenvalues in a Gaussian information-plus-noise model. Application
to the spiked models,  {\em Electron. J. Probab.} {\bf 16} (2011), no. 70, 1934--1959.


\bibitem{Peche06} S. P{\'e}ch{\'e}, The largest eigenvalue of small
rank perturbations of Hermitian random matrices, \emph{Probab. Theory
Related Fields} \textbf{134} (2006), 127--173.

\bibitem{PRS} A. Pizzo, D. Renfrew, and A. Soshnikov, On finite rank
deformations of Wigner matrices, \emph{Annales de l'Institut Henri
Poincar\'e (B) Probabilit\'es et Statistiques}, {\bf 49}, (2013), no. 1, 64–94..

\bibitem{RaoSil09} N. R. Rao and J. W. Silverstein, Fundamental limit
of sample generalized eigenvalue based detection of signals in noise
using relatively few signal-bearing and noise-only samples, \emph{IEEE
Journal of Selected Topics in Signal Processing} \textbf{4} (2010),
468--480.

\bibitem{Remmert84} R. Remmert, \emph{Funktionentheorie. I}, Springer-Verlag,
Berlin, 1984.

\bibitem{Speicher93a} R. Speicher, Free convolution and the random
sum of matrices, \emph{Publ. Res. Inst. Math. Sci.} \textbf{29} (1993),
731--744.

\bibitem{Voiculescu86} D. Voiculescu, Addition of certain noncommuting
random variables, {\em J. Funct. Anal.} \textbf{66} (1986), 323--346.

\bibitem{V2} \bysame, 
{Multiplication of certain noncommuting random variables}.
{\em J. Operator Theory} {\bf 18}(1987), 223--235.

\bibitem{Voiculescu91} \bysame, Limit laws for random matrices and
free products,\emph{ Invent. Math.} \textbf{104} (1991), 201--220.

\bibitem{voic-fish1} \bysame, The analogues of entropy and of Fisher's
information measure in free probability theory. I, \emph{Comm. Math.
Phys}. \textbf{155} (1993), 71--92.

\bibitem{Voiculescu00} \bysame, The coalgebra of the free difference
quotient and free probability, \emph{Internat. Math. Res. Notices}
\textbf{2000 }79--106.

\bibitem{VDN92} D. V. Voiculescu, K. J. Dykema, and A. Nica, \emph{Free
random variables}, American Mathematical Society, Providence, RI,
1992.


\end{thebibliography}
\end{document}